\documentclass{article}
\usepackage[utf8]{inputenc}
\usepackage[algonl,boxed,norelsize,lined]{algorithm2e}
\usepackage[svgnames]{xcolor}
\usepackage[bookmarksnumbered=true]{hyperref} 
\hypersetup{
     colorlinks = true,
     linkcolor = red,
     anchorcolor = blue,
     citecolor = blue,
     filecolor = blue,
     urlcolor = blue
     }
\definecolor{maroon(html/css)}{rgb}{0.5, 0.0, 0.0}\definecolor{cobalt}{rgb}{0.0, 0.28, 0.67}
\usepackage{amscd,amsfonts,bbm,mathrsfs}
\usepackage{amsmath}
\usepackage{amsthm}
\usepackage{amssymb}

\usepackage{fullpage}
\usepackage{tocloft}

\usepackage{mathtools}
\usepackage{arydshln}
\usepackage{fullpage}
\usepackage{comment}
\usepackage[all]{xy}
\usepackage{graphicx,import}
\usepackage{color,mathtools}
\usepackage[english]{babel}
\usepackage[T1]{fontenc}
\usepackage{hyperref}
\usepackage{multirow}

\usepackage{enumitem}
\usepackage{theoremref}

\usepackage{booktabs}

\usepackage[a4paper,top=2.8cm,bottom=2cm,left=2.5cm,right=2.5cm,marginparwidth=1.75cm]{geometry}

\usepackage{amsmath}
\usepackage{amssymb}
\usepackage{amsthm}
\usepackage{verbatim,bbold}

\usepackage{tikz,tikz-cd}
\usetikzlibrary{decorations.pathreplacing}
\usetikzlibrary{calc,cd}

\let\emptyset\varnothing
\renewcommand{\epsilon}{\varepsilon}

\newcommand{\R}{{\mathbb R}}
\newcommand{\C}{{\mathbb C}}

\newcommand{\dd}{\operatorname{d}}

\def \C{\mathbb{C}}

\def \R{\mathbb{R}}

\def \P{\mathbb{P}}

\def \A{\mathcal{A}}

\def \/{\big/}

\def \GL{\textup{GL}}

\def \det{{\rm{det}}}

\def \dim{\textup{dim}}

\def \ker{\textup{ker}}
\def \Gr{\textup{Gr}}

\def \Gr{{\rm{Gr}}}
\def \Mat{{\rm{Mat}}}

\usepackage{thmtools}
\declaretheorem[numberwithin=section]{theorem}

\declaretheorem[sibling=theorem]{definition}
\declaretheorem[sibling=theorem, style=definition]{property}

\declaretheorem[sibling=theorem, style=definition]{example}
\declaretheorem[sibling=theorem, style=definition]{remark}

\definecolor{caribbeangreen}{rgb}{0.0, 0.8, 0.6}
\definecolor{capri}{rgb}{0.0, 0.75, 1.0}

\theoremstyle{plain}
\newtheorem{Theorem}{Theorem}[section]
\newtheorem{Lemma}[Theorem]{Lemma}

\newtheorem{Question}[Theorem]{Question}
\newtheorem{Proposition}[Theorem]{Proposition}
\newtheorem{Conjecture}[Theorem]{Conjecture}
\newtheorem{Corollary}[Theorem]{Corollary}
\newtheorem*{NoNumberTheorem}{Theorem}

\theoremstyle{definition}

\newtheorem*{Theorem*}{Theorem}
\newtheorem*{Corollary*}{Corollary}

\renewcommand{\paragraph}[1]{\smallskip\noindent\textbf{#1}}

\begin{document}

\title{
Triangulations and Canonical Forms of Amplituhedra:\\ a fiber-based approach beyond polytopes 
}
\author{Fatemeh Mohammadi, Leonid Monin, and Matteo Parisi}
\maketitle

\noindent{\bf Abstract.} 
Any totally positive $(k+m)\times n$ matrix induces a map $\pi_+$ from the positive Grassmannian ${\rm Gr}_+(k,n)$ to the Grassmannian ${\rm Gr}(k,k+m)$, whose image is the \emph{amplituhedron} $\mathcal{A}_{n,k,m}$ and is endowed with a top-degree form called the \emph{canonical form} ${\bf\Omega}(\mathcal{A}_{n,k,m})$.
This construction was introduced by 
Arkani-Hamed and Trnka in \cite{arkani2014amplituhedron}, where they showed that ${\bf\Omega}(\mathcal{A}_{n,k,4})$ encodes scattering amplitudes in $\mathcal{N}=4$ super Yang-Mills theory. One way to compute ${\bf\Omega}(\mathcal{A}_{n,k,m})$ is to subdivide 
$\mathcal{A}_{n,k,m}$ into so-called generalized triangles and sum over their associated canonical forms.
Hence, the physical computation of scattering amplitudes is reduced to finding the triangulations of $\mathcal{A}_{n,k,4}$.
However, while triangulations of polytopes are fully captured by their secondary and fiber polytopes \cite{gelfand2008discriminants, billera1992fiber}, the study of triangulations of objects beyond
polytopes is still underdeveloped. 

In this work, we initiate the geometric study of subdivisions of $\mathcal{A}_{n,k,m}$ 
in order to establish the notion of {\em secondary amplituhedron}.
For this purpose, we first extend the projection $\pi_+$ to a rational map 
$\pi:\Gr(k,n)\dashrightarrow\Gr(k,k+m)$ and provide a concrete birational parametrization of the fibers of $\pi$. We then use this to explicitly describe a rational top-degree form $\omega_{n,k,m}$ (with simple poles) on the fibers 
and compute ${\bf\Omega}(\mathcal{A}_{n,k,m})$
as a summation of certain residues of $\omega_{n,k,m}$. 
As main application of our approach, we develop a well-structured notion of secondary amplituhedra for 
conjugate to polytopes, i.e.~when $n-k-1=m$ (even). We show that, in this case, each fiber of $\pi$ is parametrized by a projective space and its volume form $\omega_{n,k,m}$ has only poles on a hyperplane arrangement. Using such linear structures,
for amplituhedra which are cyclic polytopes or conjugate to polytopes, we 
show that the 
\emph{Jeffrey-Kirwan residue} computes ${\bf\Omega}(\mathcal{A}_{n,k,m})$ from the fiber volume form $\omega_{n,k,m}$. In particular, we give conceptual proofs of the statements of \cite{ferro2018amplituhedron}.
Finally, we propose a more general framework of \emph{fiber positive geometries} and analyze new families of examples such as fiber polytopes and Grassmann polytopes.

{
\hypersetup{linkcolor=black}
\setcounter{tocdepth}{1}
\setlength\cftbeforesecskip{1.7pt}
{\tableofcontents}}

\section{Introduction}
The combinatorics of {\em subdivisions} of polytopes is very rich and has been extensively studied in polyhedral geometry (see, e.g., \cite{de2010triangulations} and references therein). More importantly, by the seminal works of Gelfand, Kapranov and Zelevinsky \cite{zbMATH00005177,gelfand2008discriminants}, given any polytope $P$ there exists a so-called secondary polytope whose face lattice is isomorphic to the poset of regular subdivisions of $P$.
In a subsequent work \cite{billera1992fiber}, Billera and Sturmfels introduced the notion of \emph{fiber polytopes} which contain secondary polytopes as examples. More precisely, given two polytopes $Q$ and $P$ with a projection $\pi:Q \to P$ whose fibers are all polytopes, the fiber polytope $\Sigma(Q,P)$ is defined as the {\em Minkoswki integral} of the fibers of $\pi$ over $P$.
If one considers a polytope $P$ with $n$ vertices as the image of the standard projection from the simplex $\Delta_{n-1}$, then the fiber polytope $\Sigma(\Delta_{n-1},P)$ coincides with the secondary polytope of $P$. 
Therefore, the fibers of such projections relate to the subdivisions of $P$ in an elegant way. 
Hence, this approach provides a representation of the secondary polytope of $P$ as Minkowski integral over the fibers of $\pi$, or as Minkowski sum of these fibers at some discrete set of points in $P$. A primary family of polytopes whose subdivisions are well-understood is {\em cyclic polytopes} (see, e.g.~\cite{ edelman1996higher, rambau1997triangulations, athanasiadis2000fiber, ziegler}). The cyclic polytope $C(n,d)$ is the convex hull of any $n$ distinct points on the moment curve $\{(t,t^2,\ldots,t^d): t\in \mathbb{R}\}$ in $\mathbb{R}^d$. 

\smallskip

The study of subdivisions of objects beyond
polytopes is quite underdeveloped. One particular family of interest is amplituhedra. The {\em amplituhedron} is a geometric object defined by Arkani-Hamed and Trnka \cite{arkani2014amplituhedron} whose subdivisions have profound use-cases in physics, especially in computing scattering amplitudes of particles. Cyclic polytopes are 
examples of amplituhedra. However, a general amplituhedron is not necessarily a polytope.
Another special case is isomorphic to the {\em positive Grassmannian} $\Gr_+(k,n)$,
the subset of the real Grassmannian $\Gr(k,\mathbb{R}^n)$
where all Pl\"ucker coordinates are non-negative. $\Gr_+(k,n)$ is not a linear object, i.e.~it is not carved out by linear polynomial inequalities as polytopes. Based on the seminal work on {\em positivity} by Lusztig \cite{lusztig}, as well as Fomin and Zelevinsky \cite{fomin1999double}, Postnikov \cite{postnikov2006total} studied the cell decompositions of positive Grassmannians and showed that even though they are not linear objects, they have remarkable geometric structures. In particular, the restriction of the matroid stratification of $\Gr(k,\mathbb{R}^n)$ given in \cite{gelfand1987combinatorial}, to the positive Grassmannian provides a decomposition of $\Gr_+(k,n)$ into so-called \emph{positroids} cells \cite{postnikov2006total,PSW,rietsch}. 
Since Postnikov's work \cite{postnikov2006total}, positive Grassmannians have been extensively studied in mathematics and have also appeared in various other fields including in the amplituhedron theory (see, e.g.~\cite{Arkani-Hamed:2016byb} and references therein). 

\smallskip

To define the amplituhedron $\A_{n,k,m}$ more precisely, consider a map $\pi_+:\Gr_+(k, n)\rightarrow\Gr(k,k+m)$ induced by any totally positive $(k+m)\times n$ matrix $Z$. Then the image of $\pi_+$ is $\A_{n,k,m}$ which is a semialgebraic subset of $\Gr(k,k+m)$ of full dimension. 
The amplituhedron has very complicated geometric and combinatorial structures. 
Despite the extensive research on this topic, 
the structure of $\A_{n,k,m}$ and its subdivisions are only known for special cases, e.g.~when it is isomorphic to the positive Grassmannian,
or cyclic polytopes \cite{sturmfels1988totally}, or the complex of bounded regions of a cyclic hyperplane arrangement \cite{karp2019amplituhedron}, or the $m=2$ case which is related to the positive tropical Grassmannian \cite{lukowski2020positive}. 

\smallskip

Cyclic polytopes, positive Grassmannians, and amplituhedra are all part of the more general paradigm of \emph{positive geometries} \cite{positivegeom}.
In recent years, positive geometries are frequently arising as the underlying mathematical structures for quantum mechanical observables of many theories in particle physics and cosmology. Even though positive geometries have a priori no reference to any physics notions, the physical principles and properties of physical observables can be seen as emergent from the mathematical properties of these objects. Therefore, they play a fundamental role in understanding the physics they encode.
More precisely, a positive geometry is a pair $(X_\C,X_+)$ of a complex algebraic variety $X_\C$ defined over $\R$ and a semialgebraic subset $X_+\subset X_\R$ such that there exists a unique meromorphic form $\mathbf{\Omega}(X)$ on $X_\C$, called \emph{canonical form}, which has simple poles or {\em logarithmic singularities} along the (complexified) boundaries of $X_+$ and it is regular everywhere else. 
A canonical form of a positive geometry explicitly encodes its associated physical observables. 
Given a positive geometry $(X_\C,X_+)$ endowed with a canonical form $\mathbf{\Omega}(X)$, a natural way to compute $\mathbf{\Omega}(X)$ is by {\em triangulating} the geometry and summing over the canonical forms $\mathbf{\Omega}(X_a)$ of the cells $X_a$ in a triangulation $\{X_a\}$ of $X_+$. Since the forms $\mathbf{\Omega}(X_a)$ are usually straightforward to compute, the problem is reduced to finding the triangulations of positive geometries. 
Moreover, different triangulations correspond to different representations of $\mathbf{\Omega}(X)$ and lead to different expressions of the associated physical observable. Hence, it is desirable to develop a theory of {\em secondary positive geometry} to encode the 
interrelations of different representations of physical observables. 

\smallskip

In particular, understanding the subdivisions of the amplituhedron is very important. Since $\mathcal{A}_{n,k,m}$ is the image of $\Gr_+(k,n)$ which has a decomposition into positroid cells \cite{postnikov2006total}, the images of positroid cells are good candidates for decomposing $\mathcal{A}_{n,k,m}$.
Hence, the main question in this context is to characterize collections of positroid cells whose full-dimensional images give a \emph{dissection} of the amplituhedron, i.e.~they are pairwise disjoint, and together they cover a dense subset of the amplituhedron. In case the map $\pi_+$ is injective on such cells, the dissection is called \emph{triangulation} and its elements \emph{generalized triangles}. Importantly, any such collection of positroid cells
gives rise to an expression for the canonical form of the amplituhedron. 

\smallskip

The combinatorics and geometry of positroidal triangulations (and dissections) of the amplituhedron $\mathcal{A}_{n,k,m}$ is incredibly rich \cite{arkani2014amplituhedron,ferro2016towards,positivegeom,galashin2018parity,ferro2018amplituhedron} and still unexplored in its full generality. 
For $m=1$ a particular type of triangulations of the amplituhedron is provided in \cite{karp2019amplituhedron}. More recently, in \cite{bao2019m} the authors constructed many `BCFW-like' triangulations for the $m=2$ amplituhedron. It is shown in \cite{lukowski2020positive} that these triangulations (and, more in general, `BCFW-like' dissections) are in bijection with the triangulations (dissections) of the \emph{hypersimplex}. Moreover, it is conjectured that this bijection extends to any (positroidal) triangulation and dissection. In particular, there are some nice dissections coming from regular (positroidal) subdivisions of the hypersimplex, which are fully governed by the \emph{positive tropical Grassmannian} \cite{troppos}. 
We recall that the $k=1$ amplituhedron $\mathcal{A}_{n,1,m}$ is a  cyclic polytope in $\mathbb{P}^m$, hence its subdivisions are well-understood in terms of their secondary polytopes.

\medskip\noindent{\bf Our contributions.} 
One of the main questions in the amplituhedron theory is to explicitly describe the canonical form of the amplituhedron ${\bf\Omega}(\A_{n,k,m})$.
To compute ${\bf\Omega}(\A_{n,k,m})$ we study the fibers of the map $\pi_+: \Gr_+(k,n) \rightarrow \Gr(k,k+m)$ whose image is the amplituhedron $\A_{n,k,m}$. We first extend the projection $\pi_+$ to a rational map 
$\pi:\Gr(k,n)\dashrightarrow\Gr(k,k+m)$ whose fiber on each point $Y\in\A_{n,k,m}$ is related to the positive fiber, namely $\pi_+^{-1}(Y)=\pi^{-1}(Y) \cap \Gr_+(k,n)$. We then use this property to provide a compact parametrization of fibers as positive regions inside the Grassmannian $\Gr(k,n-m)$ and their boundaries; See Propositions~\ref{prop:fiber} and~\ref{prop:boundaries}.
Moreover, for each point $Y\in \mathcal{A}_{n,k,m}$ we show that the fiber $\pi^{-1}(Y)$ can be endowed with a rational top-degree form $\omega_{n,k,m}(Y)$ with simple poles along a collection of divisors. 
We call $\omega_{n,k,m}$ the \emph{fiber volume form} that can be explicitly computed using the aforementioned parametrization. In particular, this leads to an explicit description of $\omega_{n,k,m}$ as follows:

\begin{NoNumberTheorem}
Let $Y \in \mathcal{A}_{n,k,m}$ and $\omega_{n,k,m}(Y)$ be the volume form on the fiber $\pi^{-1}(Y)$. Then:
\begin{itemize}
    \item There exists an $(n-m)\times n$ matrix $A$ such that the fiber $\pi^{-1}(Y)$ is birationally parametrized via
\begin{equation} 
\Gr(k,n-m) \dashrightarrow \pi^{-1}(Y), \quad \lambda \mapsto \lambda \cdot A.
\tag{Proposition~\ref{prop:fiber}}
\end{equation}
\item Under the parametrization above, the volume form $\omega_{n,k,m}$ on the fiber $\pi^{-1}(Y)$ 
is
 \begin{equation}
  \omega_{n,k,m}(Y)= p_{J}(\lambda)^{-k} \frac{\mu_{\Gr(k,n-m)}(\lambda)}{\prod_{i=1}^n \det \left(\lambda|(A^{I_i})^\perp \right)},
  \tag{Theorem~\ref{prop:omega}}
\end{equation}
where $J$ is a $k$-subset of the rows of the matrix $A$ and $\mu_{\Gr(k,n-m)}$ is the standard measure on $\Gr(k,n-m)$. Here $(A^{I_i})^\perp$ spans the $(n-m-k)$-dim orthogonal complement of the $k$-dim span of the columns of $A$ labeled by $\{i,\ldots,i+k-1 \}$.
\item 
Every generalized triangle
$S_M$ corresponds to a pole $q_M$ of $\omega_{n,k,m}(Y)$. Moreover, given a triangulation $\mathcal{C}=\lbrace S_M \rbrace$ of $\mathcal{A}_{n,k,m}$, the canonical form ${\bf \Omega}(\mathcal{A}_{n,k,m})(Y)$ can be obtained as
\begin{equation}
    {\bf \Omega}(\mathcal{A}_{n,k,m})(Y) = \mu_{\Gr(k,k+m)}\cdot \sum_{S_M\in \mathcal{C}} {\rm Res}_{q_M} \omega_{n,k,m}(Y).
\tag{Theorem~\ref{thm:resfiber}}
\end{equation}
\end{itemize}
\end{NoNumberTheorem}

As mentioned above, a natural way to compute the canonical form of the amplituhedron is by finding its subdivisions. In particular, in \S\ref{sec:linear} we focus on the case of conjugate to polytopes, i.e.~$\mathcal{A}:=\mathcal{A}_{n,n-m-1,m}$. We show that, in this case the volume form $\omega_{n,k,m}(Y)$ has only poles along a set of hyperplanes $\mathcal{H}_1(Y),\ldots, \mathcal{H}_n(Y)$. We use this, together with the fiber-based approach developed in \S\ref{sec:fibers} to define a well-structured notion of secondary amplituhedra. This leads to a classification of positroidal dissections of $\mathcal{A}$. More precisely, let $W_i(Y)$ be a normal ray of the hyperplane $\mathcal{H}_i(Y)$ for each $i$ and let $\mathcal{F}_{\mathcal{A}}(Y)$ be the chamber fan of the rays $W_1(Y),\ldots,W_n(Y)$; See Definition~\ref{def:chamberfan}. We show that under certain conditions, $\mathcal{F}_\mathcal{A}(Y)$ is the secondary fan of $\mathcal{A}$. 
In particular, in \S\ref{subsec:triangconjpol} we show that:

\begin{NoNumberTheorem}[Proposition~\ref{prop_faneq} and Theorem~\ref{triangamp}]
The fan $\mathcal{F}_\mathcal{A}(Y)$ is combinatorially equivalent to the secondary fan of a cyclic polytope. 
Moreover, the cones of $\mathcal{F}_\mathcal{A}(Y)$ are in bijection with the positroidal dissections of $\mathcal{A}$. The inclusion between cones corresponds to coarsening between their associated dissections. In particular, maximal cones of $\mathcal{F}_\mathcal{A}(Y)$ correspond to positroidal triangulations of $\mathcal{A}$. 

\end{NoNumberTheorem}

In \S\ref{sec:JK}, we focus on the computation of the canonical form of the amplituhedron in terms of the fiber volume form. More precisely, we study the following algebraic-analytic question: 

\begin{Question}\label{ques:res}
What residue procedure on $\omega_{n,k,m}$ leads to the computation of $\mathbf{\Omega}\left(\mathcal{A}_{n,k,m} \right)$?
\end{Question}
We solve Question~\ref{ques:res} for the case of  
cyclic polytopes in $\mathbb{P}^m$ and 
conjugate to polytopes. Following \cite{ferro2018amplituhedron}, we connect the computation of $\mathbf{\Omega}\left(\mathcal{A}_{n,k,m} \right)$ to the Jeffrey-Kirwan residue $\mathrm{JK}_{\xi}(\cdot)$. We show that this approach provides the right framework to study the linear fibers, as $\omega_{n,k,m}$ has only poles on a collection of affine hyperplanes. 
More precisely, we first extend the classical notion of the Jeffrey-Kirwan residue in Definition~\ref{def:JKRes} for the fiber volume forms. Then we show that the collection of poles of $\omega_{n,k,m}$ picked up by the Jeffrey-Kirwan residues coincides with the Gale duality procedure of identifying cells in a triangulation. In particular, we give conceptual proofs of the main statements in \cite{ferro2018amplituhedron}.

\begin{NoNumberTheorem}[Theorem~\ref{th:jkfibers}]
Let $\mathcal{A}_{n,k,m}$ be either a cyclic polytope or 
conjugate to a polytope for an even $m$.  
Let $\xi$ be a generic point in the secondary fan of $\mathcal{A}_{n,k,m}$. Then the canonical form $\mathbf{\Omega}(\mathcal{A}_{n,k,m})$ can be obtained from the Jeffrey-Kirwan residue of the fiber volume form $\omega_{n,k,m}$. More precisely,
\begin{equation*}
\mathbf{\Omega}(\mathcal{A}_{n,k,m}) = \mathrm{JK}_{\xi}( \omega_{n,k,m}) \cdot \mu_{\Gr(k,k+m)}.
\end{equation*}
\end{NoNumberTheorem}

As amplituhedra are positive geometries, in \S\ref{sec:fiberposgeom} we study the more general framework of \emph{fiber positive geometries}. In this context, the positive Grassmannian can be seen as a simplex, thus the notion of fiber positive geometries is analogous to the notion of fiber polytopes. More precisely, let $\pi: X\dashrightarrow Y$ be a rational dominant map which restricts to a regular map $\pi_+: X_+\to Y_+$. Assume also that the pairs $(\pi^{-1}(y), \pi_+^{-1}(y))$ are  positive geometries for all $y \in Y_+$. Then we propose a set of properties that fiber positive geometries ought to have.
In particular, for a collection $\{X_a\}$ of strata of $X_+$ such that their images provide a dissection of $Y$, one can obtain the canonical form $\mathbf{\Omega}(Y_+)$ by summing over canonical forms of the images of such strata. 
Moreover, similar to the amplituhedron case we define the notion of volume form $\omega_\pi(y)$ on the fibers $\pi^{-1}(y)$ which would encode all the information about the triangulations of $Y_+$. Generalizing Theorem~\ref{thm:resfiber}, we expect that $\mathbf{\Omega}(Y_+)$ can be obtained via a residue procedure from $\omega_\pi(y)$. This leads us to the following question:
\begin{Question}\label{quest:general}
Is there a notion of fiber positive geometry for the projection $\pi: X\dashrightarrow Y$ which parametrizes (a nice class of) dissections of $(Y,Y_+)$ by a strata of $(X,X_+)$? 
What is the residue procedure on $\omega_\pi(y)$ which allows to obtain $\mathbf{\Omega}(Y_+)$ and reflects such combinatorics?
\end{Question}
We conclude by mentioning some examples relevant to this framework. In particular, we describe how a residue \emph{à la} Jeffrey-Kirwan can be used in the context of fiber polytopes. Finally, we consider an example of (linear) fiber positive geometry which comes from a Grassmann polytope; See Example~\ref{ex:linearfiber}.

\medskip\noindent{\bf Structure of the paper.}
In \S\ref{sec:polytopes} we review some of the main tools used in this work, including fiber polytopes, 
secondary fans, and Gale duality.
In \S\ref{sec:amplituhedron} we recall the definitions of the positive Grassmannian, the amplituhedron and its  triangulations. We also review the construction of the canonical forms of both objects and the parity duality for triangulations of the amplituhedron.
\S\ref{sec:fibers} contains some of our main results. 
In \S\ref{page:full_fiber} we define the positive and full fibers of the amplituhedron and provide them with a birational parametrization. In \S\ref{sec:formfibers} we introduce the fiber volume forms $\omega_{n,k,m}$ and express them explicitly in terms of the parametrization given in \S\ref{page:full_fiber}. We then explain the relation between the fiber volume form and the canonical form of the amplituhedron in \S\ref{subsec:fibervolformtriang}.
In \S\ref{sec:linear} we study triangulations of amplituhedra in the case of conjugate to polytopes and 
define a well-structured notion of secondary fan.
In \S\ref{sec:JK} we extend the classical Jeffrey-Kirwan residue and use it to compute the canonical form of the amplituhedron from the fiber volume forms.
Finally, in \S\ref{sec:fiberposgeom} we introduce the more general framework of fiber positive geometries and in \S\ref{sec:conclusions} we outline future directions.

\medskip
\noindent{\bf Acknowledgments.}
We would like to thank Hugh Thomas for pointing out the reference \cite{athanasiadis2000fiber} and the anonymous referees for their insightful comments, especially for pointing out the reference \cite{MR3908855}. M.P. would like to thank the generous support of the Sachs Scholarship at Princeton and the ERC Grant \#724638, and also 
the great hospitality of the School of Mathematics of the University of Bristol, in particular F.M., where this work has been developed. M.P. would also like to thank P. Benincasa, L. Ferro, T. Łukowski, and L. Williams for useful discussions.
F.M. was partially supported by the EPSRC Early Career Fellowship EP/R023379/1, the UGent Starting Grant BOF/STA/201909/038, and the FWO grants (G023721N, G0F5921N). L.M. was funded by the EPSRC 
EP/R023379/1. 

\medskip

\paragraph{Notation.}
Throughout we let $[n]=\{1,\ldots,n\}$ and we denote $\binom{[n]}{k}$ for the collection of subsets of $[n]$ of size $k$. We call $I_j:=\{j,\ldots, j+k-1\}$ the $j^{\rm th}$ {\em cyclic $k$-interval} of $[n]$ where all the numbers are considered modulo $n$. 
For any subset $I\subset [n]$ we denote $\bar{I}$ for its complement that is $\bar{I}=[n] \backslash I$.

\section{Fibers and triangulations of polytopes}\label{sec:polytopes}

In this section, we give a brief introduction to subdivisions of polytopes. Most of the results of this section are well-known, but we rephrase them to make them more suited to a generalization in the context of amplituhedra. 
We refer to \cite{billera1992fiber} and \cite[\S9]{ziegler} for further details.

A slight difference of our approach from the one in the literature (for polytopes) is that we aim to work with the secondary fan, rather than the secondary polytope. This method will, in particular, allow us to study the secondary fan of a hyperplane arrangement. Note that every convex polytope is naturally endowed with a hyperplane arrangement (corresponding to its facets).  We use this approach as it is more suitable for generalizations to the study of triangulations of amplituhedron with linear fibers. In particular, we use the framework developed along this section to extend the theory of secondary fans from convex polytopes to amplituhedra in \S\ref{sec:fan_conjugate_polytopes}. 

\medskip
\noindent{\bf Notation.} Here $Q\subset \mathbb{R}^n$ and $P\subset \mathbb{R}^m$ denote some convex polytopes. We further fix
\begin{equation}\label{eq:pi}
\text{a linear map}\ \pi:\R^n\to \R^m\ \text{which induces a surjective map}\ \pi_+:Q \to P,
\end{equation}
sending the vertices of $Q$ to the vertices of $P$.

\subsection{Triangulations and subdivisions of polytopes} 
\begin{definition}\label{def:R(P)}
{\rm A {\em subdivision} of a full-dimensional polytope $P\subset \mathbb{R}^m$ is a representation of $P$ as a finite union of a collection of full-dimensional polytopes $R(P)$ such that every two polytopes are either disjoint or intersect by a common proper face. Moreover, the vertices of the polytopes in $R(P)$ are vertices of $P$. 
A subdivision $R(P)$ is called a \emph{triangulation} of $P$ if all polytopes are simplices. A subdivision is called \emph{regular} if there exists a convex piecewise linear function which is linear on any polytope in $R(P)$.
}\end{definition}

Here, we are interested in regular subdivisions of $P$ which are induced by the map $\pi$ from~\eqref{eq:pi}. 

\begin{definition}
{\rm Given a projection $\pi_+: Q\to P$, a subdivision $R(P)$ of $P$ is called \emph{$\pi$-induced} if every polytope in $R(P)$ is the image of some face of $Q$. Regular $\pi$-induced subdivisions are called \emph{$\pi$-regular}.\footnote{Such subdivisions are also called \emph{$\pi$-coherent}, but in the context of amplituhedron, the term regular is more common.}
}\end{definition}

\begin{example}
Each regular subdivision in Definition~\ref{def:R(P)} is $\pi$-regular for some $\pi$. To see this correspondence, assume that $f$ is a piecewise linear function inducing the regular subdivision of $P$. Then one can consider $Q\subset \R^m\times \R$ to be the convex hull of the graph of $f$ and $\pi:\R^m\times \R \to\R^m$ the natural projection. 
\end{example}

\subsection{Fiber polytopes and \texorpdfstring {$\pi$}--regular subdivisions}   
To any linear projection of convex polytopes $\pi_+: Q\to P$, Billera and Sturmfels  \cite{billera1992fiber} associated the so-called \emph{fiber polytope} $\Sigma(Q, P)$ which is constructed by taking the Minkowski integral of fibers of $\pi_+$ as:
\[
\Sigma(Q, P) = \int_{x\in P}  \pi_+^{-1}(x) \,\,\mbox{d}x.
\]
The Minkowski integral of convex bodies can be defined as the limit of finite Riemann (Minkowski) sums, for a precise definition see \cite{billera1992fiber}. The fiber polytope encodes the $\pi$-regular subdivisions of $P$ as follows:

\begin{Theorem}[{\cite[Theorem 2.4]{billera1992fiber}}]\label{fiberpolytop}
The fiber polytope $\Sigma(Q, P)$ is a polytope of dimension $\dim(Q)-\dim(P)$,
whose faces correspond to $\pi$-regular subdivisions of $P$.
In particular, the vertices of $\Sigma(Q, P)$ correspond to the finest $\pi$-regular subdivisions, while the facets correspond to the coarsest subdivisions.
\end{Theorem} 

We aim to generalize this construction from polytopes to amplituhedra. To do so, we first describe the correspondence in Theorem~\ref{fiberpolytop} more concretely. For simplicity of the exposition we restrict ourselves to \emph{generic} linear projections of polytopes. A linear projection $\pi_+:Q\rightarrow P$ is called \emph{generic} if  $\pi_+(A)\ne\pi_+(B)$ for every pair of vertices $A,B$ of $Q$.

\medskip

In the following, given $P\subset \R^m$ and a linear function $\psi\in (\R^m)^*$ we denote by $P^\psi$ the \emph{truncation} of $P$ in the direction of $\psi$. That is, $P^\psi$ is the face of $P$ where $\psi$ attains its maximum.

\begin{definition}\label{height}
{\rm Consider the polytopes $Q\subset \R^n$ and $P\subset \R^m$. Let 
$\pi_+:Q \to P$ be a linear projection. The linear functional $\psi \in (\R^n)^*$ is called {\em generic} if its restriction to $\ker\, \pi$ is non-zero and $\psi(A)\ne\psi(B)$ for any pair of vertices $A, B$ of $Q$. For any generic functional $\psi$ we define a section $s_\psi:P \to Q$ of $\pi_+$ as follows. For any point $x\in P$, let $P_x\subset Q$ be its preimage under $\pi_+$. Then the section $s_\psi$ is given by
$$
s_\psi(x) =P_x^\psi, 
$$
where $P_x^\psi$ is the truncation of $P$ in the direction of $\psi$. Having a generic $\psi$ guaranties that $P_x^\psi$ is a vertex of $P_x$, and therefore the section $s_\psi$ is well-defined. 
}\end{definition}

\begin{remark}\label{rem:sec}
It is clear that the section $s_\psi$ in Definition~\ref{height} only depends on the restriction of the linear functional $\psi$ to $\ker\,\pi$. Thus one can think of $\psi$ as an element of $(\ker\,\pi)^*$, which we will do from now on.
\end{remark}

The theorem below is a reformulation of some results in \cite{billera1992fiber} which relates Theorem~\ref{fiberpolytop} with the sections from Definition~\ref{height}. We refer to the proof of \cite[Theorem 9.6]{ziegler} for further details.
\begin{Theorem}\label{triang}
Let $\pi_+:Q\to P$ be a linear projection and $\psi \in (\ker\,\pi)^*$ a generic linear functional.~Then:
\begin{itemize}
\item The image of the section $s_\psi(P)=\{P_x^\psi:\ x\in P\}$ is a union of some faces of $Q$, denoted by
$R_\psi(P)$.
\item The collection
$R_\psi(P)$ is a $\pi$-regular partition
of $P$.
\end{itemize}
Moreover, for any vertex $A$ of $\Sigma(Q,P)$, the corresponding $\pi$-regular subdivision is equal to $R_\psi(P)$ for each $\psi$ in the interior of $\sigma_A$, where $\sigma_A$ is the dual cone of $A$. 
\end{Theorem}
\begin{proof}
We give a sketch of the proof. Let $s_\psi(x)$ be contained in a face $\Gamma$ of $Q$. To show the first point, it is enough to show that $s_\psi(x')\in \Gamma$ for any $x'\in \pi_+(\Gamma)$. This follows from the construction of the section $s_\psi$, as $\psi$ attains its maximum at $\pi^{-1}(x') \cap \Gamma$ for any $x'\in \pi_+(\Gamma)$. The second point follows from the fact that $s_\psi$ is a section, and hence the images of faces $R_\psi(P)$ form a subdivision of $P$. To show that $R_\psi(P)$ is regular, notice that the function $\psi\circ s_\psi: P \to \R$ is convex piecewise linear with respect to subdivision $R_\psi(P)$. Finally, it is easy to see that for any pair of covectors $\psi,\psi'$ in the interior of some cone $\sigma_A$ of the dual fan of $\Sigma(Q,P)$, the subdivisions $R_\psi(P)$ and $R_{\psi'}(P)$ coincide.
\end{proof}

\subsection{Secondary fans and triangulations of polytopes} We now focus on the case of regular triangulations of convex polytopes. We recall the notions of \emph{secondary polytopes} and \emph{secondary fans} which were first defined in \cite{gelfand2008discriminants}. We follow the notation of \cite{billera1992fiber}.

\smallskip
Let $\Delta\subset\R^n$ be the standard simplex realized as a full-dimensional polytope with $n$ vertices. Then, any convex polytope $P\subset \R^m$ with $n$ vertices can be realized as the image of $\Delta$ under a linear projection $\pi_+:\Delta\to P$. 
For the convenience of the exposition, we will work with a linear translation of $\Delta$ to the vector space parallel to the affine span of $\Delta$.  Note that this procedure will not change the fiber polytope $\Sigma(\Delta , P)$. In this special case, the fiber polytope $\Sigma(\Delta , P)$ is called the {\em secondary polytope} of $P$ and the vertices of $\Sigma(\Delta , P)$ are in one-to-one correspondence with the regular triangulations of $P$. 
In what follows it will be more convenient for us to work with the dual fan of secondary polytope. 
The dual fan of $\Sigma(\Delta , P)$, denoted by $\mathcal{F}(\Delta,P)$, is called the \emph{secondary fan} of $P$. Here, we  denote them simply by $\Sigma_P$ and $\mathcal{F}_P$ respectively. Note that the dual fan naturally lives in $(\ker \,\pi)^*$ as $\Sigma_P$ can be realized in the space $\ker\,\pi \subset \R^{n-1}$ (uniquely up to translation).

\smallskip

To any collection $\rho_1,\ldots,\rho_k$ of rays in $\R^r$ we associate a so-called {\em chamber fan} as follows. For any subset $I\subset \{1,\ldots,k\}$, we let $C_I$ be the
cone generated by $\{\rho_i\}_{i\in I}$. We define an equivalence relation on the vectors of $\R^r$ as follows:
\begin{equation}\label{eq:eqrel}
 x\sim y\ \text{ if and only if }\ x\in C_I \Leftrightarrow y\in C_I.    
\end{equation}
It is easy to see that this is indeed an equivalence relation and that the equivalence classes are polyhedral cones, called \emph{chambers}. Moreover, the collection of chambers forms a polyhedral fan in $\R^r$.

\begin{definition}\label{def:chamberfan}
{\rm Given a collection $\rho_1,\ldots,\rho_k$ of rays in $\R^r$, its associated chamber fan is a fan whose cones are the equivalence classes arising from the equivalence relation \eqref{eq:eqrel}.
}\end{definition}

\begin{Proposition}\label{fiberfan}
Let $\pi_+:\Delta\to P$ be a projection as before. Let $\rho_1,\ldots,\rho_n \subset (\R^n)^*$ be outward directed normal rays to facets of $\Delta$. Then the secondary fan $\mathcal{F}_P$ is the chamber fan of $i^*(\rho_1),\ldots, i^*(\rho_n)\subset (\ker\,\pi)^*$ where 
$$
i^*:\R^n\to (\ker\,\pi)^*
$$
is the natural projection.
\end{Proposition}

\begin{proof}
By the construction of $\Sigma_P$, it is clear that $\mathcal{F}_P$ is the coarsest common refinement of the dual fans of all possible fibers of the projection $\pi_+$. Since $\pi$ is generic, the facets of a fiber of $\pi_+$ come as the intersection of facets of $\Delta$ with fibers of $\pi$.
Also, every facet of $\Delta$ contributes to a facet of some fiber $\pi_+^{-1}(y)$. Therefore, the fan $\mathcal{F}_P$ has $\rho_1,\ldots,\rho_n$ as rays. Moreover, since any $I\subset [n]$ corresponds to a face of $\Delta$ (and hence of $\pi_+^{-1}(y)$ for some $y\in P$), the coarsest common subdivision of the normal fans of fibers of $\pi_+$ is the chamber fan of $\rho_1,\ldots,\rho_n$. \end{proof}

In the rest of this subsection, we list simplices appearing in the triangulation $R_{\psi}(P)$ of $P$ induced by a linear functional $\psi \in (\ker\,\pi)^*$. We also numerate the vertices $v_1,\ldots,v_n$ of $\Delta$ such that $v_i$ is the vertex which is dual to the facet normal to $\rho_i$. Then we have the following theorem:

\begin{Theorem}\label{thm:explsymp}
Let $P$, $\mathcal{F}_P$, $\pi$, and $\psi\in (\ker\,\pi)^*$ be as above. Then the maximal simplices in the triangulation $R_{\psi}(P)$ of $P$ are in bijection with the collections $I=\{i_1,\ldots, i_{n-k}\}$ such that $\psi$ belongs to the cone generated by $i^*(\rho_{i_1}),\ldots, i^*(\rho_{i_{n-k}})$. Under this bijection, the collection $\{i_1,\ldots,i_{n-k}\}$ corresponds to the simplex whose vertices $\pi(v_{j})$ are labeled by $j\in [n]\backslash I$. 
\end{Theorem}

\subsection{Gale duality}\label{sec:galeduality}
The {\it Gale transform} is an involution that takes a (reasonably general) set $\Gamma$ of $n$ labeled points in $\R^m$ to a set $\Gamma^\perp$ of $n$ labeled points in $\R^{n-m}$, defined up to a linear transformation of $\R^{n-m}$.
More precisely, if we choose our coordinates so that the points in $\Gamma$ have as coordinates the columns of the matrix
\begin{equation}\label{mat}
M=\left(\frac{\mathbb{I}_{m}}{A}\right)^T,
\end{equation}
where $\mathbb{I}_{m}$ is the $m\times m$ identity matrix and $A$ is an $(n-m)\times m$ matrix, then the Gale transform of $\Gamma$ is the set of points $\Gamma^{\perp}$ whose coordinates in $\R^{n-m}$ are the rows of the matrix 
\begin{equation}\label{matperp}
M^\perp=\left(\frac{-A^{T}}{\mathbb{I}_{n-m}}\right).
\end{equation}
Note that $MM^\perp=0$. Given a polytope $P$ in the projective space $\mathbb{P}^{m}$ with the vertex set $V(P)=\{v_1,\ldots,v_n\}$, we denote  $(V(P))^\perp=\{v_1^\perp,\ldots,v_n^\perp\} \subset \mathbb{P}^{(n-m-1)}$ for the Gale transform of $V(P)$.

\medskip

We now draw the connection between the Gale transform and regular triangulations of polytopes. 

\begin{Theorem} \label{thgaledual}
Let $P$ be a polytope with the vertex set $V(P)=\{v_1,\ldots,v_n\}\subset \R^m$. Then $(V(P))^\perp$ is the set of linear generators of the rays of $\mathcal{F}_P$. More precisely, a maximal cone $\sigma$ of $\mathcal{F}_P$ corresponds to a triangulation of $P$ whose maximal simplices ${\rm conv}(v_{i_1},\ldots,v_{i_k})$ are in correspondence with the $k$-subsets $I=\{i_1,\ldots,i_k\}$ of $[n]$ where the cone generated by $\{v_j^\perp\}_{j\in [n]\backslash I}$ contains $\sigma$.
\end{Theorem}

\begin{proof}
Let $\pi:\R^n\to \R^m$ be a linear map given by a projection $\pi_+:\Delta\to P$, where $\Delta$ is the standard simplex with $n$ vertices.
If we choose the coordinates such that the coordinates of the vertices of $P$ form a matrix $M$ as in \eqref{mat}, then the corresponding matrix 
of the natural projection 
$(\R^n)^*\to (\ker\,\pi)^*$
will be given by $M^\perp$ as in \eqref{matperp}. The rest follows from Proposition~\ref{fiberfan}. 
\end{proof}

\subsection{Projective polytopes}
\begin{definition}\label{def:projpol}
{\rm A projective polytope is a subset $P\subset \mathbb{P}^{d}$ of a real projective space, which is homeomorphic to a closed $d$-dim ball and is bounded by projective hyperplanes.
}\end{definition}
A projective polytope $P\subset\mathbb{P}^{d}$ is uniquely determined by the pointed cone $C(P)\subset\mathbb{R}^{d+1}$ over $P$. Conversely, the projectivization
of any full-dimensional pointed cone $\sigma\subset \R^{d+1}$, i.e.~the image of  $\sigma$ in $\mathbb{P}^d$ under the natural projection $\R^{d+1}\setminus \{0\} \to \mathbb{P}^d$ is a projective polytope. 

\smallskip

The construction of the secondary fans of polytopes can be extended to {\it projective polytopes} as follows.
Let $P\subset\mathbb{P}^{d}$ be a projective polytope and $H\subset \P^d$ a hyperplane such that $P\cap H=\varnothing$. Then $P$ is contained in  the affine chart $\P^d\backslash H\simeq \R^d$ and can be viewed as a convex polytope in $\R^d$. 
On the other hand, the cone $C(P) \subset \R^{d+1}$ corresponds to the choice of a linear function $L\in (\R^{d+1})^*$ which is positive on $C(P)\backslash \{0\}$. Then the affine chart $\P^d\backslash H$ can be identified with the level set $L_1=\{x \in \R^{d+1} \,|\, L(x)=1\}$ and the aforementioned convex polytope in this chart is given by the intersection $C(P)\cap L_1$. By the secondary fan of a projective polytope we mean the secondary fan of its affine realization. We note that the secondary fan of a projective polytope depends on the choice of the hyperplane $H$ at infinity (resp. the linear function $L$), but its combinatorial structure is independent of $H$ as any pair of different affine realizations of $P$ can be obtained from each other by a linear transformation.

\section{The amplituhedron}\label{sec:amplituhedron}
We now introduce the main objects of this paper, including \emph{positive Grassmannians} and \emph{amplituhedra}. We also review the notions of \emph{positroidal triangulations} and \emph{canonical forms}, and their interrelations.

\subsection{The positive Grassmannian} \label{subsec:positiveGrass}
Throughout we fix a field $\mathbb{K}$ with char$(\mathbb{K})=0$. We are mainly interested in the case of $\mathbb{K}=\mathbb{R}$ or $\mathbb{C}$. 
The Grassmannian $\Gr(k, n)$ is the space of all $k$-dim linear subspaces of $\mathbb{K}^n$. For an abstract vector space $W$ over $\mathbb{K}$ we will denote by $\Gr(k, W)$ the Grassmannian of $k$-dim linear subspaces of $W$. To emphasize the underlying field, when $\mathbb{K}$ is the set of real or complex numbers, we may use $\Gr(k,\mathbb{R}^n)$ or $\Gr(k,\mathbb{C}^n)$ instead of $\Gr(k,n)$.  A point $V$ in $\Gr(k, n)$ can be represented by a $k\times n$ matrix with entries in $\mathbb{K}$. 
Let $X=(x_{ij})$ be a $k\times n$ matrix of indeterminates. For a subset $I = \{i_1,\ldots,i_k\} \in \binom{[n]}{k}$,  let $X_I$ denote the $k\times k$ submatrix of $X$ with the column indices $i_1,\ldots,i_k$. 
The \emph{Pl\"ucker coordinates} of $V$ are $p_I(V) = \text{det}(X_I)$ for $I \in\binom{[n]}{k}$ that do not depend on the choice of matrix $X$ (up to simultaneous rescaling by a non-zero constant) and determine the \emph{Pl\"ucker embedding} of
$\Gr(k,n)$ into $\mathbb{P}^{\binom{n}{k}-1}$. 
Moreover, any point $V\in\Gr(k,n)$ can be represented as $\mbox{span}\lbrace v_1,\ldots,v_k\rbrace$ for some $\mathbb{K}$-vector space basis $\{v_1,\dots,v_k\}$.

We denote the set of full-rank $k\times n$ matrices by $\mbox{Mat}(k,n)$. Let $\mathbb{K}=\mathbb{R}$. The set of totally positive matrices $\mbox{Mat}_+(k,n)$ contains the $k\times n$ matrices whose ordered maximal 
minors are all positive. The totally positive part of the Grassmannian $\Gr_+(k,n)$ 
is the subset of
$\Gr(k,n)$ where all Pl\"ucker coordinates are positive. Similarly, we denote $\Gr_{\geq}(k,n)$ for the totally non-negative part of the Grassmannian.
\medskip

The totally non-negative Grassmannian $\Gr_{\geq}(k,n)$ has a stratification as follows:
\begin{definition}
{\rm For each collection $M$ of elements of $ \binom{[n]}{k}$, we define $S_{M}$ as
\[
S_M=\{V\in \Gr_{\geq}(k,n):\ p_I(V)>0\text{ if and only if } I\in M \}.
\]
If $S_M\neq\emptyset$, then we call $M$ a \emph{positroid} and $S_M$ a \emph{positroid cell}. 
}\end{definition}
Each positroid cell $S_{M}$ is a topological cell \cite[Theorem 6.5]{postnikov2006total}, and moreover, the positroid cells of $\Gr_{\geq}(k,n)$ glue together to form a CW complex \cite{PSW}.
\begin{remark}\label{rem:boundary}
By definition, $\Gr_{\geq}(k,n)$ is a semialgebraic subset of the real part 
    of the complex algebraic variety $\Gr(k,\mathbb{C}^n)$. 
The boundary of $\Gr_{\geq}(k,n)$ in the analytic topology is $\Gr_{\geq}(k,n) \backslash \Gr_{+}(k,n)$. 
The boundary of $\Gr_{\geq}(k,n)$ is covered by $n$ hypersurfaces given by the equations $p_{I_j}=0$ for $j=1,\ldots, n$, where $I_j$ is the $j^{\rm th}$ cyclic $k$-interval of $[n]$. We will denote $S_j$ for the positroid cell defined by $p_{I_j}=0$.
\end{remark}


\subsection{The amplituhedron} \label{subsec:ampl}
The tree amplituhedron is defined by Arkani-Hamed and Trnka in \cite{arkani2014amplituhedron} as the image of the positive Grassmannian under a map induced by a totally positive matrix.

\begin{definition}\label{def:amp}
{\rm Any totally positive matrix $Z$ in 
$\Mat_+(k+m,n)$ with $k\leq k+m \leq n$ induces a map
\begin{eqnarray}\label{eq:A_{n,k,m}}
\pi_+:\Gr_+(k,n) \to \Gr(k,k+m)
\quad\text{defined by}\quad
\pi_+({\rm span}\lbrace v_1,\ldots,v_k\rbrace) := {\rm span}\lbrace Z\cdot v_1,\dots, Z\cdot v_k\rbrace, 
\end{eqnarray}
where ${\rm span}\lbrace v_1,\ldots,v_k\rbrace$ is an element of
$\Gr_+(k,n)$ written as the span of some $\mathbb{R}$-vector basis.
The \emph{(tree) amplituhedron} $\mathcal{A}_{n,k,m}(Z)$ is defined to be the image
$\pi_+(\Gr_+(k,n))$ inside $\Gr(k,k+m)$.
}\end{definition}

It is shown in \cite[\S4]{arkani2014amplituhedron} that the map $\pi_+$ is well-defined. In special cases the amplituhedron recovers familiar objects as follows.
If $k+m=n$, then $\mathcal{A}_{n,k,m}$ is isomorphic to
the positive Grassmannian $\Gr_+(k,k+m)$.
If $k=1$, then 
$\mathcal{A}_{n,1,m}$ is a {\it cyclic polytope} in the projective space $\mathbb{P}^m$, see \cite{sturmfels1988totally}.
If $m=1$, then $\mathcal{A}_{n,k,1}$ is the complex of 
bounded faces of a cyclic hyperplane arrangement, see \cite{karp2019amplituhedron}.

\begin{remark}\label{rem:Z}
Note that the definition of the amplituhedron depends on the choice of $Z$; however, its combinatorial and geometric properties are conjectured to be independent of the choice of $Z$.  For example, it is expected that
	the set of triangulations of the amplituhedron (in Definition~\ref{def:triangulations}) is independent of $Z$ (at least
	for even $m$, see e.g.~\cite{lukowski2020positive} for the $m=2$ case). Moreover, all the results proved
	in this paper are independent of the choice of $Z$. 
	Therefore, we neglect the matrix $Z$ in the definition of the amplituhedron, dissections, and triangulations, and we refer to the amplituhedron 
	as $\mathcal{A}_{n,k,m}$.  
\end{remark}

\subsection{Triangulations of the amplituhedron}\label{subsec:ampltriang}
We recall the definitions of positroidal dissections and triangulations of amplituhedra from \cite{lukowski2020positive}.
\begin{definition}\label{def:triangulations}
{\rm Let $\mathcal{C} = \{S_{M}\}$ be a finite collection 
of positroid cells of $\Gr_+(k,n)$.  We say that 
$\mathcal{C}$ is 
a \emph{(positroidal) dissection} of $\mathcal{A}_{n,k,m}$ if (for any choice of the initial data $Z$), we have that:
\begin{itemize}
	\item $\dim \, \overline{\pi_+(S_{M})}  = m\cdot k$ for each $S_M$ in $\mathcal{C}$.
	\item The images $\pi_+(S_{M})$ and $\pi_+(S_{M'})$ 
		of two distinct cells in $\mathcal{C}$ are disjoint.
	\item
	$\bigcup\overline{\pi_+(S_{M})}= \mathcal{A}_{n,k,m}$, i.e.~the union of the images of the cells in $\mathcal{C}$ is dense in $\mathcal{A}_{n,k,m}$.
	\end{itemize}
We denote $\overline{\pi_+(S_{M})}$ for the closure in the analytic topology. A positroidal dissection $\mathcal{C} = \{S_M\}$ of $\mathcal{A}_{n,k,m}$ is a \emph{triangulation} if $\pi_+$ is injective on each $S_{M}$. In this case, each $S_M$ is called a \emph{generalized triangle}.
}\end{definition}
With a slight abuse of notation, we will call 
$\overline{\pi_+(S_{M})}$ a generalized triangle, and denote it as $\Delta_M$.
Moreover, we will sometimes refer to the collection $\{\Delta_M \}$ itself as a positroidal triangulation (dissection).

\medskip\noindent{\bf Good dissections.} Among all possible positroidal dissections, there are some with particularly nice features, called \emph{good dissections}, for which the boundaries of the strata interact nicely. In particular, in addition to properties in Definition~\ref{def:triangulations}, they satisfy the following condition:
\begin{itemize}
    \item if the intersection $\overline{\pi_+(S_M)} \cap \overline{\pi_+(S_{M'})} $ of the images of two distinct cells has codimension one, then $\overline{\pi_+(S_M)} \cap \overline{\pi_+(S_{M'})}$ equals $\overline{\pi_+(S_{M''})}$, where $S_{M''}$ lies in the closure of both $S_M$
and $S_{M'}$.
\end{itemize}
We will sometimes refer to good dissections as \emph{positroidal subdivisions}. This is motivated by the connection between good dissections of certain type of amplituhedra with actual subdivisions of polytopes.

\subsection{Canonical forms}\label{sec:canform}
The geometry of the positive Grassmannian and amplituhedron can be supplemented with certain differential forms called \emph{canonical} forms. These forms are uniquely\footnote{The uniqueness of such form on a variety with boundary divisors with these properties 
is 
non-trivial. See \cite[\S 4]{Brown:2018omk}.}
defined by their property of having simple (logarithmic) poles along the boundary of the space they are associated to.
They were first introduced for the positive Grassmannians, and then were expanded for amplituhedra (see, e.g.~\cite{Arkani-Hamed:2016byb} and references therein).
In particular, Arkani-Hamed and Trnka in \cite{arkani2014amplituhedron} showed how such forms algebraically encode scattering amplitudes in $\mathcal{N}=4$ super Yang-Mills theory. 

\medskip

We now briefly recall the definition of \emph{residue}. We refer to \cite{resdua} as a standard literature as well as \cite{Griffiths}, \cite[\S2.2]{Khesbook} for more introductory texts.

\begin{definition}\label{def:residue}
{\rm Let $Y \subset X$ be an irreducible subvariety of codimension $k$, and $\bf{\Omega}$ a rational form
on $X$.
Let $\alpha_1, \alpha_2, \ldots, \alpha_d$ be local coordinates of $X$ and let $Y$ be locally determined by the vanishing locus $\alpha_1=\cdots=\alpha_k=0$. Further, suppose that
$\mathbf{\Omega}=\frac{\dd\!\alpha_1}{\alpha_1}\wedge\cdots\wedge\frac{\dd\!\alpha_k}{\alpha_k}\wedge \mathbf{\Omega}'$
with $\mathbf{\Omega}'$ of the form $f(\alpha_1, \alpha_2,\ldots, \alpha_d) \dd\!\alpha_{k+1} \wedge \dd\!\alpha_{k+2}\wedge \cdots\wedge \dd\!\alpha_d$ for some rational function $f$ which is locally analytic on $Y$. Then the \emph{residue} of $\mathbf{\Omega}$ along $Y$ is defined as ${\rm Res}_Y {\bf \Omega} := {\bf{\Omega}'}|_Y$. Note that the ${\rm Res}_Y \mathbf\Omega$ is only defined up to sign, which depends on the order of the functions $\alpha_1,\ldots,\alpha_k$.
}\end{definition}

\subsubsection{Canonical form of the positive Grassmannian}
Recall from Remark~\ref{rem:boundary} that the boundary of $\Gr_{\geq}(k,n)$ in the analytic topology is $\Gr_{\geq}(k,n) \backslash \Gr_{+}(k,n)$. Following \cite{lam2014totally, Arkani-Hamed:2016byb}, let $D$ be the Zariski closure of $\Gr_{\geq}(k,n) \backslash \Gr_+(k,n)$ in $\Gr(k,\mathbb{C}^n)$. 
Note that each irreducible component $D_i$ of $D$ has a natural semialgebraic subset ${D_i}_+$ which is the intersection of its real locus ${D_i}_\mathbb{R}$ with $\Gr_+(k,n)$.
Then there exists a unique, up to scaling, top-degree rational form $\mathbf\Omega(\Gr_+(k,n))$ on $\Gr(k,\C^n)$ with simple poles along $D$.
Similarly, the residue of $\mathbf\Omega(\Gr_+(k,n))$ along $D_i$ is the unique top-degree rational differential form on ${D_i}_+$ with simple poles along the boundaries of ${D_i}_+$. 
In general, each positroid cell $S_M$ carries a similar structure, i.e.~each $S_M$ admits a unique,  up to scaling, top-degree rational form $\mathbf\Omega(S_M)$ on its complexification with simple poles along its boundary. We call $\mathbf\Omega(S_M)$ the \emph{canonical form} of $S_M$. As before, for any positroid cell $S_{M'}$ in the boundary of $S_{M}$ (which is of codimension one), the canonical form $\mathbf\Omega(S_{M'})$ can be obtained as a residue of $\mathbf\Omega(S_M)$ along the complexification of $S_{M'}$, that is $\mbox{Res}_{S_{M'}}(\mathbf{\Omega}(S_M))=\mathbf{\Omega}(S_{M'})$. We note that the canonical form of each positroid cell can be computed from $\mathbf\Omega(\Gr_+(k,n))$ by successively taking residues along positroid boundaries. More generally, complex varieties with such structures are called \emph{positive geometries}. See \S\ref{sec:fiberposgeom} and \cite{positivegeom} for an overview of the latter.

\smallskip

We now introduce the standard volume form on the Grassmannian as follows.

\begin{definition}\label{def:formbase}
{\rm Let $C$ be a generic matrix of indeterminates in  $\Gr_+(k,n)$. The \emph{standard volume form} of $\Gr(k,n)$ is defined as \footnote{it is a covariant form of weight $n$; See \S\ref{sec:formfibers} and in particular, Definition~\ref{page:covariant}.}
\begin{equation*}
    \mu_{\Gr(k,n)}(C):= \bigwedge_{\alpha=1}^k \langle C {\rm d}^{n-k} C_{\alpha}\rangle, \quad\text{where}\quad  \langle C \dd^{n-k} C_\alpha \rangle := \sum_{I \in \binom{[n]}{k}} \epsilon_{I \bar{I}} \, p_I(C) \bigwedge_{\bar{i} \in \bar{I}} \dd C_{\alpha \bar{i}}.
\end{equation*}
Here $C_\alpha=(c_{\alpha 1},\ldots,c_{\alpha n})$ is the $\alpha^{\rm th}$ row of $C$ and $\epsilon_{j_1, \ldots, j_n}$ 
is the totally anti-symmetric Levi-Civita  symbol. That is $\epsilon_{j_1, \ldots, j_n}$ equals to the sign of the permutation $(1, \ldots, n) \rightarrow (j_1, \ldots, j_n)$. Note that if $j_s=j_r$ for some $s$ and $r$, then $\epsilon_{j_1, \ldots, j_n}$ is defined to be zero.
}\end{definition}

\begin{remark}
Let $C$ be a generic matrix of indeterminates in  $\Gr_+(k,n)$. Then following \cite{Arkani-Hamed:2016byb,lam2014totally}, we can {\em explicitly} write the canonical form of $\Gr_+(k,n)$ as follows:
\begin{equation*}\label{cangras}
    \mathbf{\Omega}\left( \Gr_+(k,n) \right)=\frac{\mu_{\Gr(k,n)}(C)}{\prod_{j=1}^n p_{I_j}(C)}.
\end{equation*}
Note that the formula above is invariant under $\GL_k$-action, hence the form is well-defined on $\Gr(k,n)$.
\end{remark}

\subsubsection{Canonical form of the amplituhedron} 
The amplituhedron $\mathcal{A}_{n,k,m}$ is conjectured to be a positive geometry \cite{arkani2014amplituhedron, lam2014totally}. In particular, one can obtain its canonical form $\mathbf{\Omega}\left( \mathcal{A}_{n,k,m} \right)$ by knowing its positroidal triangulations.

\begin{Conjecture}[\cite{arkani2014amplituhedron, lam2014totally}]\label{conj:canformfromtriang}
Let $\mathcal{C}=\lbrace S_{M}\rbrace$ be a positroidal triangulation of $\mathcal{A}_{n,k,m}$. Then
\begin{equation}\label{eq:AHformTriang}
 \mathbf{\Omega}\left(\mathcal{A}_{n,k,m} \right)=\sum_{S_M\in\mathcal{C}}   \mathbf{\Omega}\left(\pi_+(S_{M}) \right).
\end{equation}
In particular, the sum  does not depend on the positroidal triangulation $\mathcal{C}$.
\end{Conjecture}

\begin{remark}\label{rk:orientation2}
As explained in \cite[\S8]{galashin2018parity}, the signs of $\mathbf{\Omega}\left( \pi_+(S_{M}) \right)$ in the right hand side of \eqref{eq:AHformTriang} are chosen according to a fixed orientation of $\Gr(k,k+m)$. More concretely, let us fix a top-degree form on $\Gr(k,k+m)$ which is non-vanishing on $\mathcal{A}_{n,k,m}$. As an example of such form we can take $\mu_{\Gr(k,k+m)}$ from Definition~\ref{def:formbase}. Then the sign of $\mathbf{\Omega}\left( \pi_+(S_{M}) \right)$ is chosen such that 
$$
\frac{\mathbf{\Omega}\left( \pi_+(S_{M}) \right)(Y)}{\mu_{\Gr(k,k+m)}(Y)}>0, \text{ for any } Y\in \pi_+(S_{M}).
$$
\end{remark}

\begin{remark} 
We can compute $\mathbf{\Omega}\left(\mathcal{A}_{n,k,m} \right)$ from $\mathbf{\Omega}\left(\Gr_+(k,n) \right)$ by knowing a positroidal triangulation of $\mathcal{A}_{n,k,m}$.
More precisely, the form $\mathbf{\Omega}\left( \pi_+(S_{M})\right)$ in \eqref{eq:AHformTriang} can be obtained from $\mathbf{\Omega}\left(S_{M} \right)$ using the diffeomorphism $\pi_+$ between the positroid cell $S_{M}$ and its image in $\mathcal{A}_{n,k,m}$ (see Conjecture~\ref{conjtriangles}). In turn, as described before, the canonical form of a  positroid cell 
$S_{M}$ can be obtained as: $\mathbf{\Omega}\left( S_M \right)=\mbox{Res}_{S_M}\big( \mathbf{\Omega}\left(\Gr_+(k,n)\right)\big).$
\end{remark}

\noindent Since $\mathbf{\Omega}\left( \mathcal{A}_{n,k,m} \right)$ is a top-degree form on $\Gr(k,k+m)$, it can be decomposed into a product of a covariant function and
a covariant top-degree form. This leads to the following definition; See also Definition~\ref{page:covariant}.
\vspace{-3mm}
\begin{definition}[\cite{arkani2014amplituhedron}]\label{def:canonical_form}
{\rm Let $ \mathbf{\Omega}(\mathcal{A}_{n,k,m})$ be the canonical form of the amplituhedron $\mathcal{A}_{n,k,m}$ and let $\Omega (\mathcal{A}_{n,k,m})$ be a covariant function of weight $-(m+k)$ on $\Gr(k,k+m)$ such that
\begin{equation}\label{volumefunc}
    \mathbf{\Omega}(\mathcal{A}_{n,k,m})=\Omega (\mathcal{A}_{n,k,m}) \, \mu_{\Gr(k,k+m)}.
\end{equation}
Then we call $\Omega(\mathcal{A}_{n,k,m})$ the \emph{canonical function} of the amplituhedron $\mathcal{A}_{n,k,m}$.
}\end{definition}

\vspace{-4mm}

\subsection{Parity Duality} \label{subsec:paritydual}
Throughout we fix an even $m$ and we set $\ell:=n-m-k$.
We aim to review how the amplituhedra obtained by exchanging $k$ and $\ell$ are related by the so-called \emph{parity duality}. 
The definition of parity duality was originally inspired by the physical operation of parity conjugation in quantum field theory. More specifically, in the context of scattering amplitudes in $\mathcal{N} = 4$ super Yang-Mills theory  \cite{arkani2014amplituhedron}, where amplitudes can be computed from the geometry of $\mathcal{A}_{n,k,4}$. 
Furthermore, the conjectural formula given in \cite{Karp:2017ouj} for the number of cells in each triangulation of the amplituhedron is invariant under the operation of swapping the parameters $k$ and $\ell$, which has motivated further studies \cite{ferro2018amplituhedron,galashin2018parity}. 

\begin{Theorem}[{\cite[Theorems 3.5 and 7.2]{galashin2018parity}}]\label{th:parityduality}
There is a bijection between generalized triangles of $\mathcal{A}_{n,k,m}$ and generalized triangles of $\mathcal{A}_{n,\ell,m}$.
This bijection induces a one-to-one correspondence between positroidal triangulations of $\mathcal{A}_{n,k,m}$ and positroidal triangulations of $\mathcal{A}_{n,\ell,m}$. Moreover, this extends to positroidal dissections and their maximal cells.
\end{Theorem}

\begin{remark}
For $m=2$, it was shown in \cite{lukowski2020positive} that the bijection in Theorem~\ref{th:parityduality} comes naturally from composing the Grassmannian duality with \emph{T-duality}, which relates dissections of $\A_{n,k,2}$ with those of the hypersimplex. The authors proposed a similar origin of parity duality for general even $m$.
\end{remark}
We recall that $\mathcal{A}_{n,1,m}$ is a cyclic polytope in $\mathbb{P}^m$. Here, we consider the case $\mathcal{A}_{n,n-m-1,m}$ related to the latter by parity duality and we refer to it as the amplituhedron \emph{conjugate to polytope}.
 
\medskip 

We give a simple characterization for generalized triangles of amplituhedra conjugate to polytopes. 
 
\begin{Lemma}\label{lem:gentriangles}
For $\ell=n-m-1$, there is a bijection between generalized triangles of 
$\mathcal{A}_{n,\ell,m}$ and $\ell$-subsets of $[n]$.
In particular, for any $\ell$-subset $J$ there exists a unique generalized triangle $\Delta_J$ which is the image of a positroid cell $S_{\widetilde{J}}$
in $\Gr_+(\ell,n)$ for which a cyclic $\ell$-interval $I_j$ is in $\widetilde{J}$ if and only of $j\not\in J$. Moreover, the bijection generalizes to maximal cells in subdivisions of $\mathcal{A}_{n,\ell,m}$ and $r$-subsets of $[n]$ for $1 \leq r \leq \ell$.
\end{Lemma}

\begin{proof}
Note that $\mathcal{A}_{n,\ell,m}$ is conjugate to the polytope $\mathcal{A}_{n,1,m}$ whose generalized triangles are simplices in $\mathbb{P}^m$. Given an $\ell$-subset of $[n]$,
the simplex $\bar{\Delta}_J$ (with vertices corresponding to the elements of $[n] \backslash J$) is the image of the positroid cell in $\Gr_+(1,n)$ whose vanishing Pl\"ucker coordinates are $p_j$ for $j\in J$.
Moreover, all generalized triangles in $\mathcal{A}_{n,1,m}$ arise in this way. 
Then by Theorem~\ref{th:parityduality}, there is a bijection between the set of simplices $\bar{\Delta}_J$ defined above and the set of generalized triangles $\Delta_J$ of $\mathcal{A}_{n,\ell,m}$. More precisely, \cite[Lemma~11.5]{galashin2018parity} provides an explicit relation between cyclic Pl\"ucker coordinates of generalized triangles related by parity duality.
In particular, for matrices $\overline{C} \in \Gr_+(1,n)$ and $C \in \Gr_+(\ell,n)$ representing positroid cells corresponding to $\bar{\Delta}_J$ and $\Delta_J$, we have that $p_{j}(\overline{C})=0$ if and only if $p_{I_j}(C)=0$. 
Therefore, the only vanishing Pl\"ucker coordinates of $\Delta_J$ with cyclic intervals are $p_{I_{j}}$ with $j \in J$. This implies the existence of generalized triangles of $\mathcal{A}_{n,\ell,m}$ with such a property. The uniqueness follows by the fact that the parity duality provides a bijection between the generalized triangles of $\mathcal{A}_{n,\ell,m}$ and those of $\mathcal{A}_{n,1,m}$. 

We can easily generalize the proof for $r$-subsets $J$ with $1 \leq r \leq \ell$. The polytope $\bar{P}_J$ (with vertices corresponding to the elements of the set $[n] \backslash J$) is just the image of the positroid cell in $\Gr_+(1,n)$ whose vanishing Pl\"ucker coordinates are $p_j$ for $j\in J$. Using parity duality, we can then map these to maximal cells in subdivisions of $\mathcal{A}_{n,\ell,m}$. These will be images of positroid cells $S_{\tilde{J}}$ in $\Gr_+(\ell,n)$ for which a cyclic $\ell$-interval $I_j$ is in $\widetilde{J}$ if and only if $j\not\in J$.
\end{proof}
 
\section{Fibers of amplituhedra and their volume forms}\label{sec:fibers}
In this section, we study the fibers 
 of the amplituhedron 
 and we show how the canonical form of the amplituhedron 
 can be obtained by a residue computation of a volume form 
 defined on fibers. 
 
\medskip

\paragraph{Notation.} We denote the $m\times m$ identity matrix by $\mathbb{I}_m$. Given an $m\times n$ matrix $A$ we denote $A^T$ for its transpose, and $A^I_J$ for its submatrix on the columns indexed by $I\subseteq[n]$ and rows indexed by $J\subseteq[m]$. When $J=[m]$, we simply write $A^I$ for $A^I_{J}$. We also denote $(A^I)^\perp$ for the orthogonal complement 
of $A^I$.

 \subsection{Amplituhedron fibers}\label{page:full_fiber}
 \begin{definition}\label{def:pi}
{\rm The projection $\pi_+:\Gr_+(k,n)\to \mathcal{A}_{n,k,m}$ of the positive Grassmannian induced by a totally positive matrix $Z$ in \eqref{eq:A_{n,k,m}} can be extended to a rational map 
\[
\pi:\Gr(k,n)\dashrightarrow\Gr(k,k+m),
\]
which is defined on the open dense subset 
$$U=\{  V:\ V {\rm\  is\ a \ subspace\ of\ } \Gr(k,n)\ {\rm and}\ V\cap \ker\, Z  = \{0\}\}\subset \Gr(k,n).$$ 
Given a point $Y$ in $\mathcal{A}_{n,k,m}$ we denote the \emph{full} and \emph{positive} fibers of $Y$ by $\pi^{-1}(Y)\subset \Gr(k,n)$ and $\pi_+^{-1}(Y)\subset \Gr_+(k,n)$, where the positive fiber $\pi_+^{-1}(Y)$ is given by $\pi^{-1}(Y) \cap \Gr_+(k,n)$.
} \end{definition}

To explicitly parametrize the full and positive fibers of the amplituhedron, we consider the matrix $Z \in \mbox{Mat}_{+}(m+k,n)$ from Definition~\ref{def:amp}. The matrix $Z$ induces a linear map 
$$
Z:\R^n\to\R^{m+k},
$$ 
which we will denote by the same letter. Note that any point $Y$ in $\mathcal{A}_{n,k,m}$ is a $k$-dim subspace in $\mathbb{R}^{m+k}$; therefore $Z^{-1}(Y)$ is a subspace of $\mathbb{R}^n$ of dimension $\mbox{dim}(Y)+\mbox{dim}(\mbox{ker}\, Z)=n-m$.

\begin{Proposition} \label{prop:fiber} 
Let $Y$ be a point in the amplituhedron $\mathcal{A}_{n,k,m}$. Then the following hold:
\begin{itemize}
    \item[{\rm(i)}] The full and positive fibers of $Y$ can be birationally parametrized by
    \[
     \Gr(k,Z^{-1}(Y)) \simeq \Gr(k,n-m) \dashrightarrow \pi^{-1}(Y)\ \text{ and }\ \pi^{-1}_+(Y)=\Gr_+(k,n) \cap \Gr(k,Z^{-1}(Y)) \subset \Gr(k,n-m).
    \]

\item[{\rm (ii)}] Fix a basis of $Z^{-1}(Y)$ and let $A \in {\rm Mat}(n-m,n)$ be the matrix of coordinates of its vectors with respect to the standard basis of $\mathbb{R}^n$.
Then $V$ is a $k \times n$ matrix representing an element of $\pi^{-1}(Y)$ if and only if there exists an element in $\Gr(k,n-m)$ whose 
representing matrix $\lambda$ satisfies:
\begin{equation}\label{eq:Propii}
    V=\lambda \cdot A.
\end{equation}

\item[{\rm (iii)}] The matrix $V$ represents an element of the positive fiber $\pi^{-1}_+(Y)$ if and only if
\begin{equation}\label{poscond}
p_I(V)=\sum_{J \in \binom{[n-m]}{k}} p_J(\lambda) \, \det (A^I_J)>0, \quad \text{for each}\ I \in \binom{[n]}{k}.
\end{equation}
\end{itemize}
\end{Proposition}

\begin{proof}
Any element of the full fiber of $Y$ is a $k$-dim subspace of $Z^{-1}(Y)$. In other words, $\pi^{-1}(Y)=\Gr(k,Z^{-1}(Y))$. Since $Z^{-1}(Y)$ is an $(n-m)$-dim vector subspace of $\R^n$, we can pick a basis $\lbrace a_1,\ldots,a_{n-m}\rbrace$ with $a_l \in \mathbb{R}^n$ for each $1\leq l\leq n-m$. Therefore, any element of $\Gr\!\left(k,Z^{-1}(Y)\right)$ is a $k$-dim subspace spanned by some vectors $v_1,\ldots,v_k$ with $\mbox{span}\lbrace v_1,\ldots,v_k \rbrace \subseteq \mbox{span}\lbrace a_1,\ldots,a_{n-m}\rbrace$. 
More precisely,
\begin{equation*}
    v_\alpha=\sum_{l=1}^{n-m}\lambda_{\alpha l} a_l\quad \text{ for }\quad\alpha=1,\ldots,k,
\end{equation*}
where the coefficients $\lambda_{\alpha l}$ form a matrix $\lambda \in \mbox{Mat}(k,n-m)$. Now by considering the standard basis of $\mathbb{R}^n$ and the notation above, we denote the coordinates of the vectors $v_{\alpha}$ and $a_{l}$ with $v_{\alpha i}$ and $a_{l i }$, respectively. In terms of such coordinates we can write:
\begin{equation*} 
    V=\lambda \cdot A,\, \text{ where }  V \in {\rm Mat}(k,n), \, \lambda \in \mbox{Mat}(k,n-m) \, \text{ and } A \in \mbox{Mat}(n-m,n).
\end{equation*}
Since $V$ represents an equivalence class of points in 
$\Gr(k,n)$,  so does $\lambda$ in $\Gr(k,n-m)$. Therefore $\Gr\!\left(k,Z^{-1}(Y)\right) \simeq
\Gr(k,n-m)$ and it has dimension $k(n-m-k)$. An element in the positive fiber $\pi^{-1}_+(Y)$ has to be an element of $\Gr_+(k,n)$, therefore $V$ has to satisfy \eqref{poscond}
where $p_I(V)$ and $p_J(\lambda)$ are Pl\"ucker coordinates of $V$ and $\lambda$, respectively. 
\end{proof} 

\begin{remark}\label{page:AandZ}
We can characterize the space $Z^{-1}(Y)$ and therefore the matrix $A$ more explicitly, as: 
\begin{itemize}
    \item[(i)] The space $Z^{-1}(Y)$ can be realized as 
\begin{equation*}
\mbox{span}\lbrace b_1,\ldots,b_{n-m-k}, c_1(Y),\ldots,c_k(Y)\rbrace,
\end{equation*}
where $\lbrace b_1,\ldots,b_{n-m-k}\rbrace$ is a basis of $\mbox{ker}\, Z$ and $\{c_1(Y),\ldots,c_k(Y)\}$ is a basis of the complement subspace of $\mbox{ker}\, Z$ which is obtained from particular independent solutions of the equations:
\begin{equation*}
    Y_\alpha=Z(c_\alpha) \quad \text{ for }\quad \alpha=1,\ldots,k.
\end{equation*}

\item[(ii)] Let
$Z=\big(\widetilde{z}_{(m+k)\times (m+k)}|z_{(m+k)\times (n-m-k)}\big)^T$. Then 
one can take $\lbrace b_1,\ldots,b_{n-m-k}\rbrace$ to be the rows of the matrix $Z^\perp:=\big((-\widetilde{z}^{-1} \cdot z )^T|\mathbb{I}_{n-m-k}\big)$ and $c_\alpha(Y)$ to be $$c_\alpha(Y):=\big((Y_\alpha)_{1 \times m+k} \cdot \widetilde{z}^{-1}|\mathbb{0}_{1 \times {(n-m-k)}}\big).$$ Therefore, with this choice of $Z$ we have that
\begin{equation}\label{matrixA}
A=
\left(
\begin{array}{c|c}
-(\widetilde{z}^{-1} \cdot z)^T & \mathbb{I}_{n-m-k} \\
\hline
Y \cdot \widetilde{z}^{-1} & \mathbb{0}_{k \times (n-m-k)}
\end{array}
\right).
\end{equation}

\item[(iii)] As mentioned in Remark~\ref{rem:Z}, the geometric properties of the amplituhedron is independent of the choice of $Z$. Hence, without loss of generality we assume that $\widetilde{z}=\mathbb{I}_{m+k}$. From now on, we will often work with this choice of $Z$ for the concrete computations; See e.g.~Lemma~\ref{dets}. 
\end{itemize}
\end{remark}

\begin{remark}
We would like to remark that the fibers of the amplituhedron are expected to have nice topological properties. For example, in \cite{MR3908855} the authors showed that for the amplituhedra conjugate to polytopes, the fibers are homeomorphic to a ball.
\end{remark}

In the following we provide a geometric description of the boundaries of the positive fibers. 

\begin{Proposition}\label{prop:boundaries}
Let $Y$ be an element of $\mathcal{A}_{n,k,m}$ and $A$ be the matrix in \eqref{eq:Propii}. 
The boundaries of the positive fiber $\pi_+^{-1}(Y)$ are (some of) the following hypersurfaces:
\begin{equation*}
    \mathcal{B}_i:= \lbrace \lambda \in \Gr(k,n-m):\ \dim \left( \lambda \cap (A^{I_i})^\perp \right) \geq 1  \rbrace\quad\text{for}\quad i=1,\ldots,n.
\end{equation*}
\end{Proposition}

\begin{proof}
First note that for the cyclic $k$-interval $I$, the submatrix $A^{I}$ represents an element of $\Gr(k,n-m)$ and its orthogonal complement represents an element in $\Gr(n-m-k,n-m)$.
Now, by definition, the boundaries of the positive fiber $\pi_+^{-1}(Y)$ are the intersection of $\Gr_+(k,n)$ with the full fiber  $\pi^{-1}(Y)$. More precisely, for $V \in \pi_+^{-1}(Y)$ they occur when $p_{I}(V)=0$ for cyclic $k$-intervals $I$ of $[n]$. Moreover, we have: 
\begin{equation*}
   \mbox{det}(\lambda|(A^I)^\perp)=\sum_{J,L} \epsilon_{J L} p_J(\lambda) p_L((A^I)^\perp)=\sum_{J,L,S} \epsilon_{J L} p_J(\lambda) \epsilon_{SL} p_S(A^I)=\sum_{J} p_J(\lambda) p_J(A^I)=p_I(V)
\end{equation*}
where $J$ and $S$ run over $\binom{[n-m]}{k}$ and $L$ runs over $\binom{[n-m]}{\ell}$.
Therefore, on the boundary $p_I(V)=0$ one has  $\mbox{det}(\lambda|(A^I)^\perp)=0$. Since both $\lambda$ and $(A^I)^\perp$ have full rank, the claim follows immediately.
\end{proof}

\begin{remark}\label{rem:dual}
We observe that, instead of working with $\lambda$ representing an element in $\Gr(k,n-m)$, we can consider  $\bar{\lambda}=\lambda^\perp$
representing an element in $\Gr(\ell,n-m)$ for $\ell=n-m-k$. More precisely, using the standard Grassmannian duality we have $\pi^{-1}(Y) \simeq \Gr(k,n-m) \simeq \Gr(\ell,n-m)$. Now, it is straightforward to see that, mimicking Proposition~\ref{prop:boundaries}, one has boundaries on
\begin{equation}\label{eq:dual_boundaries}
    \overline{\mathcal{B}}_i:= \lbrace \bar{\lambda} \in \Gr(\ell,n-m): \dim \left( \bar{\lambda} \cap A^{I_i} \right) \geq 1  \rbrace.
\end{equation}
\end{remark}
\begin{remark}
A consequence of the above parametrization of fibers of $\mathcal{A}_{n,k,m}$ is that the \emph{parity duality} manifests itself as the simple Grassmannian duality between $\Gr(k,n-m)$ and $\Gr(\ell,n-m)$.
In particular, the boundaries $\mathcal{B}_i$ of fibers of $\mathcal{A}_{n,k,m}$ correspond to the boundaries $\overline{\mathcal{B}}_i$ of $\mathcal{A}_{n,\ell,m}$.  See, e.g.~\cite{galashin2018parity}.
\end{remark}

We now provide two examples which are parity dual to each other, namely cyclic polytopes and their conjugates. We show that the positive fibers of these amplituhedra are bounded by some hyperplanes. 

\begin{example}[Cyclic polytopes]\label{ex:polytope2}
The positive fibers of $\mathcal{A}_{n,1,m}$ are given by elements $V=(v_i)$ in $\Gr(1,n)$. 
Note that in this case $\lambda \in \Gr(1,n-m) \simeq \mathbb{P}^{n-m-1}$. Moreover, each cyclic $1$-interval $I_i$ is the singleton $\{i\}$, hence $A^{I_i}$ is a vector which we simply denote it by $A^i$. The positive fibers of $\mathcal{A}_{n,1,m}$ are therefore regions in $\mathbb{P}^{n-m-1}$ bounded by (some of) the hyperplanes 
which are normal to the vectors $A^i$. By Definition~\ref{def:projpol}, $\pi^{-1}_+(Y)$ is therefore a projective polytope in $\mathbb{P}^{n-m-1}$. Alternatively, by Proposition~\ref{prop:boundaries}, the boundaries $\mathcal{B}_i$ 
are hyperplanes whose normal vectors are $A^i$.
Let us now choose $A$ such that its first $n-m-1$ rows span $\mbox{ker}\, Z$, e.g.~as in \eqref{matrixA}. By choosing an affine chart for which $\lambda_{n-m-1}=1$, we note that the normal fan of this affine realization of the projective polytope is the Gale dual of $Z$, as expected by Theorem~\ref{thgaledual}. In particular, the rays of the secondary fan of $\mathcal{A}_{n,1,m}$ are the columns of $Z^\perp$.
\end{example}

\begin{example}[Conjugate to polytopes]\label{cor:conjfibers}
Consider the amplituhedron $\mathcal{A}_{n,n-m-1,m}$ and fix a point $Y$ in it.
The positive fiber $\pi_+^{-1}(Y)$ are bounded by a configuration of hyperplanes in $\Gr(n-m-1,n-m)$.
Moreover, its facets are contained on (some of) the hyperplanes defined by $A^{I_i}$ for cyclic $(n-m-1)$-intervals $I_i$ of $[n]$. 
To see this, note that by Remark~\ref{rem:dual}, the boundaries of the positive fibers $\pi_+^{-1}(Y)$ are the hypersurfaces $\overline{\mathcal{B}}_i$ in \eqref{eq:dual_boundaries} with $i \in [n]$. These, in particular, contain all the elements $\bar{\lambda}$ in $\Gr(1,n-m) \simeq \mathbb{P}^{n-m-1}$ which intersect the hyperplane $A^{I_i}$. Therefore, each hyperplane $\overline{\mathcal{B}}_i$ coincides with the hyperplane $A^{I_i}$.
\end{example}

\subsection{Fiber volume form of the amplituhedron} \label{sec:formfibers}
Our goal in this section is to exploit the fibers of $\mathcal{A}_{n,k,m}$ in order to compute the canonical form  $\mathbf{\Omega}(\mathcal{A}_{n,k,m})$. For this purpose, we decompose 
$ \mathbf{\Omega}(\Gr_+(k,n))$ in two parts: one only depending on the base coordinates of $\mathcal{A}_{n,k,m}$, and the other one depending on the fiber coordinates. 
To do this, we recall from~\eqref{volumefunc} that the canonical form of $\mathcal{A}_{n,k,m}$ can be written in terms of its canonical function as:
\begin{equation*}
    \mathbf{\Omega}(\mathcal{A}_{n,k,m})=\Omega (\mathcal{A}_{n,k,m}) \, \mu_{\Gr(k,k+m)}(Y).
\end{equation*}
We will introduce the notion of \emph{fiber volume form} $\omega_{n,k,m}$ for a general amplituhedron and explicitly compute it in terms of the birational parametrization of fibers given in \S\ref{sec:fibers}. 
To encode all the information of 
$\Omega (\mathcal{A}_{n,k,m})$ in the fiber volume form, we choose the pullback of $\mu_{\Gr(k,k+m)}$ in $\Gr(k,n)$ via $\pi$ as the volume form of the base, and make the following definition:
\begin{definition}\label{def:omegafiber}
{\rm Let $\omega$ be a $k(n-m-k)$ form
on $\Gr(k,n)$ such that
\begin{equation}\label{eq:form}
    \mathbf{\Omega}(\Gr_+(k,n))=\omega \wedge \pi^*( \mu_{\Gr(k,k+m)}).
\end{equation}
Here, $\pi^*( \mu_{\Gr(k,k+m)})$ is the pullback of $\mu_{\Gr(k,k+m)}$ in $\Gr(k,n)$ via the map $\pi$.
We define $\omega_{n,k,m}$ to be the \emph{volume form  on the fiber} $\pi^{-1}(Y)$ of
$\mathcal{A}_{n,k,m}$ by restricting $\omega$ to the fiber:
\begin{equation*}
    \omega_{n,k,m}:=\left. \omega \right|_{\pi^{-1}(Y)}.
\end{equation*}
}\end{definition}

\begin{remark}
Note that the form $\omega$ in Definition~\ref{def:omegafiber} is a covariant form of degree $-(m+k)$ which is not unique. But its restriction to the fiber does not depend on the choice of $\omega$. To see this, let $\omega'$ be another form on $\Gr_+(k,n)$, such that 
$\mathbf{\Omega}(\Gr_+(k,n))=\omega' \wedge \pi^*( \mu_{\Gr(k,k+m)})$. Then 
$
(\omega' -\omega)\wedge \pi^* (\mu_{\Gr(k,k+m)})=0,
$
and so $\left.(\omega' -\omega)\right|_{\pi^{-1}(Y)}=0$. More precisely, the map $\pi:\Gr(k,n)\to \Gr(k,n+m)$ defines a decreasing filtration on the space of differential forms:
$$
J_0\supset J_1 \supset \ldots \supset J_{km},
$$
where $J_i$ is a differential ideal generated by the pullback of the $i$-forms on $\Gr(k,k+m)$. It is easy to see that for any form $\eta \in J_r$ with $r\geq 1$, the restriction of $\eta$ to the fiber vanishes, i.e.~we have $\eta|_{\pi^{-1}(Y)}=0$. Also, the condition that $(\omega' -\omega)\wedge \pi^*( \mu_{\Gr(k,k+m)})=0$ guarantees that $(\omega' -\omega) \in J_1$.
\end{remark}

In this section, instead of working with differential forms on $\Gr(k,n)$ we will equivalently work with \emph{basic forms} on the corresponding space of matrices $\Mat(k,n)$. We quickly recall the relevant definitions and refer to \cite{arnold} and \cite[Appendix C]{positivegeom} for more details. 

\begin{definition} \label{basic}
{\rm Let $p:\Mat(k,n) \to \Gr(k,n)$ be the natural projection, which is a principal $\GL_k$-bundle.
We call a form $\eta$ on $\Mat(k,n)$ \emph{basic} if it is a pullback of a form on $\Gr(k,n)$.
}\end{definition}
\noindent Moreover, it is well-known that a form $\eta$ is basic if and only if the following two conditions hold:
\begin{itemize}
    \item[(i)] for any vector field $v$ coming from differentiation of $\GL_k$-action, we have $\iota_v \eta =0$, where $\iota_v \eta$ is the \emph{contraction} of a vector field $v$ with a differential form $\eta$;
    \item[(ii)] the form $\eta$ is $\GL_k$-invariant, or equivalently the Lie derivative $\mathscr{L}_v(\eta)=0$ for any $v$ as above.
\end{itemize}
We will also work with a more general class of forms which are covariant under $\GL_k$-action. 
\begin{definition}\label{page:covariant}
{\rm Let $\eta$ be a form on $\Mat(k,n)$.
\begin{itemize}
    \item The form $\eta$ is called \emph{covariant} of degree $r$ if it scales by $\det^{r}(G)$ 
under the action of $G$ in $\GL_k$.
\item A covariant form $\eta$ is called a \emph{basic covariant form} of degree $r$ if it satisfies 
$\mathrm{(i)}$ and \eqref{eq:covariant}:
\begin{equation}\label{eq:covariant}
   \mathscr{L}_v(\eta)=r \, \eta,
\end{equation}
for any vector field $v$ coming from differentiation of $\GL_k$-action.
\end{itemize}
}\end{definition}
The basic covariant forms define twisted forms on $\Gr(k,n)$. More precisely, let $U\to \Gr(k,n)$ be the universal bundle. Then a covariant form on $\Mat(k,n)$ of degree $r$ defines a form on $\Gr(k,n)$ with coefficients in the line bundle $(\det\,U)^r:=\left(\bigwedge^k(U)\right)^r.$

\begin{Lemma}\label{lem:contrations}
Let $v$ be the wedge product of $k^2$ vector fields generating a $\GL_k$-orbit in ${\rm Mat}(k,n)$ as follows:
\begin{equation}\label{eq:fieldGL}
      v := \bigwedge_{\alpha,\beta=1}^k \sum_{i=1}^n c_{\alpha i} \frac{\partial}{\partial c_{\beta i}}.
\end{equation}
Then we have:
\[
{\rm (i)}\   \iota_v \mathbf{\Omega}(\Mat_+(k,n))=\mathbf{\Omega}(\Gr_+(k,n))\quad\text{and}\quad
 {\rm (ii)}\   \iota_v {\rm d}^{k(k+m)} Y^*= \mu^*_{\Gr(k,k+m)},
\]
where ${\rm d}(Y^A_\alpha)^*$ are the pullbacks of ${\rm d}(Y^A_\alpha)$ to $\Gr(k,n)$ from the amplituhedron via the map $Y=C \cdot Z$.
\end{Lemma}
\begin{proof}
We first rewrite $v$ as: 
\begin{equation*}
    v=\bigwedge_{\alpha=1}^k \sum_{I \in \binom{[n]}{k}} p_{I}(C) \bigwedge_{i \in I} \frac{\partial}{\partial c_{\alpha i}}.
\end{equation*}
Then (i) follows immediately from the following relation:
\begin{equation*}
  \iota_v \bigwedge_{\alpha=1}^k \bigwedge_{j=1}^n \mbox{d}c_{\alpha j}= \bigwedge_{\alpha=1}^k \sum_{I \in \binom{[n]}{k}} \epsilon_{I \bar{I}} \, p_{I}(C) \bigwedge_{j \in \bar{I}} \mbox{d} c_{\alpha j}=\bigwedge_{\alpha=1}^k \langle C \mbox{d}^{n-k} C_{\alpha} \rangle.
\end{equation*}
Now, let us consider:
\begin{eqnarray*}
  \mbox{d}^{k(m+k)}Y^*=\bigwedge_{\alpha=1}^k \sum_{L \in \binom{[n]}{m+k}} p_{L}(Z) \bigwedge_{l \in L} \mbox{d} c_{\alpha l}=\bigwedge_{\alpha=1}^k \langle Z^\perp \mbox{d}^{m+k} C_{\alpha}\rangle.
\end{eqnarray*}
We also have:
\begin{eqnarray*}
   \mu^*_{\Gr(k,k+m)}(C) &=& \bigwedge_{\alpha=1}^k \epsilon_{A_1 \cdots A_{k+m}} c_{1 i_1} \cdots c_{k i_k} Z_{i_1}^{A_1} \cdots  Z_{i_k}^{A_k} \wedge \mbox{d} c_{\alpha i_{k+1}} Z_{i_{k+1}}^{A_{k+1}} \wedge\cdots \wedge \mbox{d} c_{\alpha i_{k+m}} Z_{i_{k+m}}^{A_{k+m}}\\
   &=& \bigwedge_{\alpha=1}^k p_{i_1 \ldots i_k}(C) p_{i_1 \ldots i_{k+m}}(Z) \, \mbox{d} c_{\alpha i_{k+1}}\wedge \cdots \wedge \mbox{d} c_{\alpha i_{k+m}}=\bigwedge_{\alpha=1}^k \langle C \, Z^\perp \, \mbox{d}^{m}C_{\alpha}\rangle,
\end{eqnarray*}
where the sums are performed over all indices $i_a$ from $1$ to $n$ and over all indices $A_a$ from $1$ to $m+k$, with $a \in [k]$.
This implies the following relation
\begin{equation*}
\iota_v \, \mbox{d}^{k(m+k)}Y^* = \bigwedge_{\alpha=1}^k \sum_{L \in \binom{[n]}{m+k}} \sum_{M \in \binom{[n]}{m}} \sum_{I \in \binom{[n]}{k}}\epsilon_{\bar{L}L} \, \epsilon_{\bar{L} I M} \, p_{L}(Z) \, p_I(C) \bigwedge_{j \in M} \mbox{d} c_{\alpha j}=\bigwedge_{\alpha=1}^k \langle C \, Z^\perp \, \mbox{d}^{m}C_{\alpha}\rangle,
\end{equation*}
which completes the proof of (ii).
\end{proof}

\begin{Proposition}\label{prop:formfiberchar} 
Let $\omega$ be a $k(n-m-k)$ form on $\mathbb{R}^{kn}$ such that:
\begin{itemize}
    \item[{\rm(i)}] for any vector field $v$ tangent to $\GL_k$-orbits $\iota_v \omega=0$,
    \item[{\rm (ii)}] $\mathbf{\Omega}({\Mat}_+(k,n))=\omega \wedge {\rm d}^{k(m+k)} Y^*$.
\end{itemize}
Then the restriction of $\omega$ on the fiber $\pi^{-1}(Y)$ of $\mathcal{A}_{n,k,m}$ is the fiber volume form $\omega_{n,k,m}$.
\end{Proposition}
\begin{proof}
Consider $v$ from \eqref{eq:fieldGL}. Then, by Lemma~\ref{lem:contrations}, property (ii) and the fact that $\iota_v \omega=0$, we have:
\begin{equation*}
 \mathbf{\Omega}(\Gr_+(k,n))  =\iota_v \mathbf{\Omega}(\Mat_+(k,n))=\omega \wedge \iota_v \mbox{d}^{k(k+m)} Y^*= \omega \wedge \mu^*_{\Gr(k,k+m)}.
\end{equation*}
Moreover, thanks to (i) and \eqref{eq:covariant},
$\omega$ is a basic covariant form on $\mbox{Mat}(k,n)$ which completes the proof.
\end{proof}

\begin{Theorem}\label{prop:omega}
The volume form $\omega_{n,k,m}$ on the fiber $\pi^{-1}(Y)$ of $\mathcal{A}_{n,k,m}$ is
 \begin{equation}\label{eq:omega}
  \omega_{n,k,m}= p_{J}(\lambda)^{-k} \frac{ \mu_{\Gr(k,n-m)}(\lambda)}{\prod_{i=1}^n \det \left(\lambda|(A^{I_i})^\perp \right)},
\end{equation}
where the fiber $\pi^{-1}(Y)$ is parametrized as in \eqref{eq:Propii} and $J \in \binom{[n]}{k}$ denotes the indices of the rows of the matrix $A$ in \eqref{eq:Propii} that span $Z^{-1}(Y) \backslash {\rm ker}\, Z$. 
\end{Theorem}
\begin{proof}
Referring to Proposition~\ref{prop:formfiberchar},
we make the following ansatz for $\omega$:
\begin{equation}\label{eq:wtilde_decomposition}
    \omega=\frac{\widetilde{\omega}}{\prod_{i=1}^n p_{I_i}(C)}, \quad \text{with}\quad \widetilde{\omega}=f(C) \bigwedge_{\alpha=1}^k \langle C \, A^\perp \, \mbox{d}^{n-m-k} C_{\alpha} \rangle,
\end{equation}
where $A^\perp$ is the $m$-dim complement of the $(n-m)$-dim subspace in $\mathbb{R}^n$ spanned by the rows of $A$, and $f(C)$ is a function of $C$ of weight $-k$ which has to be fixed.
Let us prove that 
\begin{equation*}
    \iota_{v_{\alpha \beta}} \widetilde{\omega}=0 \quad\text{and}\quad  v_{\alpha \beta} := \sum_{i=1}^n c_{\alpha i} \frac{\partial}{\partial c_{\beta i}}\quad\text{for}\quad  \alpha,\beta=1, \ldots, k, 
\end{equation*}
where $v_{\alpha \beta}$ are the generators of $\GL_k$-orbits. We have:
\begin{equation*}
 \iota_{v_{\alpha \beta}} \langle C \, A^\perp \, \mbox{d}^{n-m-k} C_{\beta} \rangle = \sum_{i=1}^n \epsilon_{i_1 \cdots i_n} c_{1 i_1} \cdots   c_{k i_k} c_{\beta i} (A^{\perp})^{i_{k+1}}\cdots (A^{\perp})^{i_{k+m}} \mbox{d} c_{\beta i_{k+m+1}} \wedge \cdots  \wedge \widehat{\mbox{d} c_{\beta i}} \wedge \cdots \wedge \mbox{d} c_{\beta i_n},
\end{equation*}
which vanishes by anti-symmetrization of $c_{\beta i_{\beta}}$ and $c_{\beta i}$ via the Levi-Civita symbol.
Then the vanishing of $\iota_{v_{\alpha \beta}} \omega$ follows immediately. 
We now fix $f(C)$ by demanding that 
    $\widetilde{\omega} \wedge \mbox{d}^{k(m+k)}Y^*=\mbox{d}^{kn} C.$
Then:
\begin{eqnarray*}
 f(C)^{-1} \widetilde{\omega} \wedge \mbox{d}^{k(m+k)}Y^* &=& \bigwedge_{\alpha=1}^k \epsilon_{i_1 \cdots i_n} c_{1 i_1} \ldots   c_{k i_k} p_{i_{k+1} \cdots i_{k+m}}(A^\perp) p_{L}(Z) \, \mbox{d} c_{\alpha i_{k+m+1}} \wedge \cdots \wedge \mbox{d} c_{\alpha i_n} \bigwedge_{l \in L} \mbox{d} c_{\alpha l}
 \\
  &=&  \bigwedge_{\alpha=1}^k \epsilon_{i_1 \cdots i_n} \epsilon_{i_{k+m+1} \cdots i_n L} \, p_{i_1 \cdots i_k}(C) \, p_{i_{k+1} \cdots i_{k+m}}(A^\perp) \, p_{L}(Z) \, \bigwedge_{j=1}^n \mbox{d} c_{\alpha j}
  \\
  &=& \bigwedge_{\alpha=1}^k \langle C \, A^\perp Z^\perp \rangle\, \bigwedge_{j=1}^n \mbox{d} c_{\alpha j}=\langle C \, A^\perp Z^\perp \rangle^k\, \mbox{d}^{kn} C.
\end{eqnarray*}
Therefore: $f(C)= \langle C \, A^\perp Z^\perp \rangle^{-k}.$
Now, by restricting $\widetilde{\omega}$ on the fiber $\pi^{-1}(Y)$ we have:
\begin{eqnarray}\label{eq:w_tilde}
  \widetilde{\omega}|_{\pi^{-1}(Y)}=\langle (\lambda \cdot A) \, A^\perp Z^\perp \rangle^{-k} \,  \langle A \, A^\perp \rangle^{k} \bigwedge_{\alpha=1}^k \langle \lambda_{\alpha} \, \mbox{d}^{n-m-k} \lambda \rangle=p_{J}(\lambda)^{-k} \, \bigwedge_{\alpha=1}^k \langle \lambda_{\alpha} \, \mbox{d}^{n-m-k} \lambda \rangle,
\end{eqnarray}
where $J$ is the index set of rows of $A$ spanning $Z^{-1}(Y) \backslash \mbox{ker}\, Z$. Now \eqref{eq:w_tilde} and \eqref{eq:wtilde_decomposition} gives \eqref{eq:omega}, as desired.
\end{proof}

\begin{remark}
Note that it is easy to understand the factor $p_J(\lambda)$ in \eqref{eq:omega} as carrying the $\GL_k$-scaling which $Y$ comes with. In particular, for $\ell=n-m-k$ we choose
 $J=\lbrace \ell+1,\ldots,n-m \rbrace$, that is:
\begin{equation*}
    A=
\left(
\begin{array}{c}
 (Z^\perp)_{\ell \times n}\\
\hline
B_{k \times n}
\end{array}
\right),
\end{equation*}
where $B_{k \times n}$ is a matrix whose rows span the subspace $Z^{-1}(Y) \backslash \mbox{ker}\, Z$.
Then:
\begin{equation*}
 \lambda \cdot A \cdot Z=  \lambda \cdot  \left(
\begin{array}{c}
 \mathbb{O}_{\ell \times (m+k)}\\
\hline
(B \cdot Z)_{k \times (m+k)}
\end{array}
\right) = \lambda_{J} \cdot (B \cdot Z)_{k \times (m+k)},
\end{equation*}
where $\lambda_J$ is the submatrix obtained from $\lambda$ by considering its columns with indices in $J$.
Clearly, since the row-span of $B$ gives $Z^{-1}(Y) \backslash \mbox{ker}\, Z$, the row-span of $(B \cdot Z)_{k \times (m+k)}$ gives the row-span of $Y_{k \times (m+k)}$. Moreover, multiplication by $\lambda_J$ is just a $\GL_k$-transformation for this subspace.
\end{remark}
\begin{remark}\label{rk:formfiberdual}
We observe that, by duality on Grassmannians, we can equivalently describe the volume form $\omega_{n,k,m}$ on the fiber $\pi^{-1}(Y)$ of the amplituhedron $\mathcal{A}_{n,k,m}$ as a volume form on $\Gr(\ell,n-m)$ for $\ell=n-m-k$, instead of $\Gr(k,n-m)$. Given $\lambda \in \Gr(k,n-m)$, we denote $\bar{\lambda} \in \Gr(\ell,n-m)$ as its dual. Then the fiber volume form in \eqref{eq:omega} can be expressed as:
\begin{equation*}
    \omega_{n,k,m}= p_{\bar{J}}(\bar{\lambda})^{-k} \frac{ \bigwedge_{\bar{\alpha}=1}^{\ell} \langle \bar{\lambda} \, \mbox{d}^{k} \bar{\lambda}_{\bar{\alpha}} \rangle}{\prod_{i=1}^n \det \left(\bar{\lambda}|A^{I_i} \right)}.
\end{equation*}
We recall that $p_I(\lambda)=\epsilon_{I \bar{I}} p_{\bar{I}}(\bar{\lambda})$ and $\bigwedge_{\alpha=1}^{k} \langle \lambda \, \mbox{d}^{\ell} \lambda_{\alpha} \rangle= \bigwedge_{\bar{\alpha}=1}^{\ell} \langle \bar{\lambda} \, \mbox{d}^{k} \bar{\lambda}_{\bar{\alpha}} \rangle$. Moreover, we have:
\begin{equation*}
    \mbox{det} (\lambda | {(A^{I_i})}^\perp)=\sum_{J \in \binom{[n]}{k}} p_J(\lambda) A^{I_i}_J=\sum_{J \in \binom{[n]}{k}} \epsilon_{J \bar{J}} p_{\bar{J}}(\bar{\lambda}) A^{I_i}_J=\det \left(\bar{\lambda}|A^{I_i} \right).
\end{equation*}
\end{remark}

\begin{remark}\label{rk:oldpaper}
We would like to highlight that the definition of the \emph{fiber volume form} given in \eqref{eq:form} is natural from a geometric perspective, once we fix a volume form on the base $\mu^*_{\Gr(k,k+m)}$.
In \cite{ferro2018amplituhedron}, a differential form denoted as $\omega^{(m)}_{n,k}$ was formally defined as: 
\begin{equation}\label{olddef}
    \omega^{(m)}_{n,k}=\int \frac{\mbox{d}^{k \times n} C}{\prod_{i=1}^n p_{I_i}(C)} \delta^{k(m+k)}(Y-C \cdot Z).
\end{equation}
This can be rigorously written as the residue (see Definition \ref{def:residue}) of $\mathbf{\Omega}(\mbox{Mat}_+(k,n))$ as:
\begin{eqnarray*}
  \omega^{(m)}_{n,k}=\mbox{Res}_{Y=C \cdot Z} \, \frac{\mathbf{\Omega}(\mbox{Mat}_+(k,n))}{Y-C \cdot Z},
\end{eqnarray*}
where $Y-C\cdot Z$ denotes the function given by the product of all $k(m+k)$ entries of the matrix $Y-C\cdot Z$.
If we have a $k(n-m-k)$ form $\omega$ on $\mathbb{R}^{kn}$ as in Proposition~\ref{prop:formfiberchar}, then we have:
\begin{equation*}
    \mbox{Res}_{Y=C \cdot Z} \, \frac{\mathbf{\Omega}(\mbox{Mat}_+(k,n))}{Y-C \cdot Z}= \mbox{Res}_{Y=C \cdot Z} \, \left( \omega \wedge \frac{\mbox{d}^{k(m+k)}(Y-C \cdot Z)}{Y-C \cdot Z} \right)=\omega|_{Y=C \cdot Z}=\omega_{n,k,m}.
\end{equation*}
Here we used the definition of Poincar\'e residue and the fact that
$$
\mbox{d}^{k(m+k)}Y^*=\mbox{d}^{k(m+k)}(C \cdot Z)=\mbox{d}^{k(m+k)}(Y-C \cdot Z),
$$ 
since $Y$ is fixed.
Therefore we can see that, thanks to Proposition~\ref{prop:formfiberchar}, the form in \eqref{olddef} coincides with the fiber volume form $\omega_{n,k,m}$ which was defined in an invariant way in Definition~\ref{def:omegafiber}. 
\end{remark}

\begin{remark}
In \cite{Benincasa:2020uph}, the notion of \emph{covariant restriction} of a \emph{covariant form} on a subspace was introduced in the context of positive geometries with non-logarithmic singularities. In that language, $\omega$ in Definition~\ref{def:omegafiber}  is the covariant restriction of $\mathbf{\Omega}(\mbox{Gr}_+(k,n))$ on the fiber $\pi^{-1}(Y)$. More precisely, using Theorem~\ref{prop:omega}, $\omega$ is a covariant form of degree $-(m+k)$ which is the covariant restriction of $\mathbf{\Omega}(\mbox{Mat}_+(k,n))$ on the $k \ell$-dim subspace $\{C \in \mbox{Mat}(k,n): C \cdot Z =Y \}$.
\end{remark}

In the following, we explicitly compute the fiber volume forms for polytopes and their conjugates, and show that they only have simple poles on some hyperplane arrangement. 

\begin{Corollary}
\label{cor:volumefibers}
Let $\ell=n-m-1$ and denote $\omega_{n,1,m}$ and $\omega_{n,\ell,m}$ for the volume forms of the fibers in the cases of cyclic polytopes and conjugate to polytopes. Then $\omega_{n,1,m}$ and $\omega_{n,\ell,m}$ are volume forms with simple poles on affine hyperplanes arrangements. In particular, there is a choice of coordinates $\lambda_1,\ldots,\lambda_{\ell}$ and $\bar{\lambda}_1,\ldots,\bar{\lambda}_{\ell}$ on the respective fibers such that:
\begin{equation}\label{eq:volumefiberspolconjpol}
 \omega_{n,1,m}=\frac{ {\rm d} \lambda_1 \wedge \cdots \wedge {\rm d} \lambda_{\ell}}{\prod_{i=1}^n \left( \lambda \cdot A^i \right)} \quad\text{and}\quad   \omega_{n,\ell,m}=\frac{{\rm d} \bar{\lambda}_1 \wedge \cdots \wedge {\rm d} \bar{\lambda}_{\ell}}{\prod_{i=1}^n  \bar{\lambda} \cdot {(A^{I_i})}^\perp},
\end{equation}
where $\{A^i\}_{i \in [n]}$ and $\{{(A^{I_i})}^\perp \}_{i \in [n]}$ are the normal vectors of the affine hyperplanes in the respective cases. 
\end{Corollary}
\begin{proof}
For $k=1$, the differential form \eqref{eq:omega} is just the following form on $\mathbb{P}^{n-m-1}$:
\begin{equation*}
    \omega_{n,1,m}=\frac{\langle \lambda \mbox{d}^{n-m-1} \lambda \rangle}{\lambda_j \, \prod_{i=1}^n \left( \lambda \cdot A^i \right)},
\end{equation*}
where the $j^{\rm th}$ row of $A$ spans $Z^{-1}(Y) \backslash {\rm ker}\, Z$. Without loss of generality, we fix $j=n-m$, compatible with the choice in \eqref{matrixA}.
Choosing the chart $\lambda=(\lambda_1,\ldots, \lambda_{n-m-1},1)$, this simply becomes:
\begin{equation*}\label{eq:formpol}
     \omega_{n,1,m}=\frac{ \mbox{d} \lambda_1 \wedge \cdots \wedge \mbox{d} \lambda_{n-m-1}}{\prod_{i=1}^n \left( \lambda \cdot A^i \right)}.
\end{equation*}
As in Remark~\ref{rk:formfiberdual}, for the conjugate to polytopes it is better to express the fiber volume form in the dual space. Doing so, we have that $\Gr(n-m-\ell,n-m)=\Gr(1,n-m) \cong \mathbb{P}^{n-m-1}$. Then:
\begin{equation*}
\omega_{n,\ell,m}=\frac{  \langle \bar{\lambda} \, {\rm d}^{\ell} \bar{\lambda} \rangle}{\bar{\lambda}_{\bar{J}}^{\ell} \, \prod_{i=1}^n \bar{\lambda} \cdot {(A^{I_i})}^\perp},
\end{equation*}
where $\bar{\lambda} \in \mathbb{P}^{\ell}$ and $\bar{J}$ is the only index in $[n-m] \backslash J$.  
Moreover, $A^{I_i}$ is an $\ell$-dim space in $\mathbb{R}^{n-m}$, and so $(A^{I_i})^\perp$ is $1$-dim.
Without loss of generality, we fix $J=\lbrace 2, \ldots, n-m \rbrace$, compatible with the choice in \eqref{matrixA}. So $\bar{J}=\{1\}$.
Choosing the chart $\bar{\lambda}=(1,\bar{\lambda}_2,\ldots, \bar{\lambda}_{n-m})$, the fiber volume form simply becomes:
\begin{equation*}\label{formconjpol}
    \omega_{n,\ell,m}=\frac{\mbox{d} \bar{\lambda}_2 \wedge \cdots \wedge \mbox{d} \bar{\lambda}_{n-m}}{\prod_{i=1}^n  \bar{\lambda} \cdot {(A^{I_i})}^\perp}.
\end{equation*}
Now, the assertions in \eqref{eq:volumefiberspolconjpol} follow immediately by relabeling the coordinates on the fibers.  
\end{proof}

\subsection{Triangulations and fiber volume forms} \label{subsec:fibervolformtriang}
Here we show that the canonical form of the amplituhedron can be computed adopting an alternative approach based on its fibers. 
In particular, we use the residues of the fiber volume form Definition~\ref{def:omegafiber} to compute the canonical function of the amplituhedron; See Theorem~\ref{thm:resfiber}.

\medskip
We first recall further nice properties of positroid cells as follows:
\begin{Conjecture}[{\cite[Conjecture 8.4]{galashin2018parity}}] \label{conjtriangles}
Let $S_M$ be a $km$-dim positroid cell of $\Gr_+(k,n)$. Then the following are equivalent:
\begin{itemize}
\item[\rm{(i)}] $\pi_+: S_M \rightarrow \pi_+(S_M)$, i.e.~the restriction of $\pi_+$ to $S_M$ is injective.

\item[\rm{(ii)}] $\pi: \overline{S_M} \dashrightarrow
\Gr(k,\mathbb{C}^{k+m})$ is a birational map, with $\dim(\overline{S_M})=\dim(\pi(\overline{S_M}))=mk$.
\end{itemize}
In each case, $\pi_+: S_M \rightarrow \pi_+(S_M)$ is a diffeomorphism. 
\end{Conjecture}
This conjecture holds true for $k=1$ (polytopes) and $k=n-m-1$ (conjugate to polytopes).
Note that given a triangulation of $\mathcal{A}_{n,k,m}$, by definition, the property (i) above holds. The property (ii) is used to relate the canonical form of $\mathcal{A}_{n,k,m}$ to its triangulations and the fiber volume form $\omega_{n,k,m}$.

\begin{Theorem} \label{thm:resfiber}
Let $\omega_{n,k,m}(Y)$ be the volume form on the fiber $\pi^{-1}(Y)$ for $Y \in \mathcal{A}_{n,k,m}$. Then:

\begin{itemize}
    \item[{\rm(i)}] Every generalized triangle $S_M$ of $\mathcal{A}_{n,k,m}$ corresponds to a pole $q_M$ of the fiber volume form $\omega_{n,k,m}$, where $q_M=\pi^{-1}(Y) \cap \overline{S_M}$.

    \item[{\rm (ii)}] The residue of the fiber volume form $\omega_{n,k,m}$ around the pole $q_M$ computes the canonical function of the corresponding generalized triangle $S_M$ as:
    \begin{equation*}
\Omega(S_M)(Y) = {\rm Res}_{q_M} \omega_{n,k,m}(Y).
    \end{equation*}
    In particular, given a triangulation $\mathcal{C}=\lbrace S_M \rbrace$ of $\mathcal{A}_{n,k,m}$, the canonical function is given by:
\begin{equation*}
    \Omega(\mathcal{A}_{n,k,m})(Y) = \sum_{S_M\in \mathcal{C}} {\rm Res}_{q_M} \omega_{n,k,m}(Y).
\end{equation*}
\end{itemize}
\end{Theorem}

\begin{proof}
(i) We recall that by Definition~\ref{def:triangulations} for every generalized triangle the map $\pi^+$ is injective on $S_M$. Moreover, by the property (ii) in Conjecture~\ref{conjtriangles}, for every generalized triangle $S_M$ the map $\pi: \overline{S_M} \dashrightarrow  \Gr(k,\mathbb{C}^{k+m})$ is birational. Hence, there exists a Zariski open subset $U$ of $\Gr(k,\mathbb{C}^{k+m})$ such that for any $Y\in U$ there exists a unique point $q_M$ with $\pi^{-1}(Y)\cap \overline{S_M}=q_M$. Hence, the volume form $\omega_{n,k,m}$ on the fiber $\pi^{-1}(Y)$ has poles at the points $q_M \in \Gr(k,\mathbb{C}^{k+m})$ which are the intersection between the full fiber $\pi^{-1}(Y)$ and the Zariski closure $\overline{S_M}$ of any generalized triangle $S_M$.

\medskip

(ii) Consider a triangulation $\mathcal{C}=\lbrace S_M \rbrace$ of $\mathcal{A}_{n,k,m}$. Then by~\eqref{eq:AHformTriang} we have:
\begin{equation*}
    \mathbf{\Omega}(\mathcal{A}_{n,k,m})(Y) = \sum_{S_M\in \mathcal{C}} {((\pi|_{S_M})^{-1})}^*\mathbf{\Omega}(S_M)(Y)= \sum_{S_M\in \mathcal{C}} {((\pi|_{S_M})^{-1})}^*\mbox{Res}_{S_M}
    \mathbf{\Omega}(\Gr_+(k,n))(Y).
\end{equation*}    
But since $\mathbf{\Omega}(\Gr_+(k,n))=\omega\wedge  \pi^*(\mu_{\Gr(k,k+m)})$ and $\mu^*_{\Gr(k,k+m)}|_{\pi^{-1}(Y)}=0$ we have:
\[
 \mbox{Res}_{S_M}\mathbf{\Omega}(\Gr_+(k,n)) = (\mbox{Res}_{S_M}\omega) \wedge \pi^*(\mu_{\Gr(k,k+m)}).
\]
\vspace{-2mm}Therefore, 
\begin{eqnarray*}
\sum_{S_M\in \mathcal{C}} {((\pi|_{S_M})^{-1})}^*\mbox{Res}_{S_M}\mathbf{\Omega}(\Gr_+(k,n))(Y) &=& \sum_{S_M\in \mathcal{C}} \mbox{Res}_{q_M} \omega_{n,k,m}(Y) \, \mu_{\Gr(k,k+m)}(Y)\\ &=&\Omega(\mathcal{A}_{n,k,m})(Y) \, \mu_{\Gr(k,k+m)}(Y),
\end{eqnarray*}
where  $q_M=\overline{S_M} \cap \pi^{-1}(Y)$. Similar to Remark~\ref{rk:orientation2}, the sign of $\mbox{Res}_{q_M} \omega_{n,k,m}(Y)$ above is taken such that 
\begin{equation*}
\mbox{Res}_{q_M} \omega_{n,k,m}(Y) > 0 \text{ for any } Y\in \pi_+(S_M). \tag*{\qedhere}
\end{equation*}
\end{proof}

\begin{example}[Amplituhedron with non-linear fibers]\label{exam:non-linear}
We consider the first non-trivial non-linear example, i.e.~$(n,k,m)=(6,2,2)$. Let $J=\{3,4\}$ and choose a chart where $\lambda_{34}=1$. Then \eqref{eq:omega} reads as: 
\begin{equation}\label{ex:non-linear}
  \omega_{6,2,2}=\frac{\mbox{d}\lambda_{11} \wedge \mbox{d} \lambda_{12} \wedge \mbox{d} \lambda_{21} \wedge \mbox{d} \lambda_{22} }{\prod_{i=1}^6 \det \left(\lambda|(A^{I_i})^\perp \right)}  \quad\text{for}\quad \lambda= \left(
\begin{array}{cccc}
\lambda_{11}  & \lambda_{12} & 1 & 0\\
\lambda_{21}  & \lambda_{22} & 0 & 1
\end{array}
\right).
\end{equation}
We observe that the denominator vanishes on the locus of $6$ quadrics $\mathcal{Q}_i$ in $\mathbb{P}^4$ given by:
\begin{equation*}
    \Lambda \cdot M_i \cdot \Lambda^t=0\quad\text{for}\quad M_i = 
\begin{pmatrix}
  a & \vline & \beta \\
\hline
\beta^t & \vline &
  \begin{matrix}
  0 & 0 & 0 & b \\
    0 & 0 & -b & 0 \\
      0 & -b & 0 & 0 \\
        b & 0 & 0 & 0 \\
  \end{matrix}
\end{pmatrix} \quad\text{and}\quad \Lambda=(1 \, \lambda_{11} \,  \lambda_{12} \,  \lambda_{21} \,  \lambda_{22} \,),
\end{equation*}
where $a,b \in \mathbb{R}, \,\beta \in \mathbb{R}^4$, and $i \in [6]$. The generalized triangles of $\mathcal{A}_{6,2,2}$ correspond to $48$ intersection points of these quadrics. Considering all collections $I$ of $4$ quadrics, $30$ intersection points arise as: 
\begin{equation*}
    \bigcap_{i \in I} \mathcal{Q}_i \supseteq \lbrace q_I^{(1)}, q_I^{(2)}\rbrace.
\end{equation*}
The remaining $18$ points come from collections of $3$ quadrics such that their intersection is a $1$-dim curve with self-intersections. This happens for the following cases:
\begin{equation*}
    \mathcal{Q}_i \cap \mathcal{Q}_{i+1} \cap \mathcal{Q}_j \supseteq \lbrace q_{i,i+1,j} \rbrace, \quad i, j \in [6],\ i+1<j.
\end{equation*}
There are $18$ points of this type.
In order to better describe this phenomenon, consider as a toy model 
the surfaces defined as:
\begin{equation*}
    \mathcal{S}_1=\lbrace z=0 \rbrace, \, \mathcal{S}_2=\lbrace z+xy=0 \rbrace,\, \mathcal{S}_3=\lbrace x-1=0 \rbrace.
\end{equation*}
Then we have:
\vspace{-3mm}\begin{equation*}
     \mathcal{S}_1 \cap  \mathcal{S}_2= \lbrace x,z =0 \rbrace \cup \lbrace y,z =0 \rbrace
\end{equation*}
which has a self-intersection at the point $(x,y,z)=(0,0,0)$. Whereas $  \mathcal{S}_1,  \mathcal{S}_2, \mathcal{S}_3$ intersect at the point $(x,y,z)= (1,0,0)$. We can then associate both the points $(1,0,0)$ and $(0,0,0)$ to the configuration of surfaces $\lbrace \mathcal{S}_1, \mathcal{S}_2, \mathcal{S}_3 \rbrace$.
Therefore, all together we have $48$ intersection points $\lbrace q_I^{(1)}, q_I^{(2)}, q_{i,i+1,j} \rbrace$.

Contrary to the configuration of hyperplanes, there is no natural fan associated to a configuration of quadrics $\lbrace \mathcal{Q}_i \rbrace$.
From \cite{lukowski2020positive}, we know that there are $120$ finest dissections of $\mathcal{A}_{6,2,2}$ and exactly $48$ good triangulations. The latter can be obtained as cyclic rotations of $11$ of them. We list below one representative of each cyclic class of good triangulations in terms of the corresponding intersection points:
\begin{eqnarray*}
  \{q_{126},q_{1246}^{(1)},q_{1456}^{(1)},q_{1234}^{(2)},q_{1345}^{(2)},q_{345}\}, &\quad & \{q_{126},q_{1456}^{(2)},q_{3456}^{(1)},q_{1246}^{(2)},q_{1234}^{(1)},q_{2345}^{(2)}\},\\
  \{q_{126},q_{1246}^{(1)},q_{1456}^{(1)},q_{1234}^{(2)},q_{3456}^{(2)},q_{2345}^{(1)}\}, &\quad & \{q_{126}, q_{146}, q_{1346}^{(2)},q_{1346}^{(1)},q_{134},q_{345}\},\\
   \{ q_{1256}^{(2)},q_{1456}^{(1)},q_{1236}^{(1)},q_{1234}^{(2)},q_{3456}^{(2)},q_{2345}^{(1)}\},&\quad & \{q_{126}, q_{1246}^{(1)}, q_{1456}^{(1)},q_{3456}^{(2)},q_{2346}^{(1)},q_{234}\},\\
    \{ q_{126},q_{1245}^{(1)},q_{145},q_{124},q_{345},q_{1245}^{(2)}\},&\quad & \{q_{126}, q_{1456}^{(2)}, q_{3456}^{(1)},q_{1246}^{(2)},q_{2346}^{(2)},q_{234}\},\\
     \{q_{126},q_{146},q_{1346}^{(2)},q_{1346}^{(1)},q_{346},q_{234}\},&\quad & \{q_{126}, q_{456},q_{1246}^{(2)}, q_{2456}^{(2)},q_{2346}^{(2)},q_{234}\},\end{eqnarray*}
     \[
     \{q_{126}, q_{1246}^{(1)}, q_{456}, q_{2456}^{(1)},q_{2346}^{(1)},q_{234}\}.
     \]
\end{example}
\medskip

The description of good triangulations above gives rise to the following question: How to characterize $6$-element subsets of the $48$ intersection points giving rise to triangulations of $\mathcal{A}_{6,2,2}$? 
In particular, understanding the geometry of the configuration of quadrics $\lbrace \mathcal{Q}_i \rbrace$ is useful to identify such sets.

\section{Linear fibers and triangulations}\label{sec:linear}
Our goal in this section is to study the fiber geometry of amplituhedra. In particular for the amplituhedra conjugate to polytopes, we define the notion of amplituhedra fiber.  We will present other examples of linear fibers in \S\ref{sec:fiberposgeom}, in the context of Grassmann polytopes with linear fibers. 

\subsection{Secondary fan and Gale dual for conjugate to polytopes}\label{sec:fan_conjugate_polytopes}
We first study the amplituhedra conjugate to polytopes. We fix $\mathcal{A}:=\mathcal{A}_{n,k,m}$ where $m=n-k-1$. By Example~\ref{cor:conjfibers}, the positive fibers of $\mathcal{A}$ are bounded by some hyperplanes. 
More precisely, each full fiber contains a collection of hyperplanes which bound the corresponding positive fiber. Here, we will first define a family of fans $\mathcal{F}_\mathcal{A}(Y)$ for $Y\in \mathcal{A}$. Then we show that all such fans are combinatorially equivalent and play the role of the secondary fan for $\mathcal{A}$. To do this, we have to choose an affine chart in every full fiber (see the discussion after Definition~\ref{def:projpol}).
We recall our notation from Remark~\ref{page:AandZ}(ii) for
$Z=(\widetilde{z}_{(m+k)\times (m+k)}|z_{(m+k)\times 1})^T$ that:
\begin{equation*}\label{matrixAZperp}
Z^\perp:=((-\widetilde{z}^{-1} \cdot z )^T|1)\quad\text{and}\quad A=
\left(
\begin{array}{c|c}
-(\widetilde{z}^{-1} \cdot z)^T & 1 \\
\hline
Y \cdot \widetilde{z}^{-1} & \mathbb{0}_{k \times 1}
\end{array}
\right).
\end{equation*}
We will further take 
\begin{equation}\label{eq:affinelambda}
\lambda=\left(
\begin{array}{ccc}
x^T_{k \times 1}&\vline&\mathbb{I}_{k}
\end{array}
\right)\quad\text{with}\quad x=(x_1, \ldots, x_k)
\end{equation}
as our affine chart of the full fiber; See Remark~\ref{rem:chart}(ii) for a justification of this choice. Note that under the identification of $\Gr(k,k+1)$ with $(\P^{k+1})^*$ this affine chart is given by: \begin{equation}\label{eq:dualchart}
\bar{\lambda}=\begin{matrix}(1|-x_1&\cdots&-x_k).\end{matrix}
\end{equation}

\begin{definition}\label{def:sec_fan}
{\rm Let $Y\in\mathcal{A}$ and let $\mathcal{H}_i$ be the hyperplane in the affine chart given in \eqref{eq:affinelambda} defined by $p_{I_i}=0$ for the cyclic $k$-interval $I_i$. 
We define 
$\mathcal{F}_\mathcal{A}(Y)$ to be the chamber fan containing the normal rays (negatively directed w.r.t.~the functions $p_{I_i}$) of the hyperplanes $\{\mathcal{H}_1,\ldots,\mathcal{H}_n\}$. 
}\end{definition}

Given any cyclic  $k$-interval $I_i$ of $[n]$, we define 
\begin{equation}\label{eq:W}
W^{I_i}:=-((A^{I_i})^\perp_2,\ldots,(A^{I_i})^\perp_{k+1}) \in \mathbb{R}^k\quad\text{and}\quad W_0^{I_i}:=-(A^{I_i})^\perp_1,
\end{equation}
where $(A^{I_i})^\perp _\ell$ denotes the $\ell^{\rm th}$ coordinate of the orthogonal complement of $A^{I_i}$.

\medskip

In the affine chart given in \eqref{eq:affinelambda}, we have that:
\begin{Lemma}\label{lem:Wi}
The rays $W^{I_1},\ldots,W^{I_n}$ are the normal rays of the hyperplanes $\mathcal{H}_1,\ldots,\mathcal{H}_n$. In particular, the fan $\mathcal{F}_{\mathcal{A}}(Y)$ is the chamber fan of the rays $W^{I_1},\ldots,W^{I_n}$.
\end{Lemma}
\begin{proof}
The hyperplanes $\mathcal{H}_1,\ldots,\mathcal{H}_n$ are given by
\begin{equation*}
    p_{I_i}(V)=\bar{\lambda} \cdot (A^{I_i})^\perp= (A^{I_i})^\perp_1-\sum_{a=1}^k x_a (A^{I_i})^\perp_{a+1}, \quad I_i=\lbrace i,i+1,\ldots,i+k-1\rbrace.
\end{equation*}
Hence, $W^{I_1},\ldots,W^{I_n}$ are their normal directions, as desired.
\end{proof}

Assuming that the functions $\det(W^{I_{j_1}},\ldots,W^{I_{j_k}})$ have a specific sign on $\A$ (see Conjecture~\ref{con:main}), we will show that the combinatorial structure of $\mathcal{F}_\mathcal{A}(Y)$ does not depend on $Y$. More precisely, the cones of $\mathcal{F}_\mathcal{A}(Y)$ 
are in correspondence with the dissections of $\mathcal{A}$; See Theorem~\ref{triangamp}. In particular, in \S\ref{subsec:geometriccharact} we will give a geometric interpretation to the functions
$
\det(W^{I_{j_1}},\ldots,W^{I_{j_k}}).
$

 \begin{remark}\label{rem:chart}
 \begin{itemize}
      \item[(i)] 
Conjecture~\ref{con:main} would also imply that our choice of affine charts of fibers in \eqref{eq:affinelambda} is compatible, i.e.~the removed hyperplane is never parallel to any boundary hyperplane $\mathcal{H}_i$.
The key point is that if $\mathcal{H}_i$ becomes parallel to a facet, then the corresponding ray in the normal fan disappears in the affine chart. Hence, all the determinants containing its generator converge to 0.

\item[(ii)] We observe that a family of hyperplanes of type $\lbrace \bar{\lambda}_i =0 \rbrace$ with $i \not = 1$ will be parallel to a facet of the fiber in $\mathbb{R}^k$ which gives a non-generic fan. Indeed, any choice other than \eqref{eq:affinelambda} would give rays with $(A^I)^\perp_1$ as one of their coordinates (w.l.o.g.~we may assume that this is $(A^I)^\perp_1$). However,
\begin{equation} \label{zerocone}
(A^I)^\perp_1=A^I_{\{2,\ldots, k+1\}}=0, \quad \mbox{if } k+1 \in I
\end{equation}
since the last $k$ entries of the $(k+1)^{\rm th}$ column of  $A$ are all zero. Therefore, considering all cyclic $k$-intervals $I_i=\lbrace i,i+1,\ldots,i+k-1\rbrace$ containing $k+1$, the corresponding cone is not maximal, i.e.
\begin{equation*}
    \det \left\lbrace (A^{I_i})^\perp \right\rbrace_{I_i \ni k+1}=0
\end{equation*}
since the matrix in the determinant has the first column zero, by virtue of \eqref{zerocone}. \end{itemize}
 \end{remark}

\subsection{Characterization of the rays of the secondary fan} \label{subsec:geometriccharact}
To prove the combinatorial equivalence of the fans $\mathcal{F}_\mathcal{A}(Y)$, we will first define a collection of functions on general amplituhedra (not necessarily conjugate to polytopes).

\begin{definition}\label{def:functions}
{\rm Let $P_1,\ldots,P_k$ be a collection of $(m+1)$-dim subspaces in $\mathbb{R}^{m+k}$ represented by matrices denoted with the same letters. 
We define 
 the following covariant function on 
 $\Gr(k,m+k)$:
\begin{equation}\label{eq:functions}
    \langle Y P_1 \cap \cdots \cap P_k \rangle:= \det\big( \langle Y^\perp_\alpha P_\beta \rangle \big)_{\genfrac{}{}{0pt}{2}{1\leq \alpha\leq k}{1\leq \beta\leq k}},
\end{equation}
where $\langle Y^\perp_\alpha P_\beta \rangle:=\sum_{1\leq\alpha_1<\cdots<\alpha_{k-1}\leq k} \epsilon_{\alpha \alpha_1\cdots\alpha_{k-1}} \langle Y_{\alpha_1} \cdots Y_{\alpha_{k-1}} P_\beta \rangle$ for $Y\in \Gr(k,m+k)$. Here $\langle \cdots \rangle$ denotes the determinant of the matrix obtained by stacking all the columns (or submatrices) in the dots.
}\end{definition}

The following proposition gives a more geometric characterization of the functions $\langle Y P_1 \cap \cdots \cap P_k \rangle$ 
in \eqref{eq:functions} for a generic $Y$ in $\Gr(k,m+k)$. 
\begin{Proposition}\label{propvect}
There exist vectors $v_{P_\alpha} \in Y\cap P_{\alpha}$ with $v_{P_\alpha}=\sum_{\beta=1}^k v_{P_\alpha}^\beta Y_\beta$ such that
$$
\langle Y P_1 \cap \cdots \cap P_k \rangle = \det{ \big( v_{P_\alpha}^\beta \big)}_{\genfrac{}{}{0pt}{2}{1\leq \alpha\leq k}{1\leq\beta\leq k}}.
$$ 
In particular, the function $\langle Y P_1 \cap \cdots \cap P_k \rangle$ vanishes if and only if $v_{P_1}, \ldots, v_{P_k}$ are linearly dependent.
\end{Proposition}
\begin{proof}
Suppose that $P_\alpha$ is spanned by the vectors $P_{\alpha 1}\ldots P_{\alpha m+1}$. Note that $\dim (P_\alpha)= m+1$ and $\dim (Y) =k$, hence the intersection $Y\cap P_{\alpha}$ is of dimension one for a generic $Y$.
Let $v_{P_{\alpha}} \in \mathbb{R}^{m+k}$ be a non-zero vector in the intersection $Y\cap P_{\alpha}$. 
One can construct such vector for $P_\alpha$ in the following way:
\begin{equation}\label{intersection}
v_{P_\alpha} = \sum_{\beta=1}^k (-1)^{\beta+1} Y_\beta \langle Y_1\cdots \widehat{Y}_\beta\cdots Y_k P_{\alpha 1}\cdots P_{\alpha {m+1}} \rangle
\end{equation}
\begin{equation}\label{intersection2}
\quad\quad\quad\quad=\sum_{j=1}^{m+1} (-1)^{j+1+k} P_{\alpha j} \langle Y_1\cdots Y_k P_{\alpha 1}\cdots\widehat{P}_{\alpha j}\cdots P_{\alpha {m+1}}\rangle.
\end{equation}
The equality above holds since for any collection $w_1,\ldots, w_{d+1}$ of $d+1$ vectors in $\R^d$ we have that
$$
\sum_{i=1}^{d+1} (-1)^{i+1}w_i\langle w_1\cdots \widehat w_i \cdots w_{d+1}\rangle =0.
$$
By the definition of $v_{P_\alpha}$ in \eqref{intersection} it is obvious that $v_{P_\alpha}$ belongs to the space $Y$. On the other hand, by~\eqref{intersection2} we have that $v_{P_\alpha}$ belongs to $P_{\alpha}$, therefore $v_{P_\alpha}\in Y\cap P_{\alpha}$.
Moreover, note that the matrix of $v_{P_\alpha}$'s in the basis $Y_{\beta}$ of $Y$ coincides with $\big(\langle Y^\perp_\alpha P_{{\beta}} \rangle \big)_{\genfrac{}{}{0pt}{2}{1\leq \alpha\leq k}{1\leq\beta\leq k}}$. This completes the proof.
\end{proof}

In the following, we will consider the case when the subspaces $P_{\alpha}$ are spanned by the columns of the matrix $Z$ indexed by $J_{\alpha} \in \binom{[n]}{ m+1}$. We denote the corresponding function \eqref{eq:functions} with the shorthand $\langle Y J_1 \cap \cdots \cap J_k \rangle$. Moreover, we focus on the amplituhedron $\mathcal{A}=\mathcal{A}_{n,k,n-k-1}$ (conjugate to a polytope).

\begin{Theorem} \label{geomchar}
Let $J=\{j_1,\ldots,j_k\}$ be a $k$-subset of $[n]$ and $\{W^{I_{j_1}},\ldots,W^{I_{j_k}}\}$ be the corresponding collection of rays of $\mathcal{F}_{\mathcal{A}}(Y)$ as defined in \eqref{eq:W}. Then
\begin{equation*}
  (-1)^{\xi_J} \cdot \det (W^{I_{j_1}},\ldots,W^{I_{j_k}})=\langle Y  \overline{I}_{j_1}\cap\cdots\cap \overline{I}_{j_k} \rangle = \det(v_{\overline{I}_{j_1}},\ldots,v_{\overline{I}_{j_k}}), 
\end{equation*}
where $\xi_J:=\sum_{j \in J} s_{I_j}$, and $s_{I_j} := k j$ if $n\notin I_j$ and $s_{I_j} := j-1$ otherwise.
In particular, a collection of rays $\{W^{I_{j_1}},\ldots,W^{I_{j_k}}\}$ is degenerate if and only if the vectors $v_{\bar{I}_{j_1}},\ldots,v_{\bar{I}_{j_k}}$ are linearly dependent.
\end{Theorem}
\begin{proof}
The first equality follows from Lemma~\ref{dets}, since both sides are determinant of matrices related by transposition and multiplication of some rows by $-1$. The second equality is Proposition~\ref{propvect}. 
\end{proof}

\begin{Lemma} \label{dets}
The rays $W^{I_j}$ of the fan $\mathcal{F}_\mathcal{A}(Y)$ can be expressed in terms of the functions in \eqref{eq:functions} as:
\begin{eqnarray*}
W^{I_j}_\alpha= (-1)^{s_{I_j}} \cdot \langle Y^\perp_\alpha \bar{I_j}\rangle,
\end{eqnarray*}
where $s_{I_j}:= kj$ if $n\notin I$ and $s_{I_j} := j-1$ otherwise.
\end{Lemma}

\begin{proof}
We recall that $W^{I_j}_\alpha=(-1)^{\alpha} \det(A_{\widehat{\alpha}}^{I_j})$ 
and
$A_{\widehat{\alpha}}^{I_j}$ is the submatrix of $A$ with the columns indexed by $I_j$ and all the rows except the $\alpha^{\rm th}$. We will show that, up to a sign, $\det (A_{\widehat{\alpha}}^{I_j})$ is the same as $\langle Y^\perp_\alpha \bar{I_j} \rangle$. 
First, we define the matrices $A$ and $B$ as follows:
\begin{eqnarray*}
B=\left[\begin{array}{@{}c|c|c|c@{}}
\mathbb{I}_{j-1} & \mathbb{O} & \begin{matrix}z_1\\ \vdots\\ z_{j-1} \end{matrix}   & Y_{\widehat\alpha}^1 \\ \hline
\mathbb{O}_{k\times (j-1)} & \mathbb{O} & \begin{matrix}z_j\\\vdots \\ z_{j+k-1} \end{matrix} &   \vdots         \\ \hline
\mathbb{O} & \mathbb{I}_{m-j+1} & \begin{matrix}z_k\\\vdots \\ z_{k+m} \end{matrix}   & Y_{\widehat\alpha}^{m+k}
\end{array}\right]\quad\text{and}\quad B'=\left[\begin{array}{@{}c|c|c|c@{}}
\mathbb{I}_{j-1}  & \begin{matrix}z_1\\ \vdots\\ z_{j-1} \end{matrix}    & Y_{\widehat\alpha}^1         & \mathbb{O}  \\ \hline
\mathbb{O}_{k\times (j-1)}  & \begin{matrix}z_j\\\vdots \\ z_{j+k-1} \end{matrix}  & \vdots                   & \mathbb{O}  \\ \hline
\mathbb{O}  & \begin{matrix}z_k\\\vdots \\ z_{k+m} \end{matrix}    &   Y_{\widehat\alpha}^{m+k}   & \mathbb{I}_{m-j+1}
\end{array}\right]
\end{eqnarray*}
If $I_j$ does not contain $m+k+1$, then we can write  the determinant $\langle Y^\perp_\alpha \bar{I_j} \rangle$ as follows: 
\begin{eqnarray*}
\langle Y^\perp_\alpha \bar{I_j} \rangle&=& 
(-1)^{\alpha+1} \langle Y_{\hat{\alpha}} \bar{I_j} \rangle
=
(-1)^{\alpha+1} (-1)^{(k-1)(m+1)} \det(B)=
(-1)^{\alpha+1} 
(-1)^{kj+m+1}
\det(B')
\\
&=&(-1)^{\alpha+kj}\det \left[\begin{array}{@{}c|c@{}}
z_j& Y_{\widehat\alpha}^j\\ 
\vdots & \vdots\\
z_{j+k-1}& Y_{\widehat\alpha}^{j+k-1}
\end{array}\right] =
(-1)^{\alpha+kj} \det(A_{\widehat{\alpha}}^{I_j})=(-1)^{kj}W^{I_j}_{\alpha}.
\end{eqnarray*}
If $I_j$ contains $m+k+1$, then the similar calculation shows that: 
\begin{equation*}
\langle Y^\perp_\alpha \bar{I_j} \rangle = (-1)^{\alpha+1} \langle Y_{\hat{\alpha}} \bar{I_j} \rangle=  (-1)^{\alpha+j+k} p_{I_j\backslash\{m+k+1\}}(Y_{\widehat \alpha})=(-1)^{\alpha+j} \det(A_{\widehat{\alpha}}^{I_j})=(-1)^{j-1}W^{I_j}_{\alpha}.\tag*{\qedhere}
\end{equation*}
\end{proof}

Now we are ready to relate the functions $\det(W^{I_{j_1}},\ldots,W^{I_{j_k}})$ to the canonical functions of the generalized triangles of the amplituhedron.
\footnote{A conjectural representation of $\Omega(\Delta_J)$ 
in the case of conjugate to polytopes was given in \cite{ferro2018amplituhedron}.}

\begin{Proposition}\label{prop:detiscan} 
Let $J=\{j_1,\ldots,j_k\}$ be a $k$-subset of $[n]$ and $I_{j_\alpha}$ be the $j_\alpha^{\rm th}$ cyclic $k$-interval.
Then the canonical function of the corresponding generalized triangle $\Delta_J$ has the following form:
\begin{equation}\label{eq:volfun}
 \Omega(\Delta_J)= \frac{\det (W^{I_{j_1}},\ldots,W^{I_{j_k}})^{m}}{\displaystyle \prod_{i \ne j_1,\ldots,j_k} \det \left( {(A^{I_{i}})}^\perp,{(A^{I_{j_1}})}^\perp,\ldots,{(A^{I_{j_k}})}^\perp \right)}.
\end{equation}
\end{Proposition}

\begin{proof}
We first note that $\Delta_J$ is the image of an $mk$-dim positroid cell $S_{\widetilde{J}}$ of $\Gr_+(k,n)$ whose only vanishing cyclic Pl\"ucker coordinates are $p_{I_{j_1}},\ldots,p_{I_{j_k}}$.
By Lemma~\ref{lem:gentriangles}, $S_{\widetilde{J}}$ is uniquely determined.
Hence, by Theorem~\ref{thm:resfiber} we have $\pi_+^{-1}(Y) \cap S_{\widetilde{J}}=q_{\widetilde{J}}$ for a unique point $q_{\widetilde{J}}$
and $\Omega (\Delta_J)$ can be computed as:
\begin{equation*}\label{prop35eq1}
    \Omega(\Delta_J)=\mbox{Res}_{q_{\widetilde{J}}}\omega_{n,k,m}=\frac{1}{|\det( W^{I_{j_1}},\ldots,W^{I_{j_k}})|} \cdot\frac{1}{\prod_{i \not \in J} \left({\bar{\lambda}}_J \cdot {(A^{I_i})}^\perp \right)},
\end{equation*}
where $\bar{\lambda}_J=(1,(\bar{\lambda}_J)_2,\ldots,(\bar{\lambda}_J)_{k+1}) \in \mathbb{R}^{k+1}$ is the unique point in the intersection $ \bigcap_{j \in J} \lbrace \bar{\lambda} \cdot {(A^{I_i})}^\perp=0\rbrace$. Now, applying Cramer's rule we have that:
\begin{equation*}
(\bar{\lambda}_J)_{\alpha}=(-1)^{\alpha+1}\frac{\mbox{det}({(A^{I_{j_1}})}^\perp_{\hat{\alpha}},\ldots,{(A^{I_{j_k}})}^\perp_{\hat{\alpha}})}{\det(W^{I_{j_1}},\ldots,W^{I_{j_k}})} \quad\text{for}\quad \alpha=2,\ldots, k+1,
\end{equation*}
where ${(A^{I_{j_a}})}^\perp_{\hat{\alpha}}=(W^{I_{j_a}}_0, W^{I_{j_a}}_1,\ldots, \widehat{W^{I_{j_a}}_\alpha},\ldots,W^{I_{j_a}}_k)$.
Hence, the volume form $\Omega(\Delta_J)$ is given by:
\begin{equation}\label{eq:denominator}
   \Omega(\Delta_J)= \frac{\left(\mbox{det}(W^{I_{j_1}},\ldots,W^{I_{j_k}})\right)^{m}}{\displaystyle\prod_{i \not \in J}\left(\det(W^{I_{j_1}},\ldots,W^{I_{j_k}}) W_0^{I_i}+\sum_{\alpha=2}^{k+1}  (-1)^{\alpha+1}\det \left({(A^{I_{j_1}})}^\perp_{\hat{\alpha}},\ldots,{(A^{I_{j_k}})}^\perp_{\hat{\alpha}}) \right) W_\alpha^{I_i}\right)}.
\end{equation}
Notice that each factor in the denominator of~\eqref{eq:denominator} that corresponds to some $i$ with $i \not \in J$, can be rewritten as the determinant of the matrix $\left( {(A^{I_{i}})}^\perp,{(A^{I_{j_1}})}^\perp,\ldots,{(A^{I_{j_k}})}^\perp \right)$ as in \eqref{eq:volfun}. This completes the proof.
\end{proof}

\begin{Conjecture}\label{con:main}
Let $J=\{j_1,\ldots,j_k\}$ be a $k$-subset of $[n]$ and $I_{j_\alpha}$ be the ${j_\alpha}^{\rm th}$ cyclic $k$-interval.  Consider $\xi_J$ as in Theorem~\ref{geomchar}.
Then the function
\begin{equation}\label{eq:conj}
\epsilon_{J \bar{J}} \, (-1)^{\xi_J} \cdot\langle Y \bar{I}_{j_1}\cap\cdots\cap \bar{I}_{j_k}\rangle
\end{equation} 
is strictly positive, where $\epsilon_{J \bar{J}}$ is the Levi-Cevita symbol in which the elements of both $J$ and $\bar{J}$ are listed in the increasing order.
\end{Conjecture}

\begin{remark} 
The positivity of the functions in \eqref{eq:conj} 
is proven in the following cases:
\begin{itemize} 

     \item[(i)] We have verified Conjecture~\ref{con:main} by explicit calculations for $m=2$ up to $n=7$, $m=4$ up to $n=8$ and $m=6$ up to $n=10$. We also expect that even if $I_{j_\alpha}$ are not cyclic intervals, the corresponding functions are still of a constant sign inside the amplituhedron.

\item[(ii)] The functions in~\eqref{eq:conj} appeared in \cite[\S7.2.4]{positivegeom} as numerators of certain canonical functions of the $m=2$ amplituhedra. The authors provided a (sketch of) proof that 
     $\langle Y \bar{I}_{j_1}\cap\cdots\cap \bar{I}_{j_k}\rangle$ are positive, when $j_1<\cdots<j_k$ and $\bar{I}_{j_{\alpha}}$
     do not contain both $1$ and $n$.
     It is easy to see that this is implied by Conjecture~\ref{con:main}. 
     Moreover,  all generalized triangles of $m=2$ amplituhedra are classified in \cite[\S2.4]{parisicluster} and these functions appear in the numerator of their canonical functions.
    
    \item[(iii)]  
    The positivity of the functions in~\eqref{eq:conj} is related to the positivity of the matrices $C$ and $Z$, by~\eqref{eq:A_{n,k,m}}. For example, for the $k=2$ case, $\langle Y {I}_{j_1} \cap {I}_{j_2} \rangle$ can be expressed as a positive linear combination of the Pl\"ucker coordinates of $C$ and the functions
       $\langle ij{I}_{j_1} \cap {I}_{j_2} \rangle := \langle i {I}_{j_1}\rangle \langle j {I}_{j_2}\rangle-\langle i {I}_{j_2}\rangle \langle j {I}_{j_1}\rangle.$
 Moreover, for $m=2$ these functions are $\langle i \, j \, (j_1,j_1+1,j_1+2) \cap (j_2,j_2+1,j_2+2) \rangle$, which are cluster variables of $\Gr(4,n)$ \cite[\S6.3]{Golden:2013xva}, and so they are positive by the Scott's seminal work \cite[\S6]{scott}. It would be interesting to explore the positivity of \eqref{eq:conj} in connection with Grassmannian cluster algebras.

\item[(iv)]
One corollary of Conjecture~\ref{con:main} is that the fan $\mathcal{F}_\mathcal{A}(Y)$ defined in \S\ref{sec:fan_conjugate_polytopes} is combinatorially equivalent to the secondary fan of a cyclic polytope (see Proposition~\ref{prop_faneq}).
In general, the combinatorics of the secondary fan of a cyclic polytope depends on the choice of its vertices on the moment curve \cite[Lemma 4.9]{athanasiadis2000fiber}. However, we expect that the combinatorial structure of the fan $\mathcal{F}_{\mathcal{A}}(Y)$ does not depend on the point $Y\in \mathcal{A}$ which is a slight strengthening of Conjecture~\ref{con:main}. See Lemma~\ref{lem:combtype}.
\end{itemize}
\end{remark}

Now we are ready to state and prove our main results, assuming that Conjecture~\ref{con:main} holds true and the combinatorial structure of the fan $\mathcal{F}_{\mathcal{A}}(Y)$ does not depend on the point $Y\in \mathcal{A}$.
\begin{Theorem}
Let $\mathcal{A}$ be an amplituhedron conjugate to a polytope.
Then:
\begin{itemize}
    \item[$(\mathrm{i})$] The combinatorial structure of the fan $\mathcal{F}_{\mathcal{A}}(Y)$ 
    defines the secondary fan of triangulations of $\mathcal{A}$.
        \item[$(\mathrm{ii})$] The volume function of a generalized triangle $\Delta$ does not vanish on $\mathcal{A}$. In particular, it has a constant sign inside $\Delta$.
    \item[$(\mathrm{iii})$] The volume function of $\mathcal{A}$ can be computed using Jeffrey-Kirwan residue.
\end{itemize}
\end{Theorem}

\begin{proof}
We note that $(\mathrm{ii})$ follows immediately from the expression of the volume function given in Proposition~\ref{prop:detiscan},  $(\mathrm{i})$ is  
Theorem~\ref{triangamp}, and $(\mathrm{iii})$ is Theorem~\ref{th:jkfibers}. 
\end{proof}

\subsection{Triangulations of amplituhedra conjugate to polytopes} \label{subsec:triangconjpol}
We fix $\mathcal{A}:=\mathcal{A}_{n,k,m}$ where $m=n-k-1$. Let $\pi:\Gr(k,n)\dashrightarrow\Gr(k,n-1)$ be the rational map from Definition~\ref{def:pi} and let $\mathcal{F}_\mathcal{A}(Y)$ be the fan from Definition~\ref{def:sec_fan} for $Y\in\mathcal{A}$.
We also assume that Conjecture~\ref{con:main} holds true and the combinatorial structure of the fan $\mathcal{F}_{\mathcal{A}}(Y)$ does not depend on the point $Y\in \mathcal{A}$. Our main goal in this section is to prove that $\mathcal{F}_\mathcal{A}(Y)$ is the secondary fan of $\mathcal{A}$. Moreover, we study positroidal dissections of $\mathcal{A}$ coming from $\mathcal{F}_\mathcal{A}(Y)$.
In analogy with the case of polytopes, we will call such dissections \emph{regular}.

\medskip
\noindent{\bf Boundaries of fibers.} Recall that the codimension one boundaries of 
$\Gr_+(k,n)$ are the positroid cells $S_i$ defined by $p_{I_i}=0$ for cyclic $k$-intervals $I_i$. Each codimension one positroid cell
intersects a full fiber $\pi^{-1}(Y)$ by a linear hyperplane. As in Definition~\ref{def:sec_fan}, we denote the collection of resulting hyperplanes by
\begin{equation*}
\mathcal{H}_i(Y) =  \pi^{-1}(Y) \cap S_i\quad\text{for}\quad i=1,\ldots,n.
\end{equation*}
Hence, in each full fiber $\pi^{-1}(Y)$ there is a collection of hyperplanes $\mathcal{H}_1(Y),\ldots,\mathcal{H}_n(Y)$ such that the positive fiber is 
bounded by some of them. Where there is no
ambiguity, we 
denote them simply by $\mathcal{H}_i$.

\medskip

Let $\rho_{j_1},\ldots,\rho_{j_k}$ be a collection of rays of $\mathcal{F}_{\mathcal{A}}(Y)$. The positivity of  functions in \eqref{eq:conj} guaranties that the ray generators $W^{I_{j_1}},\ldots,W^{I_{j_k}}$ in \eqref{eq:W}
form a basis for $\R^k$. Now, for every subset $J\subset [k]$, we let 
\[
O_J(Y)= \{(x_1,\ldots,x_k)\,|\, x_i\geq 0 \text{ for } i\in J  \text{ and }x_i\leq 0 \text{ otherwise} \}
\]
be the orthant given by the basis $W^{I_{j_1}},\ldots,W^{I_{j_k}}$. Then by the positivity of functions in \eqref{eq:conj} we have:
\begin{Lemma}\label{lem:combtype}
Let $Y,Y'$ be two points of $\A$ and let $\rho_{j_1},\ldots,\rho_{j_k}$ be a collection of rays of $\mathcal{F}_{\mathcal{A}}(Y)$. Assume that $\rho_{j_0}\in O_J(Y)$ for some $J\in [k]$. Then 
$\rho_{j_0}\in O_J(Y')$.
In particular, if $\rho_{j_0}$ belongs to the cone generated by $\rho_{j_1},\ldots,\rho_{j_k}$ over $Y$, then the same is true over $Y'$.
\end{Lemma}
\begin{proof}
On the contrary, suppose that $\rho_{j_0}\in O_{J'}(Y')$ for some $J'\ne J$.
Consider a path $\gamma:[0,1]\to \mathcal{A}$ in the interior of $\mathcal{A}$ which connects the points $Y, Y'\in \mathcal{A}$ and denote $\gamma(t)$ by $Y_t$.
Since $\rho_{j_0}$ belongs to different orthants over $Y$ and $Y'$, there is a point $t\in [0,1]$ such that $\rho_{j_0}$ belongs to a coordinate hyperplane over $Y_t$. Therefore, the rays $\rho_{j_0},\ldots,\rho_{j_k}$ are not in general position over $Y_t$ which contradicts  Conjecture~\ref{con:main}.
\end{proof}

\begin{Proposition}\label{prop_faneq} 
The fan $\mathcal{F}_\mathcal{A}(Y)$ is combinatorially equivalent to the secondary fan of a cyclic polytope.
\end{Proposition} 
\begin{proof}
Let $\{V_1,\ldots, V_n\}$ be the Gale transform of $\{W^{I_{1}},\ldots,W^{I_{n}}\}$.
For any $k$-subset $J=\{j_1,\ldots,j_k\}$ and its complement $\bar{J}=\{\ell_1,\ldots,\ell_{m+1}\}$, we know from \S\ref{sec:galeduality} that
\[
\det(W^{I_{j_1}},\ldots,W^{I_{j_k}}) = \epsilon_{J,\bar{J}} \det(V_{\ell_1},\ldots, V_{\ell_{m+1}}).
\]
On the other hand, by Theorem~\ref{geomchar} we have that
\[
\det(W^{I_{j_1}},\ldots,W^{I_{j_k}}) =  (-1)^{\xi_J}\cdot \langle Y \bar{I}_{j_1}\cap\cdots\cap \bar{I}_{j_k}\rangle.
\]
Hence,
\[
\det(V_{\ell_1},\ldots, V_{\ell_{m+1}}) = \epsilon_{J,\bar{J}}  (-1)^{\xi_J} \cdot \langle Y \bar{I}_{j_1}\cap\cdots\cap \bar{I}_{j_k}\rangle.
\]
The right side of the above equation is strictly positive by Conjecture~\ref{con:main}, hence $\det(V_{\ell_1},\ldots, V_{\ell_{m+1}})$ 
is strictly positive and so
$V_i$'s are the vertices of a cyclic polytope $P$. Moreover, $\mathcal{F}_\mathcal{A}(Y)$ is the chamber fan of $\{W^{I_{j_1}},\ldots,W^{I_{j_k}}\}$ which is 
the Gale transform of $V(P)=\{V_{\ell_1},\ldots, V_{\ell_{m+1}}\}$, hence it is the secondary fan of a cyclic polytope $P$, see e.g.~Theorem~\ref{thgaledual}. 
\end{proof}

We call $\mathcal{F}_\mathcal{A}$ the {\em secondary fan} of $\A$. From now on, whenever we refer to $\mathcal{F}_\mathcal{A}$, we mean the combinatorial type $\mathcal{F}_\mathcal{A}(Y)$ for some $Y\in\mathcal{A}$. 
In the following, given a point $Y\in \mathcal{A}$, we denote $\mathcal{F}^+_\mathcal{A}(Y)$ for the normal fan of the positive fiber over $Y$. We now show that
the normal fan $\mathcal{F}^+_\mathcal{A}(Y)$ is a coarsening of $\mathcal{F}_\mathcal{A}(Y)$.
 
\begin{Corollary}\label{cor:min}
$\mathcal{F}_\mathcal{A}$ is the minimal combinatorial common refinement of 
$\mathcal{F}^+_\mathcal{A}(Y)$ over all points $Y\in \mathcal{A}$.
\end{Corollary}
\begin{proof}
The fact that $\mathcal{F}_\mathcal{A}$ is a common refinement is obvious from Lemma~\ref{lem:combtype} and our assumption that $\mathcal{F}_\mathcal{A}(Y)$ does not depend on $Y$.
To prove the minimality, note that by Lemma~\ref{lem:gentriangles}, for any $J\in \binom{[n]}{k}$ there is  a unique generalized triangle $\Delta_J$ which is the image of a positroid cell $S_{\widetilde{J}}$
in $\Gr_+(k,n)$ for which a cyclic $k$-interval $I_j$ is in $\widetilde{J}$ if and only of $j\not\in J$. Hence, for any $J\in \binom{[n]}{k}$, there exists a point $Y\in \mathcal{A}$ such that the cone $C_J$ generated by the rays indexed by $J$ is a cone in $\mathcal{F}^+_\mathcal{A}(Y)$. Therefore, any common refinement of the normal fans of positive fibers $\mathcal{F}^+_\mathcal{A}(Y)$ over all points $Y\in \mathcal{A}$ should contain the chamber fan with rays $\rho_1,\ldots, \rho_n$. Thus, by Definition~\ref{def:sec_fan}, $\mathcal{F}_\mathcal{A}$ is the minimal common refinement of $\mathcal{F}^+_\mathcal{A}(Y)$.
\end{proof}

Now we are ready to relate the fan $\mathcal{F}_\mathcal{A}$ to the positroidal  subdivisions of $\mathcal{A}$.  In the following, for any cone $\sigma\in\mathcal{F}_\mathcal{A}$ we define a collection of positroid cells $\mathcal{C}_\sigma$ and prove that it provides a subdivision of $\A$. 
As before, we let $C_{J}$ be the cone generated by the rays of $\mathcal{F}_\mathcal{A}$ indexed by $J\subset [n]$. Note that since $\mathcal{F}_\mathcal{A}$ is a chamber fan for the rays $\rho_1,\ldots,\rho_n$, for any cone $\sigma\in\mathcal{F}_\mathcal{A}$ and any $J\subset [n]$ we can check whether $\sigma \subset C_J$.

\begin{definition}\label{def:triangamp}
{\rm Let $\sigma$ be a cone in $\mathcal{F}_\mathcal{A}$ and $\widetilde{\mathcal{C}}_\sigma$ be the collection of $J$ such that $\sigma$ is a subset of the cone $C_J$ of $\mathcal{F}_\mathcal{A}$. Then we define 
\[
\mathcal{C}_\sigma := \left\{ \text{positroid cell } S_{\widetilde{J}} \,\,  \Big| \,\,  J \text{ is a minimal element of }  \widetilde{\mathcal{C}}_\sigma \text{ by inclusion}\right\},
\]
where $S_{\widetilde{J}}$ is the unique positroid cell corresponding to $J$ given in Lemma~\ref{lem:gentriangles}.
}\end{definition}

The main result of this section is the following theorem.
\begin{Theorem} \label{triangamp}
Let  $\mathcal{F}_\mathcal{A}$ be the secondary fan of $\mathcal{A}$. Then each cone $\sigma$ of $\mathcal{F}_\mathcal{A}$  corresponds to a positroidal dissection $\mathcal{C}_\sigma$. The inclusion of two cones corresponds to the coarsening of the associated dissections. In particular, the maximal cones of $\mathcal{F}_\mathcal{A}$ correspond to positroidal triangulations of $\A$. Under this identification a maximal cone $\sigma$ corresponds to the triangulation given by $\mathcal{C}_\sigma$.
\end{Theorem}

\begin{proof}
We will first provide another presentation for the positroid cells in $\mathcal{C}_\sigma$. To do this, for a fixed cone $\sigma$ of $\mathcal{F}_\mathcal{A}$ and for a point $Y\in \mathcal{A}$, let $\sigma(Y)$ be the minimal cone of $\mathcal{F}^+_\mathcal{A}(Y)$ (by inclusion) which contains $\sigma$. Since by Corollary~\ref{cor:min} the fan $\mathcal{F}_\mathcal{A}$ is the minimal combinatorial common refinement of the fans $\mathcal{F}^+_\mathcal{A}(Y)$, the cone $\sigma(Y)$ uniquely exists. Let $f_\sigma: \mathcal{A}\rightarrow \{\text{positroid cells}\}$ be the following map:
\[
    f_\sigma: Y\mapsto \text{ the positroid cell } S_{\widetilde{J}}, \text{ where } \sigma(Y)=C_J.
\]
Moreover, note that by Lemma~\ref{lem:gentriangles} for every minimal element $J$ of $\widetilde{\mathcal{C}}_\sigma$, the cone $C_J$ appears in $\mathcal{F}^+_\mathcal{A}(Y)$ for some $Y\in \mathcal{A}$. Hence, the collection $\mathcal{C}_\sigma$ coincides with $f_\sigma(\mathcal{A})$.

We now show that $\mathcal{C}_\sigma$ gives a positroidal subdivision of $\A$. First, note that $\sigma(Y)$ is a cone of $\mathcal{F}^+_\mathcal{A}(Y)$, hence its corresponding positroid cell intersects the positive fiber $\pi_+^{-1}(Y)$. Thus, the collection of positroid cells $f_\sigma(\mathcal{A})=\mathcal{C}_\sigma$ covers $\mathcal{A}$. 
Finally, for any $Y' \in \pi(f_\sigma(Y))$ we have  $\sigma(Y)=\sigma(Y')$ and so 
$f_\sigma(Y') =f_\sigma(Y)$. This implies that the images of the positroid cells in $\mathcal{C}_\sigma$ do not overlap, which completes the proof.
\end{proof}

\begin{remark}
The construction above can be generalized to any Grassmann polytope with linear fibers; See \S\ref{subsec:Grassmann} for more details.  
Moreover, by construction, dissections from Theorem~\ref{triangamp} are in fact good dissections. We refer to them as \emph{regular subdivisions} because via parity duality, they are in bijection with regular subdivisions of cyclic polytopes. 
Non-regular subdivisions of polytopes correspond to a family of dissections of amplituhedra conjugate to polytopes which are not gathered in Theorem~\ref{triangamp}.
\end{remark}

\section{Jeffrey-Kirwan residue for linear fibers}\label{sec:JK}
In this section, we first recall the original definition of Jeffrey-Kirwan residue 
from \cite{JEFFREY1995291, BrionVergne}. Then, we extend it to accommodate its application in the amplituhedron theory such as computing the residues of fiber volume forms. In particular, we make precise connections between the residue computation and the geometric framework developed in \S\ref{sec:polytopes} and \S\ref{sec:linear}.

\subsection{Jeffrey-Kirwan residue} \label{subsec:jkres}
Jeffrey and Kirwan introduced the notion of \emph{Jeffrey-Kirwan residue} (JK-residue)
in their study of localization of group actions \cite{JEFFREY1995291}. More precisely, given a symplectic manifold and a group action on it, the JK-residue relates elements of the equivariant cohomology of the manifold to the ones of the cohomology of its symplectic quotient. 
They were inspired by Witten \cite{Witten:1992xu} who re-examined the non-abelian localization of Duistermaat and Heckman \cite{Duistermaat1982} and applied it in the context of $2$-dim Yang-Mills theory. 
Since then, JK-residue has played an increasingly important role in physics, with particular applications in \emph{supersymmetric localization} for gauge theories in various dimensions \cite{Benini:2013xpa,Closset:2015rna,Benini:2015noa}.
The notion of JK-residue has also been extended out of the realm of localization, and can be generally regarded as an operation on rational differential forms with poles on hyperplane arrangements 
\cite{BrionVergne,2004InMat.158..453S}. Furthermore, the study of the connection between JK-residues and amplituhedra is started in \cite{ferro2018amplituhedron}.

\medskip
\paragraph{Classical JK-residue.}
Fix $1 \leq r \leq n$, and consider an arrangement of hyperplanes $\{\mathcal{H}_{i} \}_{i \in [n]}$ in $\mathbb{R}^r$ where $\mathcal{H}_{i}=\{x \in \R^r\,|\, \langle x, \beta_i\rangle\}=0$. Here, $\beta_i$ is the normal vector of the hyperplane $\mathcal{H}_{i}$ for each $i$.  
Let $I$ be an $r$-subset of $[n]$ such that $\{\beta_i \}_{i \in I}$ forms a basis of $\mathbb{R}^r$. Let $f_I$ be the rational function given by:
\begin{equation}\label{eq:rational_function}
f_I(x) = \frac{1}{\prod_{i\in I}\langle x, \beta_i \rangle}.
\end{equation}
Note that the function $f_I$ has simple poles on the hyperplanes $\{\mathcal{H}_i\}_{i \in I}$. 
Then we define:

\begin{definition}
\label{def:jkorig}
{\rm The \emph{JK-residue} of the function $f_I$ at a point $\xi \in \mathbb{R}^r$ is defined as:
\begin{equation}\label{def:jkonbasic}
    \mathrm{JK}_{\xi}(f_I) := \begin{cases} {\left| \mbox{\textup{det}}(\{ \beta_i\}_{i \in I}) \right|}^{-1} &\mbox{if } \xi \in C_I \\
0 & \mbox{if } \xi \not \in C_I \end{cases},
\end{equation}
where $C_I$ denotes the positive span of the vectors $\{ \beta_i\}_{i \in I}$.
}\end{definition}

\begin{remark}
Following \cite{BrionVergne}, the definition of  $\mathrm{JK}_{\xi}(\cdot)$ can be extended to the space of rational functions $\phi$ with poles on the hyperplane arrangement $\{\mathcal{H}_i \}_{i \in [n]}$. The space of such functions is graded by the degree, and JK-residue acts non-trivially only on the part of degree $-r$. In particular, every such function $\phi$ can be decomposed in terms of $f_I$'s and other fractions whose denominators correspond to some vectors $\beta$ which do not necessarily span $\mathbb{R}^r$. 
To make this more precise, given a flag $F$
\begin{equation*}
    0 \subset F_1 \subset F_2 \subset \cdots \subset F_r=\mathbb{R}^r,
\end{equation*}
we consider the vectors $\{ \kappa^F_j\}_{j \in [r]}$ with $\kappa^F_j:=\sum_{i:\beta_i \in F_j} \beta_i$ and define a function $\nu(F)$ and a cone $C_F$ as:
\[
\nu(F):=\mbox{sign} \left(\mbox{det}(\kappa^F_1,\ldots,\kappa^F_r)\right) \quad\text{  and }\quad C_F:=\mbox{span}_+(\kappa^F_1,\ldots,\kappa^F_r).
\]
Moreover, given a vector $\xi \in \mathbb{R}^k$, we denote $\mathcal{FL}^+(\xi)$ for the set of flags $F$ such that $\xi$ belongs to (the interior of) the cone $C_F$. Then by \cite[Theorem 2.6]{2004InMat.158..453S}, we can express the action of $\mathrm{JK}_{\xi}(\cdot)$ on rational functions $\phi$ in terms of actual residues as:
\begin{eqnarray}\label{eq:jkoriginal}
  \mathrm{JK}_{\xi}(\phi)=\sum_{F \in \mathcal{FL}^+(\xi)} \nu(F) \, \mbox{Res}_F(\phi).
\end{eqnarray}
\end{remark}

\begin{remark}
Recalling the notion of chamber fan from Definition~\ref{def:chamberfan}, note that every pair of vectors $\xi,\xi' \in \mathbb{R}^r$ contained in the same chamber cone lead to the same JK-residue on the space of functions $\phi$, i.e.~$\mathrm{JK}_{\xi}(\cdot)=\mathrm{JK}_{\xi'}(\cdot)$.
In other words, chambers are exactly the connected components of the set of all generic vectors $\xi \in \mathbb{R}^r$ with respect to the hyperplane arrangement $\{\mathcal{H}_i \}_{i \in [n]}$. 
\end{remark}

\paragraph{Extended JK-residue.}
We extend the definition of JK-residue to differential forms on $\R^r$ with simple poles along configuration of affine hyperplanes in $\R^r$, i.e.~the hyperplanes which do not necessarily pass through the origin. More precisely, consider a top-degree form $\omega$ of degree $r$ in $\R^r$ with simple poles on a generic
affine hyperplane arrangement $\{\mathcal{H}_i\}_{i \in [n]}$. where $\mathcal{H}_{i}=\{x \in \R^r\,|\, \langle \beta_i, x\rangle + a_i =0\}$. Note that $\{\beta_i \}_{i \in[n]}$ is the set of rays of the normal fan of such arrangement.
By the genericity assumption, each collection $\{\mathcal{H}_i\}_{i \in I}$ of hyperplanes, with $I \in \binom{[n]}{r}$, intersects in a point $q_I \in \mathbb{R}^r$.
We fix one such point $q_I$ and consider the leading term of the Laurent expansion of $\omega$ around $q_I$ as:
 \begin{equation}\label{eq:localomega}
     \omega^{(q_I)}:=f_{\bar{I}}(q_I) \, \omega_I, 
 \end{equation}
where $f_{\bar{I}}$ is a function holomorphic in a neighborhood of $q_I$ and $\omega_I$ is a  top-degree form with simple poles on $\{\mathcal{H}_i\}_{i \in I}$.
Moreover, we translate $\omega^{(q_I)}$ to the origin as:   $\omega^{(q_I)}_0:=\omega^{(q_I)}(x-q_I).$ We repeat this procedure for all $\binom{n}{r}$ intersection points $q_I$ of the hyperplanes $\{\mathcal{H}_i\}_{i \in I}$ and define:
\begin{equation*}
\omega_0:=\sum_{I \in \binom{[n]}{r}} \omega^{(q_I)}_0.
\end{equation*}
Moreover, let $\phi$ be a rational function such that 
\begin{equation}\label{eq:w0}
\omega_0=\phi \, \mbox{d} x_1\wedge \cdots \wedge \mbox{d} x_r.
\end{equation}
Note that since we work with a generic hyperplane arrangement, the function $\phi$ is of pure degree $-r$.

\medskip

Now, following the above notation we define the (extended)
JK-residue of the form $\omega$ as follow:

\begin{definition}\label{def:JKRes}
{\rm Let $\omega$ be a top-degree form of degree $r$ in $\mathbb{R}^r$ and $\phi = \frac{\omega_0}{\mbox{d} x_1\wedge \cdots \wedge \mbox{d} x_r}$ as before. Then we define the  \emph{(extended) JK-residue} of $\omega$ at a point $\xi \in \mathbb{R}^r$ as the classical JK-residue of $\phi$:
\begin{equation*}
    \mathrm{JK}_{\xi} \, \omega :=  \mathrm{JK}_{\xi} \, \phi.
\end{equation*}
}\end{definition} 

We now express the extended JK-residue in terms of a residue computation on $\omega$ as follows:

\begin{Proposition}\label{prop:jkresasold}
Let $\mathcal{H}_i$ be an affine hyperplane with the defining equation $D_i(x)= \langle \beta_i, x\rangle + a_i =0$ for $i=1,\ldots,n$. Consider the hyperplane arrangement $\{\mathcal{H}_i\}_{i=1}^n$ and the differential form $\omega$:
\begin{equation*}
    \omega= \frac{{\rm d}x_1\wedge \cdots \wedge {\rm d}x_r}{D_1(x)\cdots D_n(x)}. 
\end{equation*} 
Furthermore, let ${\rm Res}_{C_I} \omega$ denote the residue of $\omega$ at the point $q_I= \bigcap_{i\in I}\mathcal{H}_i$ that is:
\begin{equation*}
    {\rm Res}_{C_I} \, \omega :=\frac{1}{|\det(\{ \beta_i\}_{i \in I})|}\cdot \frac{1}{\prod_{j\notin I} D_j(q_I)}.
\end{equation*}
Then the (extended) JK-residue of $\omega$ at a point $\xi\in\mathbb{R}^r$ can be computed as:\footnote{In \cite{ferro2018amplituhedron,Benini:2013xpa}, the right side of~\eqref{def:jkphysics} was used as a starting definition for the JK-residue of $\omega$.} 
\begin{equation}\label{def:jkphysics}
    \mathrm{JK}_{\xi} \omega = \sum_{I: \xi \in C_I} \, {\rm Res}_{C_I} \, \omega.
\end{equation}
\end{Proposition}

\begin{proof}
Recalling the definitions of $\omega_0^{(q_I)}$ and $\phi$ from \eqref{eq:w0} and \eqref{eq:localomega} we have: 
\begin{eqnarray*}
\omega_0^{(q_I)}=\frac{\bigwedge_{i =1}^{r} \mbox{d} x_i}{\prod_{i \in I} D_i(x)}\cdot \frac{1}{\prod_{j \not \in I } D_j(q_I)}
\quad\text{and}\quad 
\phi = \sum_{I\in \binom{[n]}{r}} \frac{1}{\prod_{i \in I} D_i(x-q_I)}\cdot \frac{1}{\prod_{j \not \in I } D_j(q_I)}.
\end{eqnarray*}
Note that $\omega_0^{(q_I)}$ has poles on the hyperplanes whose normal vectors are the rays of the cone $C_I$. Since the corresponding $r$ hyperplanes intersect in a point $q_I$, $C_I$ is full-dimensional. Hence, by~\eqref{def:jkonbasic}, the function $\omega_I$ contributes in~\eqref{eq:localomega} by ${\left| \mbox{\textup{det}}(\{ \beta_i\}_{i \in I}) \right|}^{-1}$ when $\xi \in C_I$. Thus, by linearity of the functional $\mathrm{JK}_{\xi}(\cdot)$:
\begin{equation}\label{eq:jkproof}
    \mathrm{JK}_{\xi}(\phi)= \sum_{I : \xi \in C_I} \frac{1}{|\det(\{ \beta_i\}_{i \in I})|}\cdot \frac{1}{\prod_{j\notin I} D_j(q_I)}.
\end{equation}
Finally, we note that each term in \eqref{eq:jkproof} is, up to a sign, the multivariate residue of $\omega$ computed at~$q_I$. 
\end{proof}

\vspace{-2mm}
\subsection{Canonical functions from JK-residue}\label{subsec:canjk}
Here we apply the JK-residue to answer Question~\ref{ques:res} for amplituhedra which are either polytopes or conjugate to polytopes. More precisely, we show how the residue computation naturally ties in with the geometric framework introduced in \S\ref{sec:polytopes} and \S\ref{sec:linear}.

\medskip

Using the JK-residue, we first compute the volume function of the $k=1$ amplituhedra. The statement was already made in \cite{ferro2018amplituhedron}. Here, we present an explicit proof using our terminology.

\begin{Proposition} \label{prop:jkpolytopes}
Let $\xi$ be a generic vector in the secondary fan of the cyclic polytope $\A_{n,1,m}$. Then the canonical function $\Omega(\A_{n,1,m})$ can be obtained from the JK-residue of the fiber volume form $\omega_{n,1,m}$ as: 
\begin{equation*}
   \Omega(\A_{n,1,m})(Y)= 
   \mathrm{JK}_{\xi} \, \omega_{n,1,m}(Y).
\end{equation*}
Moreover, if $\xi$ is in a given cone $\sigma$ of the secondary fan of $\A_{n,1,m}$, then
\begin{equation*}
   \Omega(\A_{n,1,m})= \sum_{I: \sigma \subseteq C_I} \Omega(\Delta_{\bar{I}}),
\end{equation*}
where $\Delta_{\bar{I}}$ is the simplex whose vertices are the columns $Z_i$ of the matrix $Z$ defining $\A_{n,1,m}$ for ${i\in\bar{I}}$.
\end{Proposition}

\begin{proof}
As in \S\ref{sec:fibers}, without loss of generality, we work on the chart $\lambda=(\lambda_1,\ldots, \lambda_{n-m-1},1)$. Then:
\begin{equation*} 
     \omega_{n,1,m}=\frac{ \mbox{d} \lambda_1 \wedge \cdots \wedge \mbox{d} \lambda_{n-m-1}}{\prod_{i=1}^n \left( \lambda \cdot A^i \right)},
\end{equation*}
by Corollary~\ref{cor:volumefibers}. Now, let $q_I$ be a pole of $\omega_{n,1,m}$ given in Theorem~\ref{thm:resfiber}(i) as the intersection point of the fiber $\pi^{-1}(Y)$ with some generalized triangle. 
Then by Proposition~\ref{prop:jkresasold} we have:
\begin{equation}\label{form1}
\mathrm{JK}_{\xi}\, \omega_{n,1,m}=  \sum_{I: \,  \xi \in C_I } {\mathrm{Res}}_{C_I} \, \omega_{n,1,m},
\end{equation}
where $C_I$ are simplicial cones whose rays are $\lbrace Z^\perp_i\rbrace_{i \in I}$ for $I \in \binom{[n]}{n-m-1}$. As explained in Example~\ref{ex:polytope2}, the Gale dual of $Z$ is indeed the normal fan of the affine hyperplane arrangement $\{\mathcal{H}_i\}_{i \in [n]}$ in this chart, where $\mathcal{H}_i=\{\lambda \cdot A^i=0\}$.
Moreover, since $\mathcal{A}_{n,1,m}$ is a polytope, Theorem~\ref{thgaledual} implies that $\lbrace C_I: \xi \in C_I \rbrace$ corresponds to a triangulation $\lbrace\Delta_{\bar{I}}\rbrace$ of $\mathcal{A}_{n,1,m}$, where $\Delta_{\bar{I}}$ are simplices with vertices $\lbrace Z_i\rbrace_{i \in \bar{I}}$. Hence, by Theorem~\ref{thm:resfiber}(ii), we have that
\begin{equation}\label{eq:sign}
  {\mathrm{Res}}_{C_I} (\omega_{n,1,m})= \Omega (\Delta_{\bar{I}}).
\end{equation}
Let $q_I= \bigcap_{i\in I}\mathcal{H}_i$. Note that both sides of \eqref{eq:sign} have the same sign, because $q_I\in\Delta_{\bar{I}}$ for any $Y\in \Delta_{\bar{I}}$ and so $q_I \cdot A^i\geq 0$ for all $i$. Now, \eqref{eq:sign} together with \eqref{eq:AHformTriang} implies that \eqref{form1} evaluates to $\Omega(\mathcal{A}_{n,1,m})$. Finally, choosing $\xi$ from another chamber of the normal fan of $\{\mathcal{H}_i\}_{i \in [n]}$ corresponds to choosing another chamber in the secondary fan $\{Z_i^\perp \}_{i \in [n]}$ which leads to another regular triangulation of $\mathcal{A}_{n,1,m}$. Therefore, following the same argument as above, we are able to express the canonical function of $\mathcal{A}_{n,1,m}$ as the sum over the canonical functions of simplices in the triangulation. This concludes the proof. 
\end{proof}

\begin{remark}
We observe that the proof of Proposition~\ref{prop:jkpolytopes} can be easily extended beyond cyclic polytopes, since Gale duality can be applied to any configuration of points; See \S\ref{sec:galeduality}. In particular, with similar arguments, one can use the JK-residue to obtain the canonical function of any convex polytope.
\end{remark}

\begin{example}[Pentagon] 
\label{ex:pentagon}
Consider the (projective) pentagon $\mathcal{A}_{5,2,1}$.
As in \eqref{matrixA}, we choose $A$ to be:
\begin{equation*}
   A= \left(\begin{array}{c|c@{}}
\begin{matrix}  -\frac{\langle 234 \rangle}{\langle 123 \rangle} & \frac{\langle 134 \rangle}{\langle 123 \rangle}& -\frac{\langle 124 \rangle}{\langle 123 \rangle}\\ 
  -\frac{\langle 235 \rangle}{\langle 123 \rangle} & \frac{\langle 135 \rangle}{\langle 123 \rangle}& -\frac{\langle 125 \rangle}{\langle 123 \rangle}   
  \\ \end{matrix}    & \mathbb{I}_{2 \times 2}\\
  \hline
\begin{matrix}
  \frac{\langle Y 23 \rangle}{\langle 123 \rangle}& -\frac{\langle Y 13 \rangle}{\langle 123 \rangle}& \frac{\langle Y 12 \rangle}{\langle 123 \rangle}\end{matrix} & \mathbb{O}_{1 \times 2}
\end{array}\right)
\end{equation*}
We parametrize $\pi_+^{-1}(Y)$ using local
coordinates $\lambda=(\lambda_1,\lambda_2,1) \in \mathbb{P}^2$. Then
$    \omega_{5,1,2}=\frac{ \mbox{d} \lambda_1 \wedge \mbox{d} \lambda_{2}}{\prod_{i=1}^5 \left( \lambda \cdot A^i \right)},$
where $A^i$ is the $i^{\rm th}$ column of $A$. With our choice of coordinates, the cones $\mathfrak{C}_{k_1 k_2}$ are spanned by positive linear combinations of $\lbrace Z^\perp_{k_1},Z^\perp_{k_2}\rbrace$, depicted in $\mathbb{R}^2$ in Figure~\ref{fig:coneschambers}. We now fix a vector $\xi \in \mathbb{P}^2$ as in Figure~\ref{fig:coneschambers} such that $\xi$ is in the chamber $\mathfrak{c}_1$. Then by Proposition~\ref{prop:jkresasold} the JK-residue is computed as:
\begin{equation*}
  \mathrm{JK}_{\xi} \, \omega_{5,1,2} = \sum_{\mathfrak{C}_I \ni \xi} \mbox{Res}_{\mathfrak{C}_I} \omega_{5,1,2}= \left(\mbox{Res}_{\mathfrak{C}_{25}} +\mbox{Res}_{\mathfrak{C}_{45}} +\mbox{Res}_{\mathfrak{C}_{23}} \right) \omega_{5,1,2},
\end{equation*}
since $\xi$ is contained in the cones $\mathfrak{C}_{25},\mathfrak{C}_{45},\mathfrak{C}_{23}$.
\begin{figure}[ht] 
\centering{
\def\svgwidth{0.6\linewidth}{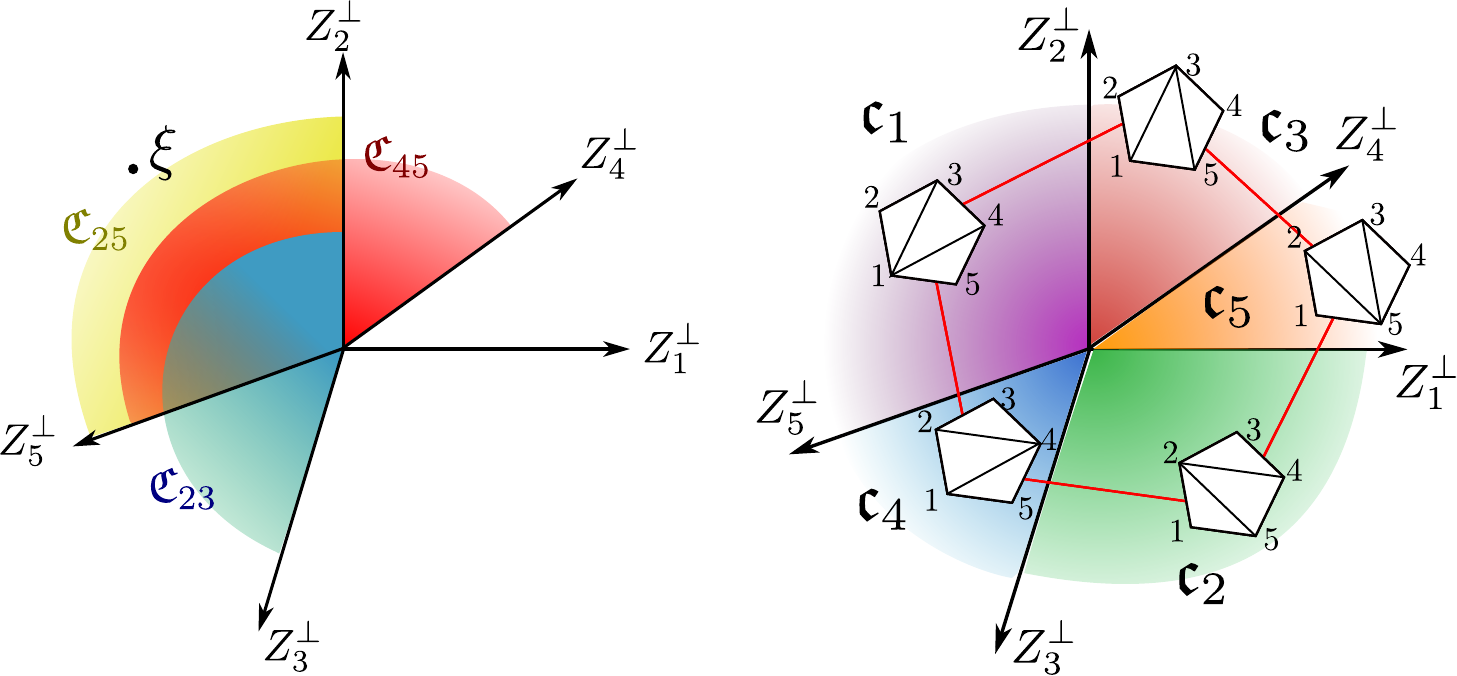}
}
\caption{Illustration of cones and chambers for $\A_{5,2,1}$}
\label{fig:coneschambers}
\end{figure}
This leads to the following representation of $\Omega(\mathcal{A}_{5,2,1})$:
\begin{equation*}
  \mathrm{JK}_{\xi} \, \omega_{5,1,2} = \Omega(\Delta_{134})+  \Omega(\Delta_{123})+\Omega(\Delta_{145})=\Omega(\mathcal{A}_{5,2,1}),
\end{equation*}
where $\Delta_{k_1 k_2 k_3}$ are triangles in $\mathbb{P}^2$ with the vertices $k_1,k_2$ and $k_3$. Clearly, this corresponds to the triangulation of the pentagon into $\lbrace \Delta_{134},\Delta_{123},\Delta_{145}\rbrace$. All the other four triangulations of $\mathcal{A}_{5,2,1}$ can be analogously obtained by choosing the reference vector $\xi$ in different chambers.
\end{example}

We are now in a position to prove an analogous statement of Proposition~\ref{prop:jkpolytopes} for linear fibers, in particular for the case of conjugate to polytopes for even $m$. Following our notation $\mathcal{C}_{\sigma}$ from Definition~\ref{def:triangamp}, and $S_{\tilde{I}}$ and $\Delta_{I}=\overline{\pi_+(S_{\tilde{I}})}$ from Definition~\ref{def:triangulations} and Lemma~\ref{lem:gentriangles}, we prove the following theorem.

\begin{Theorem}\label{th:jkfibers}
Let $\mathcal{A}=\mathcal{A}_{n,\ell,m}$ be an amplituhedron conjugate to a polytopes for $\ell=n-m-1$. Then the canonical function $\Omega(\mathcal{A})$ can be obtained from the JK-residue of its fiber volume form $\omega_{n,\ell,m}$  as: 
\begin{equation*}
   \Omega(\mathcal{A})(Y)= 
   \mathrm{JK}_{\xi} \, \omega_{n,\ell,m}(Y)\quad\text{for any}\quad Y\in\A,
\end{equation*}
where $\xi$ is a generic point in the secondary fan $\mathcal{F}_{\mathcal{A}}(Y)$. Moreover, if $\xi$ is in a cone $\sigma$ of $\mathcal{F}_{\mathcal{A}}(Y)$, then
\begin{equation*}
   \Omega(\mathcal{A})= \sum_{S_{\tilde{I}} \in \mathcal{C}_{\sigma}} \Omega(\Delta_{I}).
\end{equation*}
\end{Theorem}

\begin{proof}
Without loss of generality, we work on the chart $\bar{\lambda}=(1,\bar{\lambda}_2,\ldots, \bar{\lambda}_{n-m})$ compatible with \eqref{matrixA}. Then, by Corollary~\ref{cor:volumefibers}, we have that:
\begin{equation*} 
    \omega_{n,\ell,m}=\frac{\mbox{d} \bar{\lambda}_2 \wedge \cdots \wedge \mbox{d} \bar{\lambda}_{n-m}}{\prod_{i=1}^n  \bar{\lambda} \cdot {(A^{I_i})}^\perp}.
\end{equation*}
Consider a hyperplane arrangement $\{\mathcal{H}_i\}$ with defining equations $D_i(\bar{\lambda})=\bar{\lambda} \cdot {(A^{I_i})}^\perp=0$ for $i=1,\ldots, n$. By Lemma~\ref{lem:Wi} the rays $\{W^{I_i} \}$ of the secondary fan $\mathcal{F}_{\mathcal{A}}(Y)$ are the normal rays of these hyperplanes. For an $\ell$-subset $I$, let $C_I$ be the simplicial cone generated by rays $\{W^{I_i} \}_{i\in I}$. By Theorem~\ref{thm:resfiber}(ii), we have:
\begin{equation}\label{eq:ResCI}
 \pm {\mathrm{Res}}_{C_I} \omega_{n,\ell,m}=  {\mathrm{Res}}_{q_I} \omega_{n,\ell,m} = \Omega (\Delta_{I}),
\end{equation} 
where $q_I= \bigcap_{i\in I}\mathcal{H}_i$. Note that for any $Y\in \Delta_{I}$, the point $q_I$ belongs to the positive fiber $\pi_+^{-1}(Y)$. Hence we have that $D_i(q_I)\geq 0$ for any $i$, as $D_i\geq 0$ are defining inequalities of $\pi_+^{-1}(Y)$.
In particular,
\[
{\mathrm{Res}}_{C_I} (\omega_{n,\ell,m})(Y)>0 \text{ for any } Y\in \Delta_{I}.
\]
On the other hand, by the sign convention from Remark~\ref{rk:orientation2} we have ${\mathrm{Res}}_{q_I} \omega_{n,\ell,m}(Y)= \Omega (\Delta_{I})(Y)>0.$
Furthermore, by Proposition~\ref{prop:jkresasold} we have that:
\begin{equation}\label{form1conj}
\mathrm{JK}_{\xi}\, \omega_{n,\ell,m}=  \sum_{I: \,  \xi \in C_I } {\mathrm{Res}}_{C_I} \, \omega_{n,\ell,m}(Y).
\end{equation}
Moreover, by the proof of Theorem~\ref{triangamp}, $\{\Delta_{I}:\xi \in C_I \}$ is a positroidal triangulation of $\mathcal{A}$. So \eqref{eq:AHformTriang} implies that \eqref{form1conj} evaluates to $\Omega(\mathcal{A})$ which completes the proof.
\end{proof}

\begin{example}[Conjugate to Pentagon] \label{ex:conjpentaton}
For the amplituhedron $\mathcal{A}_{5,2,2}$, as in \eqref{matrixA}, we choose $A$ to be:
\begin{equation*}
   A= \left(\begin{array}{c|c@{}}
\begin{matrix} -z_1 & -z_2 & -z_3 & -z_4 \end{matrix}   & 1\\
  \hline
Y_{2 \times 4} & \mathbb{O}_{2 \times 1}
\end{array}\right).
\end{equation*}
We parametrize $\pi_+^{-1}(Y)$ using local inhomogeneous coordinates $\bar{\lambda}=(1,\bar{\lambda}_2,\bar{\lambda}_3) \in \mathbb{P}^2$. Then:
\begin{equation*} 
     \omega_{5,2,2}=\frac{ \mbox{d} \bar{\lambda}_2 \wedge \mbox{d} \bar{\lambda}_{3}}{\prod_{i=1}^5 \left( \bar{\lambda} \cdot {(A^{I_i})}^\perp \right)}.
\end{equation*}
Note that each ${(A^{I_i})}^\perp$ is a $1$-dim space in  $\mathbb{R}^3$. In particular,
\begin{eqnarray*}
{(A^{I_i})}^\perp&=&(p_{i,i+1}(Y),z_i Y_2^{i+1}-z_{i+1}Y_2^i,-z_i Y_1^{i+1}+z_{i+1}Y_1^i), \quad\text{for}\quad i \in \{1,2,3\}\\
{(A^{I_4})}^\perp&=&(0,Y^4_2,-Y^4_1), \quad {(A^{I_5})}^\perp=(0,Y^1_2,-Y^1_1).
\end{eqnarray*}
Then, the normal vectors $\{W^{I_i}\}_{i \in [5]}$ of the affine hyperplanes are:
\begin{eqnarray*}
W^{I_i}&=&-(z_i Y_2^{i+1}-z_{i+1}Y_2^i,-z_i Y_1^{i+1}+z_{i+1}Y_1^i), \quad\text{for}\quad i \in \{1,2,3\}\\
W^{I_4}&=&-(Y^4_2,-Y^4_1), \quad  W^{I_5}=-(Y^1_2,-Y^1_1). 
\end{eqnarray*}
For example, we have:
\begin{equation*}
    \det\binom{W^{12}}{W^{23}} = z_1z_2p_{23}(Y)+z_2^2p_{31}(Y)+z_3z_2p_{12}(Y) =- \langle Y (345) \cap (145) \rangle.
\end{equation*}

By choosing $\bar{\lambda}=(1,\bar{\lambda}_2,\bar{\lambda}_3)$ as 
coordinates, the cones $\mathfrak{C}_{k_1 k_2}$ are spanned by positive linear combinations of $\lbrace W^{I_{k_1}}, W^{I_{k_2}}\rbrace$. The rays form a configuration depicted in Figure~\ref{fig:coneschambers}. The JK-residue computation is analogous to the one in Example~\ref{ex:pentagon}.
More specifically, if we fix $\xi$ in the cones $\mathfrak{C}_{25},\mathfrak{C}_{45},\mathfrak{C}_{23}$, we have:
\begin{equation*}
  \mathrm{JK}_{\xi} \, \omega_{5,2,2} = \sum_{\mathfrak{C}_I \ni \xi} \mbox{Res}_{\mathfrak{C}_I} \omega_{5,2,2}= \left(\mbox{Res}_{\mathfrak{C}_{25}} +\mbox{Res}_{\mathfrak{C}_{45}} +\mbox{Res}_{\mathfrak{C}_{23}} \right) \omega_{5,2,2}.
\end{equation*}
This leads to the following representation of $\Omega(\mathcal{A}_{5,2,2})$:
\begin{equation*}
  \mathrm{JK}_{\xi} \, \omega_{5,2,2} = \Omega(\Delta_{25})+  \Omega(\Delta_{45})+\Omega(\Delta_{23})=\Omega(\mathcal{A}_{5,2,2}), \quad \Omega(\Delta_{j_1 j_2})(Y)= \frac{\langle Y \bar{I}_{j_1} \cap \bar{I}_{j_2}\rangle^2}{\displaystyle \prod_{i \ne j_1,j_2} \det \left( {(A^i)}^\perp,{(A^{I_{j_1}})}^\perp,{(A^{I_{j_2}})}^\perp \right)},
\end{equation*}
where $\Delta_{j_1 j_2}$ are generalized triangles of $\mathcal{A}_{5,2,2}$ in Lemma~\ref{lem:gentriangles}. This corresponds to the triangulation of $\mathcal{A}_{5,2,2}$ which is parity dual to the triangulation of the pentagon $\mathcal{A}_{5,2,1}$ into $\lbrace \Delta_{134},\Delta_{123},\Delta_{145}\rbrace$. All other triangulations of $\mathcal{A}_{5,2,2}$ can be analogously obtained by choosing the vector $\xi$ in different chambers.
\end{example}

\vspace{-4mm}

\section{Fiber positive geometries}\label{sec:fiberposgeom}
In this section, motivated by the discussions about polytopes and amplituhedra, we propose a generalization of the fiber-based approach for studying triangulations and canonical forms of positive geometries. With a more speculative tone, we introduce the general framework of \emph{fiber positive geometries}. 

\subsection{General framework} \label{subsec:genframe}
Let $(X,X_+)$ and $(Y,Y_+)$ be two positive geometries such that $\mbox{dim}(X) \geq \mbox{dim}(Y)$ and let $\pi: X \rightarrow Y$ and $\pi_+: X_+ \rightarrow Y_+$ be projections such that the fibers $(\pi^{-1}(y), \pi_+^{-1}(y))$ are also positive geometries for all $y\in Y_+$.
Analogous to Definition~\ref{def:triangulations}, we define \emph{triangulations} of $Y_+$ induced by the map $\pi_+$ as follows:
\begin{definition}
{\rm A finite collection $\mathcal{C} = \{S\}$ 
of strata of $X_+$ is 
a $\pi_+$\emph{-dissection} of $Y_+$ if we have that:
\begin{itemize}
	\item $\dim \, \overline{\pi_+(S)}  = \dim(Y)$ for each $S$ in $\mathcal{C}$.
	\item The images $\pi_+(S)$ and $\pi_+(S')$ 
		of two distinct elements in the collection are disjoint.
	\item
	$\bigcup\overline{\pi_+(S)}= Y_+$, i.e.~the union of the images of the elements in $\mathcal{C}$ is dense in $Y_+$.
	\end{itemize}
 A $\pi_+$\emph{-dissection} $\mathcal{C} = \{S\}$
	of $Y_+$ is a $\pi_+$\emph{-triangulation}
	of $Y_+$ if $\pi_+$ is injective on each $S$ in $\mathcal{C}$. In this case, each $S$ is called a  $\pi_+$\emph{-generalized triangle} of $Y_+$.
}\end{definition}
 With a slight abuse of terminology, we may also call $\Delta_S:=\overline{\pi_+(S)}$ a $\pi_+$\emph{-generalized triangle} and $ \{\Delta_S\}$ a $\pi_+$\emph{-triangulation}.
Then, as in Conjecture~\ref{conj:canformfromtriang}, we would like to have a compatibility between triangulations and canonical forms of positive geometries, as follows:
\begin{property}
Let $\mathcal{C}=\lbrace S\rbrace$ be a $\pi_+$-triangulation of $Y_+$. Then its canonical form can be computed as: 
\begin{equation*}
 \mathbf{\Omega}\left(Y_+ \right)=\sum_{S\in\mathcal{C}}   \mathbf{\Omega}\left(\pi_+(S) \right).
\end{equation*}
In particular, the sum  does not depend on the $\pi_+$-triangulation $\mathcal{C}$.
\end{property}
Let $\mu_Y$ be a standard (covariant) volume form on $Y.$ The canonical form of $Y_+$ can be written as:
\begin{equation*}
   \mathbf{\Omega}\left(Y_+ \right)=\Omega(Y_+) \, \mu_Y,
\end{equation*}
where we call $\Omega(Y_+)$ the \emph{canonical function} of $Y_+$. Moreover, analogous to Definition~\ref{def:omegafiber}, we define:
\begin{definition}
{\rm Let $\omega$ be a form on $X$ of degree $\dim(X)-\dim(Y)$ with
   $\mathbf{\Omega}(X_+)=\omega \wedge \pi^*(\mu_Y).$
Given $y \in Y_+$, we let $\omega_\pi(y)$ be the \emph{volume form on the fiber} $\pi^{-1}(y)$ of $Y$ by restricting $\omega$ to the fiber:
\begin{equation*}
    \omega_\pi(y)=\omega|_{\pi^{-1}(y)}.
\end{equation*}
}\end{definition}

As Theorem~\ref{thm:resfiber}, we would like to relate $\omega$ to the triangulations of the positive geometry, as follows: 
\begin{property} \label{thm:resfiberGEN}
Let $y \in Y_+$ and let $\omega_\pi(y)$ be the volume form on the fiber $\pi^{-1}(y)$. Then:

\begin{itemize}
    \item[{\rm(i)}] Every $\pi_+$-generalized triangle $S$ of $Y_+$ corresponds to a pole $q_S$ of $\omega_\pi(y)$, where $q_S=\pi^{-1}(y) \cap \overline{S}$.
    
    \item[{\rm (ii)}] The residue\footnote{the residues have to be performed with suited signs, as explained in Remark~\ref{rk:orientation2}.} of 
    $\omega_{\pi}(y)$ at the pole $q_S$ computes the canonical function of the generalized triangle:  
    \begin{equation*}
\Omega(\Delta_S)(y) = {\rm Res}_{q_S} \omega_{\pi}(y).
    \end{equation*}
\item[{\rm (iii)}]    In particular, given a $\pi_+$-triangulation $\mathcal{C}=\lbrace S \rbrace$ of $Y_+$, its canonical form is given by:
\begin{equation*}
    {\mathbf \Omega}(Y_+)(y) = \mu_Y \cdot \sum_{S\in \mathcal{C}} {\rm Res}_{q_S} \omega_{\pi}(y).
\end{equation*}
\end{itemize}
\end{property}

We envision that the above properties
could provide useful tools to study Question~\ref{quest:general} leading to the notion of \emph{fiber positive geometry}. Moreover, the algebraic-analytical counterpart of this problem would be finding a suited residue procedure on the fiber volume form $\omega_{\pi}$ which computes the canonical form ${\mathbf\Omega}(Y_+)$ and reflects the combinatorics of $\pi_+$-triangulations (dissections) of $Y_+$.

\subsection{Fiber polytopes and residues \emph{à la} Jeffrey-Kirwan}\label{subsec:cansubdiv}
We now study Question~\ref{quest:general}.
In particular, we show how a generalization of the $\mathrm{JK}$-residue procedure can be used in the context of fiber polytopes. 
More precisely, let $P\subset \mathbb{P}^m$ be a convex polytope of dimension $m$ and let ${\bf \Omega}(P)$ be its canonical form. To compute ${\bf \Omega}(P)$, one can consider a (standard) triangulation $\{\Delta_i\}$ of $P$. Then
$
{\bf \Omega}(P) = \sum_i {\bf \Omega}(\Delta_i).
$
As shown in \S\ref{sec:JK}, the above procedure can be also viewed as a computation of the $\mathrm{JK}$-residue of the fiber volume form of the projection of the standard simplex to $P$.

More generally, assume that $P$ comes as a sufficiently general linear projection $\pi_+:Q\to P$ of a polytope $Q$ with $\mbox{dim}(Q)=n \geq m$. For simplicity, we assume that $Q$ is simple in dimension $m=\dim (P)$, i.e.~each $m$-dim face of $Q$ belongs to exactly $n-m$ facets. 
Now, let $\{S_i\}$ be a collection of $m$-dim faces of $Q$ which induces a $\pi_+$-regular subdivision of $P$. Therefore, we have that:
\begin{equation}\label{eq:omega1}
{\bf \Omega}(P) = \sum_i {\bf \Omega}(\pi_+(S_i)).
\end{equation}
Moreover, the framework established in \S\ref{sec:polytopes} guarantees that \eqref{eq:omega1} can be viewed as a computation of residue \emph{à la} Jeffrey-Kirwan of the fiber volume form. We describe this in more detail as follows. 
As before, we define the fiber volume form $\omega_\pi(y)$ to be a form on $Q\subset \mathbb{P}^n$ such that:
\begin{equation*}
    \mathbf{\Omega}(Q)=\omega \wedge \pi^*(\mu_{\mathbb{P}^m}).
\end{equation*}
For $y\in P$, we let $\omega_\pi(y)=\omega|_{\pi^{-1}(y)}$ and $\mathcal{F}(Q,P)$ the \emph{fiber fan} of the projection $\pi_+:Q\to P$, i.e.~the normal fan of the fiber polytope $\Sigma(Q,P)$. Then by \cite{billera1992fiber} each maximal cone $\sigma$ of $\mathcal{F}(Q,P)$ corresponds to the finest subdivision of $P$.  
First, note that the rays of $\mathcal{F}(Q,P)$ are numerated by facets of $Q$, or by rays of the dual fan $\mathcal{F}_Q$ of $Q$. For any face $S$ of $Q$, we define its \emph{shadow} $\mathfrak{S}_S$ to be the convex hall of the rays of $\mathcal{F}(Q,P)$ which correspond to the rays of the normal cone $\sigma_S$ of $S$.

\medskip

Following the notation above, we generalize Theorem~\ref{thm:explsymp} to general $\pi$-regular subdivisions of $P$.
\begin{Proposition}\label{prop:Qsub}
Let $\sigma$ be a maximal cone of $\mathcal{F}(Q,P)$ and $R_\sigma(P)$ the corresponding subdivision of $P$. The collection of $m$-dim faces of $Q$ which appear in $R_\sigma(P)$ is given by 
\[
\mathcal{C}_\sigma = \left\{ S_i \,|\, \dim (S_i) = m \text{ and } \sigma\subset \mathfrak{S}_{S_i}\right\}.
\]
\end{Proposition}
\begin{proof}
The proof is similar to the proof of Theorem~\ref{thm:explsymp}. For a point $y\in P$, let $Q_y = Q\cap \pi^{-1}(y)$ and $\sigma_y$ be the maximal cone of the normal fan of $Q_y$ which contains $\sigma$. 
We define a section $s_\sigma$ of $\pi$ as:
\[
s_\sigma: P\to Q, \quad y\mapsto A_{\sigma_y},
\]
where $A_{\sigma_y}$ is the vertex of $Q_y$ dual to $\sigma_y$. 
Since $\pi_+$ is generic, each vertex $A_{\sigma_y}$ of $Q_y$ comes as the intersection of $Q_y$ with some $m$-dim face $S_y$ of $Q$. Hence, $\sigma_y= \mathfrak{S}_{S_y}$. On the other hand, the image $s_\sigma(P)$ is a union of faces of $Q$ which provides a subdivision $\mathcal{C}_\sigma$ of $P$, which completes the proof.
\end{proof}

Proposition~\ref{prop:Qsub} shows that there is a residue procedure \emph{à la} Jeffrey-Kirwan on $\omega_\pi(y)$ which computes the canonical function of $P$. Indeed, we obtain the following generalization of Proposition~\ref{prop:jkpolytopes}. 
\begin{Corollary}\label{cor:alaJK}
Using notation above, $\Omega(P)(y)=  \sum_{S_i \in \mathcal{C}_\sigma} {\rm Res}_{q_{S_i}} \,\omega_\pi(y),$ where $q_{S_i}= \pi^{-1}(y)\cap \overline{S_i}$.
\end{Corollary}
\begin{proof}
As in Theorem~\ref{thm:resfiber}(ii), since $\mathcal{C}_\sigma$ provides a subdivision of $P$  by Proposition~\ref{prop:Qsub}, we have:
\begin{eqnarray*}
\Omega(P)(y) \cdot \mu_{\mathbb{P}^m}={\bf \Omega}(P) = \sum_{S_i\in \mathcal{C}_\sigma} {\bf \Omega}(S_i) &=& \sum_{S_i\in \mathcal{C}_\sigma} \mbox{Res}_{S_i}\mathbf{\Omega}(Q)(y) \\ 
&=&\sum_{S_i\in \mathcal{C}_\sigma} \mbox{Res}_{q_{S_i}} \omega_\pi(y) \cdot \mu_{\mathbb{P}^m}.
\end{eqnarray*}
Hence, $\Omega(P)(y)=  \sum_{S_i \in \mathcal{C}_\sigma} {\rm Res}_{q_{S_i}} \,\omega_\pi(y),$ as desired.
\end{proof}
The sum of residues from Corollary~\ref{cor:alaJK} is a generalization of the classical JK-residue defined in \eqref{def:jkphysics}. More precisely, for a maximal cone $\sigma$ of $\mathcal{F}(Q,P)$ we sum up the residues of $\omega_\pi(y)$ which correspond to the shadows of normal cones of faces of $Q$ containing $\sigma$. 
In the classical procedure, i.e.~when $Q$ is a simplex $\Delta$, any collection of rays of $\mathcal{F}_\Delta$ corresponds to a face of $\Delta$, hence the above formula is a generalization of the classical JK-residue.

\subsection{Grassmann polytopes with linear fibers}
\label{subsec:Grassmann}
Instead of the whole amplituhedron $\mathcal{A}_{n,k,m}$, one can consider (the closure of) images $\pi_+(S_M)$ of positroid cells $S_M$ in $\Gr_+(k,n)$ under the map $\pi_+$ in Definition~\ref{def:amp}. Such images, interesting in their own, are part of the broader class of \emph{Grassmann polytopes} defined in \cite{lam2014totally}. Grassmann polytopes are positive geometries, and one could apply the framework above to study their fibers as proposed in Question~\ref{quest:general}.
Even though, Theorem~\ref{th:jkfibers} is formulated for the case of amplituhedra conjugate to polytopes, similar to \S\ref{subsec:cansubdiv}, one could extend it to the case of Grassmann polytopes with \emph{linear fibers} inside amplituhedra. These are Grassmann polytopes whose positive fibers are bounded by a configuration of hyperplanes; See Example~\ref{ex:linearfiber}. In this case, one has to consider \emph{à la} JK-residues, as described in \S\ref{subsec:cansubdiv}.

Let $S_M$ be a positroid cell of $\Gr_+(k,n)$ such that the fibers of the restriction map $\pi|_{S_M}$ are \emph{linear} and the fiber fans of all points $Y \in \pi_+(S_M)$ are combinatorially equivalent.
Then, the canonical function of $\pi_+(S_M)$ at a vector $\xi$ in the secondary fan of $\pi_+(S_M)$,
would be obtained using a residue \emph{à la} Jeffrey-Kirwan as follows:
\begin{equation*}
   \Omega(\pi_+(S_M))= 
   \mathrm{JK}_{\xi} \, \omega(\pi_+(S_M)).
\end{equation*} 
Here $\omega(\pi_+(S_M))$ is the volume form on the fiber $\pi|_{S_M}^{-1}(Y)$. We give an example below.

\begin{example}[Grassmann polytope inside $\mathcal{A}_{6,2,2}$ with linear fibers] \label{ex:linearfiber}
Consider the image $\pi_+(S)$ of the $6$-dim cell $S$ inside $\mathcal{A}_{6,2,2}$ that has the following vanishing Pl\"ucker coordinates: $p_{15}, p_{16}, p_{56}$. The volume form $\omega(\pi_+(S))$ on the fiber $\pi|_{S_M}^{-1}(Y)$ can be obtained by taking the residue of \eqref{ex:non-linear} around the $2$-dim hypersurface determined by $p_{15}(\lambda \cdot A)=p_{16}(\lambda \cdot A)= p_{56}(\lambda \cdot A)=0$ as:
\begin{equation*}
    \omega(\pi_+(S))=\frac{\mathcal{N}(Y,Z) \, \mbox{d}^2 \lambda}{p_{12}(\lambda \cdot A) \, p_{23}(\lambda \cdot A) \, p_{34}(\lambda \cdot A) \, p_{45}(\lambda \cdot A) \, p_{46}(\lambda \cdot A)},
\end{equation*}
where $\lambda$ are parametrized on the $2$-dim surface, and $\mathcal{N}$ is a numerator factor not depending on $\lambda$. In particular, all factors in the denominator are linear and they correspond to a configuration of hyperplanes $\lbrace \mathcal{H}_i \rbrace_{i \in [5]}$ in $\mathbb{P}^2$. The combinatorics of the corresponding fan does not vary by changing the points $Y \in \pi_+(S)$, moreover, it coincides with the secondary fan. More precisely, there are $5$ chambers (and triangulations) and $\binom{5}{2}=10$ cones (and generalized triangles). Moreover, $\mathrm{JK}_{\xi} \omega(\pi_+(S))=\Omega(\pi_+(S)),$ where $\xi$ is a fixed reference point in $\mathbb{P}^2$. For example, if we choose $\xi$ in the chamber contained in the cones $24,45,12$, then $\mathrm{JK}_{\xi} \omega(\pi_+(S))=\left( \mathrm{Res}_{C_{24}}+\mathrm{Res}_{C_{45}}+\mathrm{Res}_{C_{12}} \right) \omega(\pi_+(S)).$
\end{example}

\section{Conclusions and outlook}\label{sec:conclusions}
In this work, motivated by the works on fiber polytopes in combinatorics, and on positive geometries in mathematical physics, we introduced a new way to study canonical forms on amplituhedra via the fibers of the rational map $\pi: \Gr(k,n)\dashrightarrow \Gr(k,k+m)$ over points $Y$ in the amplituhedron $\mathcal{A}_{n,k,m}$.
We showed that the fibers of $\pi$ enjoy some properties similar to positive geometries. In particular, for each point $Y\in \mathcal{A}_{n,k,m}$, we associated a rational top-degree form $\omega_{n,k,m}(Y)$ to the fiber $\pi^{-1}(Y)$, that we call \emph{fiber volume form}, with simple poles along a certain collection of divisors.  
We showed that the value of the canonical function of $\mathcal{A}_{n,k,m}$ at $Y$ can be computed as a summation of certain residues of $\omega_{n,k,m}(Y)$. Moreover, we gave a concrete birational parametrization of the fibers of $\pi$ and explicitly wrote the fiber volume form $\omega_{n,k,m}(Y)$ in this parametrization.
As an application of our approach, we studied amplituhedra conjugate to polytopes, i.e.~$\mathcal{A}_{n,\ell,m}$ with $\ell=n-m-1$, and $m$ even. Importantly, the linear structure of the fibers of $\pi$ allows us to define a family of fan $\mathcal{F}(Y)$, for $Y \in \mathcal{A}_{n,\ell,m}$. 
We defined a collection of natural functions on $\mathcal{A}_{n,\ell,m}$ and conjectured their positivity; See Conjecture~\ref{con:main} for more details. Using this conjecture we showed that $\mathcal{F}(Y)$ is combinatorially equivalent to the secondary fan of cyclic polytope. 
We then showed that  
the fan $\mathcal{F}(Y)$ is the secondary fan of $\mathcal{A}_{n,\ell,m}$, i.e.~it gathers information about its (regular) subdivisions. This result generalizes the construction of the secondary polytopes of Gelfand–Kapranov–Zelevinsky to amplituhedra which are not necessarily polytopes.

\smallskip

To support Conjecture~\ref{con:main}, we have speculated that cluster algebras of the Grassmannian type could play a role in proving positivity of functions on amplituhedra which are not manifestly positive. From Scott's seminal work \cite{scott}, we know that all cluster variables are positive if all the Pl\"ucker coordinates on the Grassmannian are positive, i.e.~if we  restrict to the positive Grassmannian. Even though cluster variables have complicated expressions, they come from subtraction-free polynomials in Pl\"ucker coordinates. It would be fascinating to further explore the connection between the combinatorics of secondary fans of amplituhedra and Grassmannian cluster algebras.
We also proved that the \emph{Jeffrey-Kirwan residue} provides the correct residue prescription on the fiber volume form to compute the canonical function of amplituhedra with linear fibers, i.e.~cyclic polytopes $\mathcal{A}_{n,1,m}$ and conjugate to polytopes $\mathcal{A}_{n,\ell,m}$.
In particular, we showed that the collections of residues picked up by the Jeffrey-Kirwan procedure correspond to the sets of positroid cells in $\Gr_+(k,n)$ which provide (positroidal) triangulations of $\mathcal{A}_{n,k,m}$.

\smallskip

Looking ahead, we envision that our framework can be generalized to arbitrary positive geometries and could be applied in studying their triangulations and canonical forms. We hence 
suggest a \emph{fiber positive geometry} framework to find (regular) subdivisions of the positive geometries and compute their canonical forms.
As examples beyond amplituhedra, we have briefly presented the procedure to apply the framework of fiber positive geometries to \emph{fiber polytopes} and to \emph{Grassmann polytopes}, in connections to residues \emph{à la} Jeffrey-Kirwan. A comprehensive exploration of these cases is left to future works.

\smallskip

Questions~\ref{ques:res} and \ref{quest:general} are left open for general amplituhedra and fiber positive geometries. 
The hope to find answers to these questions for amplituhedra was ignited by the work of the third author in \cite{lukowski2020positive}, where it was conjectured that the $m=2$ amplituhedron $\mathcal{A}_{n,k,2}$ admits the positive tropical Grassmannian ${\rm Trop}^+\Gr(k+1,n)$ as secondary fan with respect to (regular) positroidal subdivisions.
In particular, even though $m=2$ amplituhedra are in general \emph{not} polytopes (and not even their fibers w.r.t. to the map $\pi$), their fiber positive geometries are polytopes! The construction of these polytopes, dual to the fan of ${\rm Trop}^+\Gr(k+1,n)$, was explained in the original work \cite{troppos}. How such fiber positive geometries emerge from our fiber-based approach is an interesting direction worth pursuing, and might shed light on possible generalizations beyond the $m=2$ case. Additionally, to handle non-polytopal fibers, it would be interesting to analyze possible connections with theories of \emph{convex bodies}, as in \cite{convexbody}.

\bibliographystyle{alpha}
\bibliography{JKresidue}

\bigskip
\noindent
\small {\bf Authors' addresses:}

\bigskip

\noindent Department of Mathematics: Algebra and Geometry, Ghent University, 9000 Gent, Belgium \\
Department of Mathematics and Statistics, 
UiT – The Arctic University of Norway, 9037 Troms\o, Norway
\\ E-mail address: {\tt fatemeh.mohammadi@ugent.be}

\medskip

\noindent Max Planck Institute for Mathematics in the Sciences, Leipzig, 04103, Germany\\
 E-mail address: {\tt leonid.monin@mis.mpg.de}

\medskip

\noindent Mathematical Institute, University of Oxford, Oxford, OX2 6GG, UK
\\ Department of Physics, Princeton University, NJ 08544, USA \\ E-mail address: {\tt matteo.parisi@maths.ox.ac.uk}
\medskip

\end{document}